\documentclass[10pt]{amsart}
 
\usepackage{pdfpages}
\usepackage{amsmath,amssymb,amsthm}
\usepackage{hyperref}
\usepackage{amsmath,amsfonts,amssymb,amsthm}
\usepackage{mathtools}
\usepackage{amsmath}
\usepackage{amsthm}
\usepackage{amssymb}
\usepackage{mathtools}
\usepackage[all]{xy}
\usepackage[normalem]{ulem}
\usepackage{enumitem}
 \newtheorem{Conjecture}{Conjecture}
  \newtheorem{Question}{Question}
\newtheorem{lem}{Lemma}[section]
\newtheorem{teo}[lem]{Theorem}
\newtheorem{pro}[lem]{Proposition}
\newtheorem{cor}[lem]{Corollary}

\newtheorem*{Claim*}{Claim}
\newtheorem*{rem*}{Remark}
\newtheorem*{teo*}{Theorem}

\newcounter{claimcounter}
\numberwithin{claimcounter}{lem}

\newcommand{\myeq}[1]{\ensuremath{\stackrel{\text{#1}}{=}}}

\newcommand{\s}{\mathcal S}
\newcommand{\D}{\mathcal D}
\newcommand{\U}{\mathcal U}
\newcommand{\PP}{\mathcal P}
\newcommand{\E}{\mathcal E}

 \DeclareMathOperator{\proj}{proj}\DeclareMathOperator{\RP}{RP}\DeclareMathOperator{\LP}{LP}

\DeclareMathOperator{\Id}{Id}
\DeclareMathOperator{\Tr}{Tr}
\DeclareMathOperator{\lcm}{lcm}

\DeclareMathOperator{\fm}{fam}
\DeclareMathOperator{\hei}{height}
\DeclareMathOperator{\wid}{width}
\DeclareMathOperator{\Tree}{Tree}
\DeclareMathOperator{\Rat}{Rat }

\DeclareMathOperator{\MG}{MG}

\DeclareMathOperator{\im}{Im}

\DeclareMathOperator{\rk}{rk}
\DeclareMathOperator{\wrk}{\widetilde{\rk}}
\DeclareMathOperator{\wdim}{\widetilde{\dim}}

\DeclareMathOperator{\End}{End}

\DeclareMathOperator{\source }{source}

\DeclareMathOperator{\Mat}{Mat}

\newcommand{\Z}{\mathbb{Z}}

\newcommand{\N}{\mathbb{N}}
\newcommand{\CC}{\mathbb{C}}
\newcommand{\Q}{\mathbb{Q}}

 \date{November 2019}

\title[The strong Atiyah  and  L\"uck approximation conjectures]{The strong Atiyah  and  L\"uck approximation conjectures for one-relator groups}
\author{Andrei Jaikin-Zapirain}

\address{Departamento de Matem\'aticas, Universidad Aut\'onoma de Madrid \and  Instituto de Ciencias Matem\'aticas, CSIC-UAM-UC3M-UCM}
\email{andrei.jaikin@uam.es}
\email{diego.lopez@icmat.es}

\author{Diego L\'opez-\'Alvarez}

\begin{document}

\begin{abstract}  It is shown that the strong Atiyah conjecture and the L\"uck approximation conjecture in the space of marked groups hold for locally indicable groups. In particular, this implies that one-relator groups satisfy both conjectures.  We also show that the center conjecture,  the independence conjecture and the strong eigenvalue conjecture hold for these groups. 

As a byproduct we prove that the group algebra of a locally indicable group over a field of characteristic zero has a Hughes-free  {epic} division algebra and, in particular, it is embedded in a division algebra.
 \end{abstract}
 \maketitle
 \setcounter{tocdepth}{1}
\tableofcontents
 
 \section{Introduction}\label{intro}
 \subsection{The strong Atiyah conjecture} Let $G$ be a group and assume that   the orders of finite subgroups of $G$ are bounded above. We denote by $\lcm(G)$ the least common multiple of the orders of finite subgroups of $G$. Let $X$ be a cocompact free proper $G$-$CW$ complex. The strong Atiyah conjecture for $G$ over $\Q$ predicts that the $L^2$-Betti numbers $\beta_i^{(2)}(X, G)$ belong  to $\frac 1{\lcm(G)} \Z$. In this paper we consider an algebraic  reformulation  of this conjecture  which also leads  to a natural generalization of it over an arbitrary subfield $K$ of the field of complex numbers $\CC$.

Let $G$ be a countable group. Then $G$  acts by   left and right multiplication     on $l^2( G ) $. A  finitely generated {\bf Hilbert}  $G$-module is a closed subspace $V\le (l^2(G))^n$,  invariant under the left action of $G$. We denote by  $\proj_{V}: (l^2(G))^n\to( l^2(G))^n$ the orthogonal projection onto $V$ and we define $$\dim_{G}V:=\Tr_{G}(\proj_{V}):=\sum_{i=1}^n\langle (\mathbf 1_i)\proj_{V},\mathbf 1_i\rangle_{(l^2(G))^n},$$
 where $\mathbf 1_i $ is the element of $ (l^2(G))^n$ having 1 in the $i$th entry and $0$ in the rest of the entries.     The number $\dim_{G}V $ is the {\bf von Neumann dimension} of $V$.  
 
 Let   $A \in \Mat_{n\times m}(\CC[G])$ be a matrix over $\CC[G]$. 
   The action of $A$ by  right multiplication on $l^2(G)^n$  induces   a bounded linear operator  $\phi^A_{G}:(l^2(G))^n\to( l^2(G))^m$.  We put  
   \begin{equation} \label{rkG} \rk_G(A)=\dim _G \overline{\im   \phi^A_{G}}=n-\dim_G \ker \phi^A_{G}. \end{equation}

     If $G$ is not  countable  then $\rk_G$ is defined as follows. Take a matrix $A$ over $\CC[G]$. Then the group elements that appear in $A$ are contained in  a finitely generated group $H$. We will put $\rk_G(A)=\rk_H(A)$. One easily checks that the value $\rk_H(A)$ does not depend on the subgroup $H$.

\begin{Conjecture} [{\bf The  strong Atiyah conjecture over  $K$ for a group $G$}] \label{conjat} Let  $K$ be a subfield of $\CC$. Assume that there exists an upper bound for the orders of finite subgroups of $G$.   Then for every $A\in \Mat_{n\times m}(K[G])$, $\rk_G(A)\in \frac 1{\lcm(G)} \Z.$
\end{Conjecture}

 There are many different reasons to be interested in this conjecture. From a topological point of view it is important because it imposes a strong restriction on possible values of  $\beta_i^{(2)}(X, G)$. 
 
 Ring theorists study  the strong Atiyah conjecture because it implies  that the ring $\mathcal R_{K[G]}$ (see Subsection \ref{*reg} for the definition) has a very particular structure and, in particular, when $G$ is torsion-free, the conjecture predicts that $\mathcal R_{K[G]}$ is a  division ring. This is a strong version of the Kaplansky zero-divisor conjecture for $K[G]$.  
 
 The strong Atiyah conjecture has also importance in group theory. For example, a question of R. Bieri asks whether a  group $G$ of homological dimension one is locally free. P. Kropholler, P. Linnell and W. L\"uck \cite{KLL} showed that the answer is positive provided that $G$ satisfies the  strong Atiyah conjecture over $\Q$.
 
 During the last 25 years it has been shown that many families of groups satisfy the strong Atiyah conjecture. We refer the reader to a recent survey \cite{Ja17surv} of the first author, where all  these results are described. In this paper we show that the strong Atiyah conjecture over $\CC$ holds for locally indicable groups. Recall that a group   $G$ is {\bf indicable} if either $G$ is trivial or $G$ maps onto $\mathbb{Z}$.
We say that $G$ is {\bf locally indicable} if every finitely generated subgroup   of $G$ is indicable. 
 \begin{teo}\label{main}
 Let $G$ be a locally indicable group.
 Then $G$ satisfies the strong Atiyah conjecture over $\CC$.
 \end{teo}
 Recall that the Kaplansky zero-divisor conjecture for these groups was solved by G. Higman \cite{Hi40} before I. Kaplansky formulated it. Observe   that Theorem \ref{main} implies the corresponding result for  all subfields of $\CC$.

 \subsection{Consequences of Theorem \ref{main}.}
   Let us present several applications of Theorem \ref{main}.  
 
 \begin{cor} Let $G$ be a countable locally indicable group. 
 \begin{enumerate}
 \item {\bf The strong algebraic eigenvalue conjecture}.
 
 Let $K$ be a subfield of $\CC$ closed under complex conjugation
and  $A\in\Mat_n(\mathcal R_{K[G]} )$. Then  the eigenvalues of $\phi_G^A$ are algebraic over $K$.
 \item  {\bf The center conjecture}.
  
 Let $K$ be a subfield of $\CC$ closed under complex conjugation. Then  we have $ {\mathcal R_{K[G]}}\cap \CC=K.$ 

 \item {\bf The independence conjecture}.
 
 Let $K$ be a field and let ${\varphi_1,\varphi_2}:K\to \CC$ be  two embeddings of $K$ into $\CC$. Then for every  matrix $A\in \Mat_{n\times m}(K[G]) $, 
   $\rk_G({\varphi_1}(A))=\rk_G({\varphi_2}(A)).$

 \end{enumerate}

 \end{cor}
 
  All these conjectures were proved for sofic groups in \cite{Ja17base}.
 
One-relator groups with torsion are virtually special by a theorem of D. Wise \cite{Wi}. The strong Atiyah conjecture for  virtually special  groups over $\CC$ is proved in \cite{Ja17base}. Also virtually special groups are sofic. One-relator groups without torsion are locally indicable by a result of S. Brodskii \cite{Br84}. Thus, we obtain the following corollary.
 \begin{cor}
 The strong Atiyah conjecture, the strong algebraic eigenvalue conjecture, the center conjecture and the independence conjecture hold for one-relator groups.
 \end{cor}
 In Subsection \ref{hughes} we introduce the notion of Hughes-free {epic} division $K[G]$-algebra. In \cite{Hu70} I. Hughes showed  that up to $K[G]$-isomorphism there exists at most one  Hughes-free {epic} division $K[G]$-algebra. Our main result implies   the following consequence.
 \begin{cor}\label{exh} Let $G$ be a locally indicable group and let $K$ be a field of characteristic zero. Then there exists  {a} Hughes-free  {epic} division $K[G]$-algebra.
 
 \end{cor}
 Thus, a group algebra of a locally indicable group over a field of characteristic zero is embedded in a division algebra. This solves the Malcev problem for this class of group algebras (for more details about this problem see \cite {GS15}).

 \subsection{The L\"uck approximation}
  Let $F$ be a free group freely generated by a finite set $S$.
 The {\bf space of marked groups} $\MG(F)$ can be identified with the set of normal subgroups of $F$ with the metric  $d(M_1,M_2)=e^{-n}$   where $n$ is the
largest integer such that the balls of radius $n$ in the Cayley graphs of $F/M_1$ and $F/M_2$ with respect to the generators $S$ are simplicially isomorphic (with respect to an isomorphism respecting the labelings).
For example, if $M_1\ge M_2\ge \cdots $ is a chain of normal subgroups of $F$ and $M=\displaystyle \cap_{i\in \N} M_i$, then $(M_i)_i$ converges to $M$ in $\MG(F)$.

Observe that the metric on $\MG(F)$ depends on the choice of $S$, but the topology does not.  Thus, we are thinking of $\MG(F)$ as a
topological space. 

Let $M$ be a normal subgroup of $F$  and $A$ a matrix over $\CC[F]$. By abuse of notation we write $\rk_{F/M}(A)$ instead of $\rk_{F/M}(\bar A)$, where $\bar A$ is the matrix over $\CC[F/M]$ obtained from $A$ using the canonical map $\CC[F]\to \CC[F/M]$.

 Now we can formulate {\bf the   L\"uck approximation conjecture in the space of marked groups over  $K$ for a finitely generated group $G$}.  
 \begin{Conjecture} \label{luckmarked} Let $G$ be a finitely generated group and let $K$ be a subfield of $\CC$. Let $F$ be a finitely generated free group and assume that  $(M_k\in \MG(F))_k$ converges  to $M\in \MG(F)$ with   $G\cong F/M$. Then for every     $A\in  \Mat_{n\times m}(  K[F])$,
    $$\displaystyle \lim_{k\to\infty} \rk_{F/M_k}(A)=\rk_{F/M}(A).$$   \end{Conjecture}
    We say that a   group  satisfies the   L\"uck approximation conjecture in the space of marked groups over  $K$ if its finitely generated subgroups $G$ satisfy the conclusion of Conjecture \ref{luckmarked}.
    
 Conjecture \ref{luckmarked} has a long story which starts from the paper of W. L\"uck \cite{Lu94}, where it is proved when $K=\Q$ and the groups $F/M_k$ are finite. Different extensions of L\"uck's result were obtained in \cite{DM98, Sch01, DLMSY, ES05, El06}. In \cite{Ja17base} the first author proved the conjecture in the case where the groups $F/M_k$ are sofic and $K$ is an arbitrary subfield of $\CC$.   The case where $G$ is free and $K$ is an arbitrary subfield of $\CC$ was proved in \cite{Ja17HN}. 

In our next result we prove  the   L\"uck approximation conjecture in the space of marked groups over  an arbitrary subfield of $\CC$ for  virtually locally indicable groups.

\begin{teo}\label{mainluck}
Let $G$ be a virtually locally indicable group. Then $G$ satisfies  the   L\"uck approximation conjecture in the space of marked groups over  an arbitrary subfield of $\CC$.
\end{teo}
 Since a one-relator group is virtually torsion-free \cite{FKS72} and a torsion-free subgroup of a one-relator group is locally indicable \cite{Ho02}, we obtain the following immediate corollary.
 \begin{cor}
 One-relator groups satisfy  the   L\"uck approximation conjecture in the space of marked groups over  an arbitrary subfield of $\CC$.
  \end{cor}

 \subsection{A description of the proof}

 There are two points that make our results about the strong Atiyah conjecture and the L\"uck approximation conjecture different from previous ones. 
 
  The first aspect concerns the methods that we use in the proof of the strong Atiyah conjecture in Theorem \ref{main}. Algebraic methods were already widely used in previous results on the strong Atiyah conjecture. However, all these proofs also contained some  analytic parts (as, for example, the use of the theory of Fredholm operators in \cite{Li93}  or the use of L\"uck approximation in \cite{DLMSY}). Our proof of Theorem \ref{main} is almost purely algebraic, and, in particular, this gives the first   algebraic proof of the strong Atiyah conjecture for free groups.  
  
  The second aspect is about the groups that we consider.  All the previous instances of  both conjectures concerned groups which are known to be sofic. This is not the case of locally indicable groups. In fact, it is not known yet whether one-relator groups are sofic. 
  
  Let us describe briefly the ideas behind the proofs of Theorems \ref{main} and \ref{mainluck}. In \cite{DHS}, W. Dicks, D. Herbera and J. S\'{a}nchez enlightened the argument of I. Hughes and gave a different proof of his result on the uniqueness of the Hughes-free epic division $E*G$-algebra. In order to get some insight in the techniques that they use and we adopt here, let us give a short summary of the fundamental steps of \cite{DHS}. For this purpose, let $E*G$ be a crossed product of a division ring $E$ with a group $G$, and let $\D$ be a Hughes-free epic division $E*G$-algebra. For every subgroup $H$ of $G$, denote by $\D_{H,\D}$ the division closure of $E*H$ in $\D$. 
  
   First of all, for any multiplicative group $U$, the authors introduce a universal object $\Rat (U)$, whose construction is a formal analog of the construction of a division closure, and that can be endowed with a measure of complexity that allows to compare elements. As a consequence of its universality and construction, W. Dicks, D. Herbera and J. S\'{a}nchez  get, for every subgroup $H$ of $G$, a surjective morphism
$$
\Phi_{H,\D}: \Rat (E^{\times}H)\twoheadrightarrow \D_{H,\D}\cup \{\infty\} .
$$

If $H$ is non-trivial, finitely generated and splits as a semidirect product $H = N \rtimes C$ where $C$ is infinite cyclic, $t$ is an element in $E^{\times}H$ whose image under $E^{\times}H\rightarrow E^{\times}H/E^{\times}=H$ generates $C$, and $\tau$ denotes left conjugation by $t$,   then they prove the existence of the following commutative diagram 
\begin{equation} \label{dia1}
\begin{gathered}
\xymatrix{ \Rat (E^{\times}H) \ar@{->>}[r]^{\Phi_{H,D}} \ar[d]_{\Psi} & D_{H,\D}\cup \{\infty\} \ar@{^{(}->}[d] \\
           \Rat (E^{\times}N)((t;\tau))\cup \{\infty\} \ar@{->>}[r]_{\Omega} & D_{N,\D}((t;\tau)) \cup \{\infty\},
}
\end{gathered}
\end{equation}
where $\D_{N,\D}((t;\tau))$ denotes the formal skew Laurent series (the action of $t$ by conjugation on  $E^{\times}N$   extends canonically to an action on $\D_{N,\D}$).

Finally, they proved  that for every element $\alpha\in \Rat (E^{\times}G)$, there exists an appropriate finitely generated  subgroup $ \source (\alpha)$ of $E^{\times}G$, and hence a finitely generated subgroup $H$ of $G$ (given by the image of $ \source (\alpha)$ under $E^{\times} G \rightarrow E^{\times} G/E^{\times} = G$), such that up to multiplication by a unit we have that $\alpha\in \Rat (E^{\times}H)$ and, in the above diagram, $\Psi(\alpha)$ is a series whose summands are strictly less complex than $\alpha$. This allows them to make proofs by induction on the complexity of the elements. A great reference to learn how the details work is the PhD Thesis of Javier S\'{a}nchez  \cite{Sthesis}. 

In our setting we will consider division $E*G$-closures inside rings which are non-necessarily division rings. If $(S,\phi)$ is an $E*G$-ring and, for any subgroup $H$ of $G$, $\D_{H,S}$ denotes the division closure of $\phi(E*H)$ inside $S$, then we also have a surjective morphism
$$
\Phi_{H,S}: \Rat(E^{\times}H) \twoheadrightarrow \D_{H,S}
$$
The complexity of elements of $\Rat(E^{\times}G)$ induces a notion of $G$-complexity of elements of $\D_{G,S}$. It would also be desirable to have an analog for diagram (\ref{dia1}) which permits expressing any element in $\mathcal D_{G,S}$ as a sum of less complex elements, and so using induction on this complexity. However, at first sight we can say nothing about the relation between $\D_{H,S}$ and $\D_{N,S}((t;\tau))$. 
In Proposition \ref{key} we show that, if there exists a diagram
$$
  \xymatrix{ E*N \ar@{->}[r]^{\phi} \ar@{^{(}->}[]+<0ex,-2ex>;[d] & \mathcal A \ar@{^{(}->}[]+<0ex,-2ex>;[d] & \\
             E*H \ar@{->}[r]^{\phi} & \mathcal A((t;\tau))~ \ar@{^{(}->}[r] & \PP
     }
$$
where $(\mathcal A,\phi)$ is a von Neumann regular $E*N$-ring and $\tau$ is an automorphism of $\mathcal A$ such that $\phi \circ \tau = \tau \circ \phi$, then we can develop the same sort of inductive method for $\D_{H,\PP}$ and $\D_{N,\PP}((t;\tau))$. The proofs of the strong Atiyah and the L\"uck approximation conjectures rely then on the construction of such scenarios.

 To prove the strong Atiyah conjecture, we introduce a generalization of the notion of Hughes-freeness for epic $*$-regular $E*G$-rings, expressed in terms of $*$-regular Sylvester rank functions. 
 Then, we show in Theorem \ref{maintheorem} that if $K$ is a subfield of $\CC$ closed under complex conjugation, any epic positive definite $*$-regular $K[G]$-ring $\mathcal U$ with Hughes-free Sylvester rank $\rk$ is, in fact, a division algebra. Since $\rk_G$ (defined in (\ref{rkG})) is a canonical example of a Hughes-free rank on $\CC[G]$, this, in particular, implies that $\mathcal R_{\CC[G]}$ is a division algebra, and so $G$ satisfies the strong Atiyah conjecture over $\CC$. 

What we do to deduce Theorem \ref{maintheorem} is the following. Let $\mathcal U_{H}$ (respectively $\U_N$) denote the $*$-regular closure of $K[H]$ (respectively $K[N]$) in $\mathcal U$ and set $\mathcal A = \mathcal U_N$. Using some results on epic $*$-regular $R$-rings proved in \cite{Ja17base}, we construct an environment $\PP = \PP_{\omega,\tau}^{\mathcal U_{N}}$ containing $\mathcal A((t;\tau))$ as in the latter diagram. In addition, the Hughes-free condition allows us to embed $\mathcal U_{H}$ in $\PP$, and the regularity of $\mathcal U$ implies that $\D_{H,\PP} = \D_{H,\mathcal U}$ and $\D_{N,\PP} = \D_{N,\mathcal U}$. This means that we can talk about the intersection $\D_{H,\mathcal U}\cap \D_{N,\mathcal U}((t;\tau))$ in $\PP$. Using induction on the complexity, we eventually manage to show that $\mathcal D_{H,\mathcal U} $ is a subalgebra of $\mathcal D_{N,\mathcal U}((t;\tau))$, and that every $0\ne  a\in \mathcal D_{H,\mathcal U}\setminus K^{\times}H$ can be expressed as a series with coefficients in $\mathcal D_{N,\mathcal U}$ of strictly lower complexity. Therefore, again by induction on the complexity, we obtain that $a$ is invertible, which shows that $\D_{G,\mathcal U}$, and hence $\mathcal U$, is a division algebra.  

 The paper is structured as follows. In Section \ref{sylv} we include the preliminary results on $*$-regular rings, Sylvester rank functions and the theory of $*$-regular $R$-rings. In Section \ref{nat} we introduce the notion of natural extension and Hughes-free Sylvester rank function on a group algebra of a locally indicable group. Here we also construct the aforementioned ring $\mathcal P_{\omega,\tau}^{\mathcal U_{N}}$. Section \ref{rat} is devoted to recalling the notion of rational $U$-semiring  and the examples we will use later. In Section \ref{mainsection} we will present the proof of the key proposition regarding the inductive step. Sections \ref{proofs} and \ref{Lucksection} contain the proofs of the Atiyah and the L\"uck approximation conjectures in the setting of locally indicable groups, and of its corollaries.  Finally, in Section \ref{universal} we discuss the problem of universality of Hughes-free Sylvester rank functions. 
 
 \section*{Acknowledgments} This paper is partially supported by the grant   MTM2017-82690-P of the Spanish MINECO and by the ICMAT Severo
Ochoa project SEV-2015-0554.

We would  like to thank Fabian Henneke and Dawid Kielak for suggesting that our argument  can be used  to prove that the class of torsion-free groups satisfying the  Atiyah conjecture is closed  under extensions by locally indicable  groups. We are also  very grateful  to  Javier S\'anchez for useful  discussions, to Fabian Henneke for many comments on a previous version of this paper and to an anonymous referee for a thorough reading of the paper and useful suggestions.

 \section{$*$-regular Sylvester rank functions}
\label{sylv}
In this section we recall the notions of $*$-regular ring and  Sylvester rank function, and explain the main results about epic $*$-regular $R$-rings. More information about these topics can be found in \cite{Ja17base, Ja17surv}.
\subsection{$*$-regular rings}\label{*reg}
An element $x$ of a ring $R$ is called  {\bf von Neumann regular} if  there exists $y\in  R$ satisfying  $xyx=x$. 
 A   ring $\mathcal  U$ is called  {\bf von Neumann regular} if all the elements of $\mathcal U$ are von Neumann regular.  For the sake of brevity, we will often refer to von Neumann regular rings simply as {\bf regular} rings. For instance, a division ring is regular and, for any $n$, the ring of $n\times n$ matrices over a regular ring is also regular (\cite[Lemma 1.6, Theorem 1.7]{Gdr}).  
 
 By a {\bf $*$-regular} ring $\mathcal U$ we mean a von Neumann regular ring together with a proper involution (i.e. an involution $*:\mathcal U\rightarrow \mathcal U$ for which $x^*x=0$ implies $x=0$). In this setting, for every element $x\in \mathcal U$, we can distinguish an element $x^{[-1]}$ with $xx^{[-1]}x = x$ among the others, called {\bf the relative inverse} of $x$ (see, for example, \cite[Proposition 3.2]{Ja17base}). The element $x^{[-1]}$ is characterized by the property that $\RP(x)=x^{[-1]}x$ and $\LP(x)=xx^{[-1]}$ are  {projections} (self-adjoint idempotents) and  $x^{[-1]}xx^{[-1]}=x^{[-1]}$.
 A useful remark about $*$-regular rings is the following  proposition. 
 \begin{pro} \cite[Proposition 3.3]{Ja17surv} \label{proper}
Let $\mathcal U$ be  a $*$-regular ring and $I$ a (two-sided) ideal of $\mathcal U$. Then $I$ is $*$-closed and $*$ is proper in $\mathcal U/I$, i.e., $\mathcal U/I$ is also a $*$-regular ring.
\end{pro}

 We say that a $*$-regular ring $\mathcal U$ is {\bf positive definite} if $\Mat_n(\mathcal U)$ is $*$-regular for every $n\geq 1$.
 
 If $R$ is a $*$-subring of a $*$-regular ring $\mathcal U$, then we can construct the smallest $*$-regular subring of $\mathcal U$ containing $R$, as follows.   
\begin{pro} \cite[Proposition 6.2]{AG17} \label{closure}
Let $R$ be a $*$-subring of a $*$-regular ring $\mathcal U$. Then there exists a smallest $*$-regular subring $\mathcal{R}(R,\mathcal U)$ of $\mathcal U$ containing $R$. Moreover, it can be constructed as follows.
\begin{itemize}
 \item[-] Put $\mathcal{R}_0(R,\mathcal U):=R$, a $*$-subring of $\mathcal U$.
 \item[-] Suppose $n\geq 1$ and that we have constructed a $*$-subring $\mathcal{R}_n(R,\mathcal U)$ of $\mathcal U$. Then $\mathcal{R}_{n+1}(R,\mathcal U)$ is the $*$-subring of $\mathcal{U}$ generated by the elements of $\mathcal{R}_n(R,\mathcal U)$ and the relative inverses of its elements.
 \item[-] $\mathcal{R}(R,\mathcal U) = \bigcup_{n=0}^{\infty} \mathcal{R}_n(R,\mathcal U)$.
\end{itemize}
\end{pro} We call $\mathcal{R}(R,\mathcal U)$ the {\bf $*$-regular closure} of $R$ in $\mathcal U$.

\subsection{The algebra of   affiliated operators}\label{UG}
For a countable group $G$ we denote by $\mathcal N(G)$ and $\mathcal U(G)$ its group von Neumann algebra and the algebra  affiliated to $\mathcal N(G)$. We direct the reader to the book  by  W. L\"uck   \cite{Lubook} for a detailed account of the subject. In this subsection we only recall the facts that we use in the paper.

The algebra $\U(G)$ is an example of a positive definite $*$-regular ring, where the $*$ is the adjoint operation.
The rank function $\rk_G$ can be extended to matrices over $\U(G)$. First observe that, given $A\in \Mat_{n\times m}(\mathcal N(G))$, we can see $A$ as an operator $l^2(G)^n\to l^2(G)^m$ and, by analogy with (\ref{rkG}), define $$\rk_G(A)=n-\dim_G \ker A.$$
Since  $\U(G)$ is isomorphic to the completion of $\mathcal N(G)$ with respect to the metric induced by $\rk_G$,  we can extend continuously $\rk_G$ on the matrices over $\mathcal U(G)$ as well.

If $H$ is a subgroup of $G$,  then $l^2(G)$ can be thought of as the Hilbert completion of $\oplus_{t\in T} \,t\,l^2(H)$, where $T$ is a right transversal of $H$ in $G$. Hence, we can identify any element $\varphi$ of the group von Neumann algebra $\mathcal{N}(H)$ with the element of $\mathcal{N}(G)$ that assigns to any tuple in $\oplus_{t\in T} \,t\,l^2(H)$ the tuple obtained by applying $\varphi$ component-wise. This gives an embedding of $\mathcal N(H)$ into $\mathcal N(G)$, that can be extended uniquely to an embedding of $\mathcal U(H)$ into $\mathcal U(G)$. In the following, we will often consider $\U(H)$ as a subalgebra of $\U(G)$ without further explanations.

Let  $K$ be a subfield of $\CC$ closed under complex conjugation. The group algebra $K[G]$ is a $*$-subring of $\mathcal U(G)$ with the usual involution given by $(\lambda g)^* = \bar{\lambda}g^{-1}$, and the $*$-regular closure of $K[G]$ in $\mathcal U(G)$ is denoted by $\mathcal R_{K[G]}$. For an arbitrary group $G$, $\mathcal R_{K[G]}$ is defined as the direct union of $\{\mathcal R_{K[H]}$:  $H$  is a finitely generated subgroup of $G\}$.

\color{black}

\subsection{Epic homomorphisms}
We say that a homomorphism of rings $\varphi: R\rightarrow S$ is {\bf epic} if it is right cancellable, i.e., for every ring $Q$ and homomorphisms $\psi,\phi: S\rightarrow Q$, we have that equality of compositions $\psi \circ \varphi = \phi \circ \varphi$ implies $\psi = \phi$. There exists a characterization of epic morphisms in terms of the tensor product $S\otimes_RS$. 
\begin{pro} \cite[Proposition 4.1.1]{CohSF} \label{epic}  Let $\varphi: R \rightarrow S$ be a ring homomorphism. Then, the following are equivalent:
\begin{itemize}
  \item[(i)] $\varphi$ is epic.
  \item[(ii)] in the $S$-bimodule $S\otimes_R S$, we have $x\otimes 1 = 1 \otimes x$ for every $x\in S$.
  \item[(iii)] the multiplication map $m: S\otimes_R S \rightarrow S$ given by $x\otimes y \mapsto xy$ is an isomorphism of $S$-bimodules.
\end{itemize}
\end{pro}

In addition, when $\varphi$ is a $*$-homomorphism from a $*$-ring to a $*$-regular ring, we have another nice characterization in terms of $*$-regular closures.
\begin{pro}
  Let $R$ be a $*$-ring, $\mathcal U$ a $*$-regular ring and $\varphi: R\rightarrow \mathcal U$ a $*$-homomorphism. Then $\varphi$ is epic if and only if $\mathcal U$ is the $*$-regular closure of $\varphi(R)$ in $\mathcal U$, i.e., $\mathcal U = \mathcal{R}(\varphi(R),\mathcal U)$.
\end{pro}
\begin{proof}
  The ``if" part is  \cite[Proposition 6.1]{Ja17base}.  In order to see the ``only if" part, observe that if $\varphi$ is epic, then the inclusion map $\mathcal{R}(\varphi(R),\mathcal U)\rightarrow \mathcal U$ is clearly epic, and so surjective by \cite[Proposition XI.1.4]{Ste}.
\end{proof}

The following lemma shows that in the above setting, the center $Z(\varphi(R))$ of the image of $R$ is contained in the center $Z(\mathcal{U})$ of $\mathcal{U}$:

\begin{lem} \label{center}
  Let $R$ be a subring of a ring $S$ with epic embedding $R\hookrightarrow S$. Then $Z(R)\subseteq Z(S)$. 
\end{lem}

\begin{proof}
  For every $a\in Z(R)$, the map $ S\times S \rightarrow S\otimes_R S$ given by $(x,y)\mapsto x\otimes ay$ is $R$-bilinear, and so there exists a well-defined homomorphism $\phi: S\otimes_R S \rightarrow S\otimes_R S$ with $\phi(x\otimes y) = x\otimes ay$. If $m: S\otimes_R S \rightarrow S$ denotes the multiplication map, then in view of Proposition \ref{epic}, we deduce that for all $x\in S$,
  $$
  xa = m\phi(x\otimes 1) = m \phi (1\otimes x) = ax
  $$
Therefore, $a\in Z(S)$.  
\end{proof}
\subsection{Sylvester rank functions}
The notions of Sylvester matrix rank function $\rk$ and Sylvester module rank function (on finitely presented modules) $\dim$ were introduced in \cite{Mlc}, and to learn more about their properties in our setting one can consult \cite[Section 5]{Ja17surv}. 

Let $R$  be a ring. A {\bf Sylvester matrix rank function} $\rk$ on $R$ is a function that assigns a non-negative real number to each matrix over $R$ and satisfies the following conditions.
 \begin{enumerate}
\item [(SMat1)] $\rk(M)=0$ if $M$ is any zero matrix and $\rk(1)=1$;
\item [(SMat2)]  $\rk(M_1M_2) \le \min\{\rk(M_1), \rk(M_2)\}$ for any matrices $M_1$ and $M_2$ which can be multiplied;
\item[(SMat3)] $\rk\left (\begin{array}{cc} M_1 & 0\\ 0 & M_2\end{array}\right ) = \rk(M_1) + \rk(M_2)$ for any matrices $M_1$ and $M_2$;
\item[(SMat4)] $\rk \left (\begin{array}{cc} M_1 & M_3\\ 0 & M_2\end{array}\right ) \ge \rk(M_1) + \rk(M_2)$ for any matrices $M_1$, $M_2$ and $M_3$ of appropriate sizes.
\end{enumerate}
Observe that over a von Neumann regular ring the notion of Sylvester matrix rank function coincides with the notion of pseudo-rank function that appears in \cite{Gdr}, and hence it is determined by its values on elements.

A {\bf Sylvester module rank function} $\dim$ on $R$ is a function that   assigns a non-negative real number to each finitely presented $R$-module   and satisfies the following conditions.
  \begin{enumerate}
\item [(SMod1)] $\dim \{0\} =0$, $\dim R =1$;
\item [(SMod2)]  $\dim(M_1\oplus M_2)=\dim M_1+\dim M_2$;
\item[(SMod3)] if $M_1\to M_2\to M_3\to 0$ is exact then
$$\dim M_1+\dim M_3\ge \dim M_2\ge \dim M_3.$$
\end{enumerate}
There exists  a natural bijection between Sylvester matrix and module rank functions over a ring.
\begin{pro}
Let $R$ be a ring. 
\begin{itemize}
   \item[(i)] If $\rk$ is a Sylvester matrix rank function on $R$, then we can define a Sylvester module rank function by assigning to any finitely presented module with presentation $M = R^m/R^nA$ for some $A\in Mat_{n\times m}(R)$, the value
   $$
   \dim(M):= m-\rk(A).
   $$
   This value does not depend on the given presentation.
   \item[(ii)] If $\dim$ is a Sylvester module rank function on $R$, then we can define a Sylvester matrix rank function by assigning to each $A\in Mat_{n\times m}(R)$, the value
   $$
   \rk(A):= m-\dim(R^m/R^nA).
   $$
\end{itemize}
We say in this case that $\rk$ and $\dim$ are {\bf associated}.
\end{pro}
The proof of this proposition can be found in \cite{Mlc} for integer-valued Sylvester rank functions but the proof works {similarly} without this additional assumption.

{As an easy example, if $R = \D$ is a division ring, we obtain from (SMod2) that there exists only one Sylvester module (and hence, matrix) rank function on $\D$, namely, the usual dimension $\dim_{\D}$ on vector spaces over $\D$.}

A Sylvester matrix rank function $\rk$ on $R$ is said to be {\bf faithful} if it does not vanish on elements of $R$, i.e., the (two-sided) ideal of $R$
$$\ker \rk=\{a\in R: \ \rk(a)=0\}$$
is equal  to $\{0\}$. From the property (SMat4) of a Sylvester matrix rank function it follows that if $\rk$ is faithful, then for any non-zero matrix $A$ over $R$, $\rk(A)\ne 0$.  Although the following lemma is just a standard observation, it is helpful to record it for future reference:
\begin{lem} \label{faithreg}
Let $\rk$ be  a faithful Sylvester matrix rank function on  a regular ring $\mathcal U$. Then a square matrix $A\in \Mat_n(\mathcal U)$ is invertible if and only if $\rk(A)=n$.
\end{lem}

\begin{proof}
 It is clear that any invertible matrix has maximum rank. Now, assume $x\in \mathcal U$ has rank $\rk(x)=1$ and let $y\in \mathcal U$ be such that $xyx=x$. Then, using \cite[Proposition 5.1(3)]{Ja17base},
 $$  \rk(yx-1) = \rk(x(yx-1))=0,$$ 
 and so, by faithfulness, $yx=1$. Similarly $xy=1$. Thus, $x$ is invertible. 
 
 For the general case, take $A\in \Mat_{n}(\mathcal U)$ with $\rk(A)=n$, and notice that $\rk^\prime = \frac{\rk}{n}$ defines a faithful rank on the regular ring $\Mat_n(\mathcal U)$ and $\rk^\prime(A)=1$. By the above reasoning, $A$ is invertible.
\end{proof}

We denote by $\mathbb{P}(R)$ the set of Sylvester matrix rank functions on $R$, which is a compact convex subset of the space of functions on matrices over $R$. A useful observation is that a ring homomorphism $\varphi: R \rightarrow S$ induces a continuous map $\varphi^{\sharp}: \mathbb{P}(S) \rightarrow \mathbb{P}(R)$, i.e., we can pull back any rank function $\rk$ on $S$ to a rank function $\varphi^{\sharp}(\rk)$ on $R$ by just defining
$$
\varphi^{\sharp}(\rk)(A) = \rk(\varphi(A))
$$  
for every matrix $A$ over $R$. We will often abuse the notation and write $\rk$ instead of $\varphi^{\sharp}(\rk)$ when it is clear that we speak about the rank function on $R$. Recently, H. Li \cite{Li19} proved that if $\phi$ is epic then $\phi^\sharp$ is injective, and so  $\mathbb{P}(S)$ can be seen as a closed subset of $\mathbb{P}(R)$.  

If $\rk$ is a Sylvester matrix rank function on a ring $S$, then  $\rk$ induces a faithful rank function on $S/\ker \rk$. 
 If $\rk $ is faithful on $S$, then we   say that $(S,\rk,\varphi)$ (or simply $S$, when  $\rk$ and $\varphi$ are clear from the context) is {\bf an envelope} of $\varphi^{\sharp}(\rk)$.

 We denote by $\mathbb{P}_{reg}(R)$ the space of Sylvester matrix rank functions that come from rank functions on a regular ring, {and we refer to its elements as \bf{regular rank functions}}. Since any quotient of a regular ring is also regular, this is the space of rank functions that admit {\bf a regular envelope}, i.e., an envelope $(\mathcal U,\rk ,\varphi)$ with $\mathcal U$ regular. Observe that a (regular) envelope is not unique in general.

If $\rk$ takes only integer values, then by a result of P. Malcolmson \cite{Mlc} there exists a  division algebra $\D$ such that $(\D,\rk_\D,{\varphi})$ is a regular envelope of $\rk$. Moreover we can assume that ${\varphi}$ is epic by passing to the division closure of ${\varphi}(R)$ in $\D$. Under these conditions $(\D,\rk_\D,{\varphi})$ { or, to shorten up, $(\D,\varphi)$,}  is called {\bf epic division $R$-ring}.  Two epic division $R$-rings $(\D_1,\varphi_1)$ and $(\D_2, \varphi_2)$ are said to be isomorphic if there exists an isomorphism of rings between them respecting the $R$-structure, i.e., there exists an isomorphism $\tau: \D_1\rightarrow \D_2$ such that $\varphi_2=\tau\circ\varphi_1$.

\begin{teo} \label{isomdiv} (\cite[Theorem 4.4.1]{CohSF}, \cite[Theorem 2]{Mlc}) Two {epic} division $R$-rings $(\D_1,\varphi_1)$ and $ (\D_2,\varphi_2)$ are isomorphic if and only if, for every matrix $A$ over $R$,
$$
\rk_{\D_1}(\varphi_1(A)) = \rk_{\D_2}(\varphi_2(A)).
$$ 
\end{teo}

Therefore, the epic regular envelope of an integer-valued rank function, which we will refer to as the \textbf{epic division envelope}, is completely determined by $\rk$ and hence unique up to isomorphism. We denote the set of integer-valued rank functions on a ring $R$ by $\mathbb{P}_{div}(R)$.  In the following, if $\D$ is an epic division $R$-ring we will also use $\rk_\D$ to denote the induced rank function on $R$.
 
When $R$ is a $*$-ring, $\mathcal U$ a $*$-regular ring, $\rk\in \mathbb P(\mathcal U)$ and $\varphi:R\to \mathcal U$ is a $*$-homomorphism we say that $\varphi^{\sharp}(\rk)$ is a $*$-regular rank, and we denote by $\mathbb{P}_{*reg}(R)$ the space of Sylvester matrix rank functions on $R$ obtained that way. Again, we can assume that $\rk$ is faithful, since $\mathcal U/\ker \rk$ is $*$-regular by Proposition \ref{proper}, and moreover we can assume that $\varphi$ is epic by passing to the $*$-regular closure of $\varphi(R)$ in $\mathcal U$. Under these conditions, the \textbf{$*$-regular envelope} $(\mathcal U,\rk,\varphi)$ will be called {\bf epic $*$-regular $R$-ring}. In view of the previous reasoning, anytime we consider a $*$-regular envelope $(\mathcal U,\rk,\varphi)$, we will assume that $\rk$ is faithful and $\varphi$ is epic.   
Both $\mathbb{P}_{reg}(R)$ and $\mathbb{P}_{*reg}(R)$ can be shown to be closed convex subsets of $\mathbb{P}(R)$ (\cite[Propositions 5.9 and 6.4]{Ja17base}).

 Two epic $*$-regular $R$-rings $(\mathcal U_1,\rk_1,\varphi_1)$ and $(\mathcal U_2,\rk_2, \varphi_2)$ are said to be isomorphic if there exists a $*$-isomorphism of rings between them respecting the $R$-structure and the rank, i.e., there exists a $*$-isomorphism $\tau: \mathcal U_1\rightarrow \mathcal U_2$ such that the following diagram commutes
$$
 \xymatrix{                    & \mathcal U_1 \ar[dr]^{\rk_1} \ar[dd]_{\tau} &   \\ 
 R \ar[ur]^{\varphi_1} \ar[dr]_{\varphi_2}  &                                    & \mathbb{R}_{\ge 0} \\
                               & \mathcal U_2 \ar[ur]_{\rk_2}                  &      
 }
$$ 

Notice that, inasmuch as $\mathcal U_1$ is regular, if the equality $\rk_2(\tau(x))=\rk_1(x)$ holds for every element $x\in\mathcal U_1$, then $\rk_2(\tau(A)) = \rk_1(A)$ for every matrix  over $\mathcal U_1$. 

In \cite{Ja17base}, the first author proved that, as it happens with epic division rings, an epic $*$-regular $R$-ring is completely determined by the values of the rank function on matrices over $R$.   
\begin{teo} \label{isom} \cite[Theorem 6.3]{Ja17base} Two epic $*$-regular $R$-rings $(\mathcal U_1,\rk_1,\varphi_1)$ and $(\mathcal U_2,\rk_2,\varphi_2)$ are isomorphic if and only if, for every matrix $A$ over $R$,
$$
\rk_1(\varphi_1(A)) = \rk_2(\varphi_2(A)).
$$ 
\end{teo}

\section{Natural extensions and Hughes-free Sylvester rank functions}\label{nat}

The notion of natural extension was introduced in \cite{Ja17base} in the context of (Laurent) polynomial rings (see also \cite[Section 8]{Ja17surv} for other variations of this concept). In this section we define the natural extension in the context of skew (Laurent) polynomial rings and we use it to define the notion of Hughes-free Sylvester rank function.

\subsection{The definition of the natural extension for skew (Laurent) polynomial rings}
Let $R$ be a ring and let $\tau$ be an automorphism of $R$.  
Recall that the  {\bf skew polynomial ring} $R[t;\tau]$  is a ring of polynomials in $t$ with coefficients in $R$ and subject to the relation $ta = \tau(a)t$, $a\in R$. The {\bf skew Laurent polynomial ring} $R[t^{\pm 1}; \tau]$ is a localization of $R[t;\tau]$ with respect to the set of powers of $t$. 
Similarly we can define the {\bf skew Taylor series ring} $R[[t;\tau]]$ and the   {\bf skew Laurent series ring} $R((t;\tau))$.

In the first place, to construct a rank function over  $R[t^{\pm 1}; \tau]$ from a rank function over $R$, we will need some compatibility between the latter and the twisted product, namely, $\tau$ has to preserve the rank.
We say that
a Sylvester matrix rank function $\rk$ on a ring $R$ is {\bf $\tau$-compatible}  if $\rk=\tau^{\sharp}(\rk) $, i.e., for every matrix $A$ over $R$, $\rk(A) = \rk(\tau(A))$.

We can rewrite this property in terms of the associated Sylvester module rank function. Let $M$ be a finitely presented left $R$-module, and denote by $t^nM$, $n\in \mathbb{Z}$, the finitely presented left $R$-module whose elements are of the form $t^nm$ for $m\in M$, with natural sum and $R$-product given by $r(t^nm) = t^n(\tau^{-n}(r)m)$. 
 {Observe that it is not true in general that $M\cong t^nM$. The next lemma states that $\tau$-compatibility is equivalent to both having the same rank for all $n$.}
\begin{lem}
Let $\rk$ be a Sylvester matrix rank function on a ring $R$ and $\dim$ its associated Sylvester module rank function. Let $\tau$ be an automorphism of $R$. Then $\rk$ is $\tau$-compatible   if and only if for every finitely presented $R$-module $M$, $\dim(M) = \dim(t M)$.
\end{lem}

\begin{proof}
 First notice that for every matrix $A\in \Mat_{n\times m}(R)$, the finitely presented left $R$-modules $R^m/R^n\tau(A)$ and $t(R^m/R^nA)$ are isomorphic, via $v+R^n\tau(A)\mapsto t(\tau^{-1}(v)+R^n A)$. Thus, if $\rk$ is $\tau$-compatible, then
 $$
 \begin{matrix*}[l]
  \dim(R^m/R^nA) &=& m-\rk(A) &=&m - \rk(\tau(A)) \\
        &=& \dim(R^m/R^n\tau(A)) &=& \dim(t(R^m/R^nA)).
  \end{matrix*}
 $$  
 Conversely, if $\dim(M) = \dim(tM)$ for every finitely presented $R$-module and we take a matrix $A\in \Mat_{n\times m}(R)$, then we can apply the same reasoning to the finitely presented module $R^m/R^nA$ to obtain that $\rk(\tau(A)) = \rk(A)$.
\end{proof}
Observe that the previous proposition  implies also that $\dim(M) = \dim(t^nM)$ for every $n\in \mathbb{Z}$ if $\rk$ is $\tau$-compatible.  

Suppose that we have a ring $R$ and a Sylvester rank function $\rk$ on $R$.   Let $\dim$ be the associated Sylvester matrix rank function. Then, for every $i$, we have a ring homomorphism
$$
\begin{matrix}
R[t;\tau] &\longrightarrow& \End_{R}(R[t;\tau]/R[t;\tau]t^i)\\
p&\longmapsto&\phi_{R,i}^p
\end{matrix}
$$ 
where $\phi_{R,i}^p$ is given by right multiplication by $p$.
Since the codomain is isomorphic to $\Mat_i(R)$, we can pull back to $R[t;\tau]$ the rank induced by $\rk$ on $\Mat_i(R)$. This means that we have rank functions $\wrk_i$ on $R[t;\tau]$ such that if $A\in \Mat_{n\times m}(R[t;\tau])$, then
$$
\wrk_i(A) = \frac{\rk(B)}{i}
$$
where $B\in \Mat_{in\times im}(R)$ is the matrix associated to the $R$-homomorphism of free $R$-modules $\phi_{R,i}^A:(R[t;\tau]/R[t;\tau]t^i)^n\rightarrow (R[t;\tau]/R[t;\tau]t^i)^m$ given by right multiplication by $A$ with respect to some bases in the domain and codomain. Of course, this is independent of the choice of the bases, and so we can write $\rk(\phi_{R,i}^A)$ instead of $\rk(B)$. 

Assume that $\rk$ is $\tau$-compatible.  Let $\wrk \in \mathbb  \mathbb \mathbb P(R[t;\tau])$. We say that $\wrk$ is the {\bf natural extension of $\rk$} if 
 $$\wrk=\displaystyle \lim _{i\to \infty} \wrk_i \in \mathbb P(R[t;\tau]),$$
{i.e., for every $A\in \Mat_n(R[t;\tau])$ there exists the limit $\displaystyle \lim_{i\to \infty} \wrk_i(A)$ and it is equal to $\wrk(A)$}. 
 
 Observe that in this case $\wrk(t)$ is equal to 1. Indeed, the matrix associated to $\phi_{R,i}^t$ with respect to the canonical basis in both the domain and codomain is the $i\times i$ matrix $\begin{pmatrix} 0 & I_{i-1} \\ 0 & 0\end{pmatrix}$, which is of rank $i-1$, by the properties of rank functions.
 Therefore,
 $$\wrk(t) = \displaystyle \lim _{i\to \infty} \wrk_i (t) = \displaystyle \lim _{i\to \infty}  \frac{i-1}{i} = 1$$   Thus, $\wrk$ can be extended to $R[t^{\pm 1};\tau]$ (see \cite[Corollary 5.5]{Ja17base}).   We also denote this extension by $\wrk$  and  we will call it {\bf the natural extension of $\rk$} to $R[t^{\pm 1};\tau]$. 
 
We do not know what are the necessary conditions for the existence of natural extensions. In \cite[Proposition 7.5]{Ja17base} it is shown that if $\tau$ is the identity automorphism, then the natural extension exists if $\rk$ is regular. In the next section we give an analog of this result in the case where  $\tau$ is an arbitrary automorphism.

\subsection{On the existence and characterizations of the natural extension}
 A Sylvester module rank function $\dim$ on a ring $R$ is {\bf exact} if for every surjective map between finitely presented modules $\phi: M\twoheadrightarrow N$ we have
  $$
  \dim(M)-\dim(N) = \inf\{\dim(L): L~\textrm{finitely presented}~\textrm{and}~ L\twoheadrightarrow ker~\phi\}.
  $$

Since every finitely presented module over a von Neumann regular ring is projective, we have that every short exact sequence of finitely presented modules splits, and so any Sylvester module rank function over a von Neumann regular ring is exact. Notice that the exactness condition seems to be necessary if one wants to obtain an extension which behaves additively on exact sequences. Indeed, we have the following proposition due to S.Virili (see also \cite[Corollary 4.3]{Li19}).

 \begin{pro} \label{extlength}\cite{Vi1}
  Let $\dim$ be an exact Sylvester module rank function on a ring $R$. Consider, for every finitely generated module,
  $$
    \dim(M) = \inf\{\dim(L): L~\textrm{finitely presented and}~ L\twoheadrightarrow M\}
  $$
  and set for any $R$-module
  $$
    \dim(M) = \sup\{\dim(L): L~\textrm{finitely generated and}~L\leq M\}
  $$
  The extended function $\dim: R$-$mod \rightarrow \mathbb{R}_{\geq 0}\cup\{\infty\}$ is a well-defined normalized length function, i.e., it satisfies:
  \begin{itemize}
    \item[(1)] (Normalization) $\dim(R)=1$.
    \item[(2)] (Continuity) For every $R$-module, $$\dim(M) = \sup\{\dim(L): L~\textrm{finitely generated and}~L\leq M\}$$
    \item[(3)] (Additivity) For every exact sequence $0\rightarrow M_1\rightarrow M_2 \rightarrow M_3 \rightarrow 0$, we have $\dim(M_2) = \dim(M_1)+\dim(M_3)$
  \end{itemize}
In addition, the correspondence between exact Sylvester module rank functions and normalized length functions is  {bijective}. More precisely, the restriction of a normalized length function to finitely presented modules is an exact Sylvester module rank function, and from this restriction we can recover it by means of the previous procedure.  
\end{pro} 

In view of this proposition, we can (and sometimes we will) indistinctly talk about an exact Sylvester module rank function and its associated normalized length function. Nevertheless, we will usually try to maintain the corresponding terminology in order to keep in mind the extent of the definition. It is important to notice that if $\dim$ is an exact $\tau$-compatible Sylvester module rank function, then its associated normalized length function is also $\tau$-compatible in the sense that for any $R$-module $M$, we have $\dim(M) = \dim(tM)$. This follows easily from the property for finitely presented modules and the way we extend $\dim$.

 We are now in position to present the construction of the natural extension of an exact Sylvester rank function using the construction from \cite[Theorem B and Definition 4.3]{Vi2}.

\begin{pro}\label{natext}
 Let $\dim$ be a $\tau$-compatible normalized length function on a ring $\mathcal U$ and let $\rk$ be the Sylvester matrix rank function associated with $\dim$. Define, for every $\mathcal U[t^{\pm 1};\tau]$-module $M$
 $$
  \wdim(M) = \sup \{E_{M,N}: N \textrm{\ is $\mathcal U$-submodule of M and\ } \dim(N)<\infty\},
 $$
 where
 $$
  E_{M,N} = \lim_{i\rightarrow \infty} \frac{\dim(N + tN+\dots+t^{i-1}N)}{i}.
 $$
Then $\wdim$ is a well-defined normalized length function on $\mathcal U[t^{\pm 1};\tau]$, and its associated Sylvester matrix rank function $\wrk$ is the natural extension of $\rk$   to $\mathcal U[t^{\pm 1};\tau]$. 
\end{pro}
 
This has been studied in \cite{Ja17base} for the case of Laurent polynomial rings $R[t^{\pm 1}]$, and, in fact, almost the same proofs apply in this setting with very slight modifications regarding the twist $ta= \tau(a)t$. In particular, the following characterizations of the natural extension hold.

\begin{pro} \label{prop7.2}
  Let $R$ be a ring, $\tau$ an automorphism of $R$ and $\dim$ a $\tau$-compatible normalized length function on $R$ with associated Sylvester matrix rank function $\rk$. Then, for every left ideal $I$ of $R[t^{\pm 1};\tau]$,
  $$
   \wdim(I) = \lim_{k\rightarrow \infty} \frac{\dim(P_{k-1})}{k}
  $$
where $P_{k}$ is the set of polynomials in $R[t;\tau]$ of degree at most $k$ contained in $I$. Moreover, if $R$ is regular, then the above limit equals
  $$
  \wdim(I) = \sup~\left\lbrace \rk(a_0): a_0\in R ~\textrm{and}~\exists n\geq 0, \exists a_1,\dots,a_n\in R~ \textrm{s.t.}~\sum_{i=0}^n a_it^i \in I\right\rbrace
  $$
\end{pro}

\begin{pro} \label{prop7.7}
Let $\mathcal U$ be a regular ring, $\tau$ an automorphism of $\mathcal U$ and $\rk$ a $\tau$-compatible Sylvester matrix rank function on $\mathcal U$. Let $\rk'$ be a rank on $\mathcal U[t^{\pm 1};\tau]$ that extends $\rk$. Then $\rk' $ is the natural extension of $\rk$ if and only if, for any matrix $A\in \Mat_{n}(\mathcal U)$, we have 
$$
 \rk'(I_n+At) = n
$$
\end{pro}

Now  assume that $\mathcal U$ is positive definite $*$-regular and $\tau$ is a $*$-automorphism. In this case we will show that the natural extension  $\wrk$ of a Sylvester matrix rank function $\rk$ on $\mathcal U$ is  a $*$-regular Sylvester rank function on $\mathcal U[t^{\pm 1};\tau]$. To do so first observe that, provided $\tau$ is a $*$-automorphism, we can endow $\mathcal U[t^{\pm 1};\tau]$ with an involution by setting $t^*=t^{-1}$. This is indeed consistent with the twist $ta = \tau(a)t$ because $(ta)^*=a^*t^{-1}=t^{-1}\tau(a^*) = t^{-1}\tau(a)^*=(\tau(a)t)^*$. 

 Since $\mathcal U$ is positive definite, $\Mat_n(\mathcal U)$, and so $\End_{\mathcal U}(\mathcal U[t;\tau]/\mathcal U[t;\tau]t^n)$, is $*$-regular for every $n$. In this ring we have the rank $\rk_n=\frac{\rk}{n}$ (from where we obtained $\wrk_n$).  Let us fix a non-principal ultrafilter $\omega$ on $\mathbb{N}$.   We can construct a rank function $\rk_{\omega}:=\displaystyle{\lim_{\omega}}  ~\pi_n^{\sharp}(\rk_n)$ on the $*$-regular ring $\displaystyle{\prod_{n=1}^\infty} \End_\mathcal U(\mathcal U[t;\tau]/\mathcal U[t;\tau]t^n)$, where $\pi_n$ is the natural projection onto the $n$-th factor.  We denote 

\begin{equation}\label{eq4}
\mathcal P_{\omega, \tau}^{\mathcal U}:= \left(\prod_{n=1}^\infty \End_{\mathcal U}(\mathcal U[t;\tau]/\mathcal U[t;\tau]t^n)\right) \Big{/} \ker \rk_{\omega} .
\end{equation}
and $\rk_{\omega}$ defines a faithful rank function on $\mathcal P_{\omega, \tau}^{\mathcal U}$. Consider the natural map $f_{\omega}: \mathcal U[t;\tau] \rightarrow \mathcal P_{\omega, \tau}^{\mathcal U}$, where $p\mapsto (\phi_{\mathcal U,n}^p)_n + \ker \rk_{\omega}$, and observe that  the definition of natural extension  tells us that  as a rank over $\mathcal U[t;\tau]$, $\wrk = {f_{\omega}}^{\sharp}(\rk_{\omega})$. Finally, since $\wrk(t)=1$, $f_{\omega}$ extends to a homomorphism
\begin{equation}\label{eq5}
 f_{\omega}: \mathcal U[t^{\pm 1};\tau]\rightarrow \mathcal P_{\omega, \tau}^{\mathcal U}
\end{equation} 
and $\wrk = {f_{\omega}}^{\sharp}(\rk_{\omega})$. As in \cite{Ja17base} one may check that $f_{\omega}$ is a $*$-homomorphism, and consequently, the following proposition.

\begin{pro}\label{prop7.8}
Let $\mathcal U$ be a positive definite $*$-regular ring, $\tau$ a $*$-automorphism of $\mathcal U$. If $\rk$ is a $\tau$-compatible Sylvester matrix rank function on $\mathcal U$, then the natural extension of $\rk$ on $\mathcal U[t^{\pm 1};\tau]$ is a $*$-regular rank function on $\mathcal U[t^{\pm 1};\tau]$.
\end{pro}
 
We can use the previous results to show the existence  of the natural extension for either a $*$-regular or an integer-valued Sylvester rank function. We describe this in two separate propositions for the sake of clarity.  Although the proof of the following proposition is similar to the proof of \cite[Proposition 7.5]{Ja17base}, it presents   additional technical difficulties that do not appear when $\tau$ is the identity automorphism.

\begin{pro} \label{prop7.4}
 Let $R$ be a $*$-ring, $\tau$ a $*$-automorphism of $R$ and $\rk$ a $\tau$-compatible $*$-regular Sylvester matrix  rank function   on $R$.  
Let  $(\mathcal U,\rk^{\prime},\varphi)$  be  the $*$-regular envelope of $\rk$.

\begin{itemize}
\item[(1)]
 Then $\tau$ can be extended to a $*$-automorphism of $\mathcal U$ (also denoted $\tau$) such that  $\rk^{\prime}$ is $\tau$-compatible. 
\item[(2)] Denote also by $\varphi$ the induced map $R[t^{\pm 1};\tau]\to \mathcal U[t^{\pm 1};\tau]$. Then there exists the natural extension $\widetilde{rk^\prime}$ of $\rk^\prime$ to  $	\mathcal U[t^{\pm 1};\tau]$ and $ \wrk=\varphi^{\sharp}(\widetilde{\rk^{\prime}})$ is  the natural extension of  $\rk$ to  $R[t^{\pm 1};\tau]$.
 \item[(3)] Endow $R[t^{\pm 1};\tau]$ and $\mathcal U[t^{\pm 1};\tau]$ with an involution by setting $t^*=t^{-1}$. If $\mathcal U$ is positive definite, then  $\wrk$ is a $*$-regular Sylvester matrix rank function on $R[t^{\pm 1};\tau]$.
\end{itemize}
\end{pro}

\begin{proof}  Observe that $(\mathcal U, \rk^\prime, \varphi\circ \tau)$ is also an epic $*$-regular $R$-ring. Since $\rk$ is $\tau$-compatible, by  Theorem \ref{isom} $\tau$ can be extended to a $*$-automorphism of $\mathcal U$ preserving the rank $\rk^\prime$. 
Hence  we have the following commutative diagram
 $$
 \xymatrix{                    & \mathcal U \ar[dr]^{\rk^{\prime}} \ar@{.>}[dd]_{\exists\tau} &   \\ 
R \ar[ur]^{\varphi} \ar[dr]_{\varphi \circ \tau}  &                                    & \mathbb{R}_{\ge 0} \\
                               & \mathcal U \ar[ur]_{\rk^{\prime}}                  &      
 }
$$ 
Now, inasmuch as $\rk^{\prime}$ is exact and $\tau$-compatible, using Proposition \ref{natext}, we obtain that there exists its natural extension $\widetilde{\rk^{\prime}}$, which is a regular Sylvester rank function on $\mathcal U[t^{\pm 1};\tau]$.  Since $\rk=\varphi^\sharp(\rk^\prime)$  and $\wrk_i=\varphi^\sharp(\widetilde{\rk^\prime_i})$, we conclude that   $\wrk=\varphi^\sharp(\widetilde{\rk^\prime})$ is the natural extension of $\rk$. 

 Part 3 follows from Proposition \ref{prop7.8} because the extension $\varphi: R[t^{\pm 1};\tau]\rightarrow \mathcal U[t^{\pm 1};\tau]$ is a $*$-homomorphism. 
\end{proof}
As a consequence of the latter proposition and Proposition \ref{prop7.7} we have the following corollary.
 
\begin{cor} \label{wlimext}
 Let $R$ be a $*$-ring, $\tau$ a $*$-automorphism of $R$ and $\{\rk_i\}$ a family of $\tau$-compatible $*$-regular rank functions. For every $i\in \N$, let $\wrk_i$ be the natural extension of $\rk_i$ to $R[t^{\pm 1};\tau]$. Then, for every non-principal ultrafilter on $\N$, $\displaystyle \lim_{\omega} \wrk_i$ is the natural extension of $\rk_{\omega} = \displaystyle \lim_{\omega} \rk_i$.
\end{cor}

\begin{proof}
  Since $\rk_i$ is $*$-regular and $\tau$-compatible for every $i$, $\rk_{\omega}$ is also $*$-regular and $\tau$-compatible, and therefore, their natural extensions exist by Proposition \ref{prop7.4}. Let $(\mathcal U_i, \rk'_i, \varphi_i)$ be the $*$-regular envelope of $\rk_i$ and set $\mathcal U = \prod \mathcal U_i$, $\varphi=(\varphi_i)$. Consider
  $$
  \begin{matrix}
   \xymatrix{ R \ar@{->}[r]^{\varphi} \ar@{->}[rd]_{\varphi_i} & \mathcal{U} \ar@{->}^{\pi_i}[d] \\
      & \mathcal U_i
     } &    \xymatrix{ R[t^{\pm 1};\tau] \ar@{->}[r]^{\varphi} \ar@{->}[rd]_{\varphi_i} & \mathcal{U}[t^{\pm 1};\tau'] \ar@{->}^{\pi_i}[d] \\
      & \mathcal U_i[t^{\pm 1};\tau_i]
     }
   \end{matrix}
 $$
  where $\tau_i$ is the $*$-automorphism of $\mathcal U_i$ given in Proposition \ref{prop7.4}(1), $\tau' = (\tau_i)$ and $\pi_i$ is the natural projection. By construction, $\rk'_{\omega}:= \displaystyle \lim_{\omega} \pi_i^{\sharp}(\rk'_i)$ satisfies $\rk_{\omega} = \varphi^{\sharp}(\rk'_{\omega})$. Now, since $\mathcal U$ is regular and $\rk'_{\omega}$ is $\tau'$-compatible, there exists its natural extension on $\mathcal U[t^{\pm 1};\tau']$, and so $\wrk_{\omega} = \varphi^{\sharp}(\wrk'_{\omega})$.   In addition, notice that $\displaystyle\lim_{\omega} \pi_i^{\sharp}(\wrk_i')$ extends $\rk_{\omega}'$ and that, since $\wrk_i'$ is the natural extension of $\rk_i$, we obtain from Proposition \ref{prop7.7} that
$$
 \left(\lim_{\omega} \pi_i^\sharp(\wrk_i')\right)(I_n+At) = \lim_{\omega} \wrk_i'(I_n+\pi_i(A)t) = \lim_{\omega} n = n.
$$
As a consequence, another application of Proposition 3.5 gives us that
  $$\wrk'_{\omega} = \lim_{\omega} \pi_i^{\sharp}(\wrk'_i)$$
and, therefore, $\wrk_{\omega} = \varphi^{\sharp}(\wrk'_{\omega}) = \displaystyle\lim_{\omega} \varphi_i^{\sharp}(\wrk_i') = \displaystyle\lim_{\omega} \wrk_i$.
\end{proof}

Now we consider integer-valued Sylvester rank functions.
\begin{pro} \label{prop7.4iv}
 Let $R$ be a ring, $\tau$ an automorphism of $R$ and $\rk$ an integer-valued $\tau$-compatible  Sylvester  matrix rank function on $R$.  
Let  $(\D ,\varphi)$  be  the epic division envelope of $\rk$.

\begin{itemize}
\item[(1)]
 Then $\tau$ can be extended to an automorphism of $\D$ (also denoted $\tau$) and  $\rk_\D$ is (automatically) $\tau$-compatible. 
\item[(2)] Denote also by $\varphi$ the induced map $R[t^{\pm 1};\tau]\to \D[t^{\pm 1};\tau]$. Then there exists the natural extension $\wrk_\D$ of $\rk_\D$ to  $\D[t^{\pm 1};\tau]$ and $\wrk= \varphi^{\sharp}(\wrk_\D)$ is  the natural extension of  $\rk$ to  $R[t^{\pm 1};\tau]$.
\item[(3)] The function $\wrk$ is integer-valued and its epic division envelope is isomorphic to the {Ore division} ring of fractions of $\D[t^{\pm 1};\tau]$.
\end{itemize}
\end{pro}
\begin{proof}
 The first two statements are proved as in the previous proposition invoking Theorem \ref{isomdiv} instead of Theorem \ref{isom} and taking into account that in a division ring there exists only one rank function.
 
  To prove (3) observe that, for every $p = a_it^i+a_{i+1}t^{i+1}+\dots\in \D[t;\tau]$ with $a_i\neq 0$, we have that $\wrk_\D(p)\ge \rk_\D(a_i) = 1$ by Proposition \ref{prop7.2}, and so we can use \cite[Corollary 5.5]{Ja17base} to extend $\wrk_\D$ not only to $\D[t^{\pm 1};\tau]$ but to the Ore division ring of fractions $\D(t;\tau)$. Again, by uniqueness of rank in a division ring, $\wrk_\D = \rk_{\D(t;\tau)}$ takes integer values, and since the composition $R[t^{\pm 1};\tau] \to \D[t^{\pm 1};\tau] \to \D(t;\tau)$ is also epic, we conclude that $\D(t;\tau)$ is the epic division envelope of $\wrk$.
\end{proof}

To finish this section, let $G$ be a group and suppose that $H$ is a non-trivial finitely generated indicable subgroup of $G$. Take a decomposition $H = N\rtimes_{\tau} <t>$ with $t\in H$, and notice that $K[H]\cong  K[N][t^{\pm 1};\tau]$. Assume that $\rk$ is a $*$-regular Sylvester matrix  rank  function on  $K[G]$ with positive definite $*$-regular envelope $(\mathcal U,\rk^\prime,\varphi)$, and denote by $\rk_{|K[H]}$ and $\rk_{|K[N]}$ the restrictions of $\rk$ to $K[H]$ and $K[N]$, respectively. 

In  Section \ref{intro}, we anticipated that for the proof of the main theorem we were going to construct an environment in which we could compare $\mathcal D_{H,\mathcal U}$ and $\mathcal D_{N,\mathcal U}((t;\tau))$, where $\D_{H,\U}$ and $\D_{N,\U}$ denote, respectively, the division closures of $\varphi(K[H])$ and $\varphi(K[N])$ inside $\U$. This object will be the $*$-regular ring $\mathcal P_{\omega,\tau}^{\mathcal U_{N}}$ that appeared in (\ref{eq4}), where  $\U_N$ is the $*$-regular closure of $\varphi(K[N])$ in $\U$.

On the one hand,  
we have an injective homomorphism 
$$
\mathcal U_{N}[[t;\tau]] \rightarrow \prod_{n=1}^{\infty} End_{\mathcal U_{N}}(\mathcal U_{N}[t;\tau]/\mathcal U_{N}[t;\tau]t^n).
$$ 
This induces a homomorphism  $\psi: \mathcal U_{N}[[t;\tau]] \rightarrow  \mathcal P_{\omega,\tau}^{\mathcal U_{N}}$. Moreover, $\psi$ is injective, because if $p=a_it^i+a_{i+1}t^{i+1}+\ldots \in  \mathcal U_{N}[[t;\tau]]$ with $a_i\ne 0$, then we have that $\rk_{\omega}'(\psi(p))\ge \rk^\prime(a_i)>0$.   Indeed, by the properties of Sylvester rank functions and by $\tau$-compatibility and faithfulness of $\rk'$ we have that 
$$\rk_{\omega}'(\psi(p)) = \lim_{\omega} \frac{\rk'(\phi_{\mathcal U_N,n}^p)}{n} \ge \lim_{\omega} \frac{(n-i)\rk'(a_i)}{n} = \rk'(a_i) > 0. $$  

 Since $t$ is invertible in $ \mathcal P_{\omega,\tau}^{\mathcal U_{N}}$, the property of universal localization allows us to extend $\psi$ to an embedding  
\begin{equation} \label{eqLaurent}
\mathcal U_{N}((t;\tau))\hookrightarrow  \mathcal P_{\omega,\tau}^{\mathcal U_{N}}.
\end{equation}

On the other hand, to show that $\U_{H}$ can be identified with a subring of $\PP_{\omega,\tau}^{\U_{N}}$, we need the rank $\rk$ to satisfy the following additional property 
\begin{equation}\label{cond}
 \rk_{|K[H]} \textrm{is the natural extension of\ }{\rk_{|K[N]}}. 
\end{equation} 
(Observe that the existence of  the natural extension of $\rk_{|K[N]}$ follows from Proposition \ref{prop7.4}.)  Indeed, let us  assume that (\ref{cond}) holds for $H$ and $N$.
Since $K[H]\cong  (K[N])[t^{\pm 1};\tau]$, we can consider the $*$-map $f_{\omega}: \mathcal{U}_{N}[t^{\pm 1};\tau]\rightarrow \mathcal P_{\omega,\tau}^{\mathcal U_{N}}$ as in  (\ref{eq5}).   Denote by $\rk^\prime_{\mathcal U_H}$ the restriction of $\rk^\prime$ to $ \mathcal U_{H}$. Then, if $\varphi$ denotes also the induced homomorphism $\varphi: (K[N])[t^{\pm 1};\tau] \rightarrow \mathcal U_N[t^{\pm 1};\tau]$,
$$\rk_{|K[H]}=\widetilde{\rk_{|K[N]}}={\varphi^{\sharp}}\circ f_\omega^{\sharp}(\rk_{\omega}').$$
This means that $(\mathcal U_{H}, \rk^{\prime}_{\mathcal U_H}, \varphi)$ and $(\mathcal{R}(f_{\omega} \circ {\varphi}(K[H]), \mathcal P_{\omega,\tau}^{\mathcal U_{N}}), \rk_{\omega}', f_{\omega}\circ {\varphi})$ are both $*$-regular envelopes for $\rk_{K[H]}$, and so they are isomorphic. Thus, we can think that $\mathcal U_{H}\subseteq \mathcal P_{\omega,\tau}^{\mathcal U_{N}}$ and $\rk^{\prime}_{|\mathcal U_H}$ is the restriction of $\rk_{\omega}'$.

Thus, assuming the condition (\ref{cond}), we have constructed the following diagram.
\begin{equation} \label{dia6}
\begin{gathered}
 \xymatrix{ \mathcal U_{N}~ \ar@{^{(}->}[r] \ar@<-1em>@{^{(}->}[d]  & \mathcal U_{H} \ar@<-1em>@{^{(}->}[d] \\
 \mathcal U_{N}((t;\tau))~ \ar@{^{(}->}[r] & \mathcal P_{\omega,\tau}^{\mathcal U_{N}} 
 }
\end{gathered}
\end{equation}
with $\rk_{\omega|\mathcal U_{H}}' = \rk^{\prime}_{|\mathcal U_H}$.

\subsection{Hughes-free rank functions}\label{hughes} 
  In this subsection we are going to introduce a property for rank functions on certain crossed products. Given an algebra $R$ and a group $G$, a crossed product $R*G$ is a $G$-graded ring $R*G = \bigoplus_{g\in G} R_g$ such that $R_{1_G} = R$ and for every $g\in G$ there exists an element $u_g\in R_g$ invertible in $R*G$.

As we have done in the previous sections, for group rings $R[G]$ we can (and we will) just set $u_g = g$. Since the multiplication is extended from the one in $G$, this way we canonically embed $G$ (as a group) in $R[G]$. Therefore, we can define $R^{\times}G$ to be the subgroup of the group of units in $R[G]$ consisting on the elements $rg$ for a unit $r\in R$ and a group element $g\in G$. 

In the general setting this identification is no longer possible, but still the set $R^{\times}G:=\{ru_g: r\in R^{\times}, g\in G\}$ is a subgroup of units of $R*G$ containing $R^{\times}$ as a normal subgroup and such that $R^{\times}G/R^{\times} \cong G$. When $E$ is a division ring and $G$ is locally indicable, G. Higman proved (\cite{Hi40}) that $E^{\times}G$ is precisely the set of units of $E*G$. 

In the latter situation,  we are going to introduce the property of Sylvester rank functions on $E*G$ that will be central in the proof of the main theorem of Section \ref{proofs}, namely, the Hughes-free property. This property is closely related to the one imposed to construct diagram (\ref{dia6}) and it is the analog of the Hughes-free property for epic division $E*G$-rings that appears in \cite{Hu70}. In fact we will remark that an epic division $E*G$-ring is Hughes-free if and only if the corresponding Sylvester rank function is Hughes-free. Moreover, we will see in Section \ref{proofs} that if $K$ is a subfield of $\CC$ closed under complex conjugation, then any $*$-regular Hughes-free Sylvester rank function on $K[G]$ with positive definite $*$-regular envelope takes integer values, and so its $*$-regular envelope is a division ring.

First let us recall the definition of Hughes-free epic division $E*G$-ring. If $H$ is a non-trivial finitely generated subgroup of $G$, then we can express $H = N\rtimes C$, where $C$ is infinite cyclic. Let $t$ be a preimage of a generator of $C$ under $E^{\times}H \rightarrow E^{\times}H/E^{\times}\cong H$. Then, left conjugation by $t$ induces an automorphism $\tau: E*N \rightarrow E*N$ and $E*H \cong (E*N)[t^{\pm 1}; \tau]$. Moreover, if $(\D, \varphi)$ is an epic division $E*G$-ring, then $\tau$ can be extended to an automorphism (also denoted by $\tau$) of $\D_{N,\D}$, the division closure of $\varphi(E*N)$ in $\D$. \color{black}

We say that an epic division $E*G$-ring $(\D,\varphi)$ is {\bf Hughes-free} if for every non-trivial finitely generated subgroup $H$ of $G$ and for every  expression $H = N\rtimes C$ as before, we have that $ \D_{H,\D}$ is isomorphic (as an $E*H$-ring) to the Ore ring of fractions of $\D_{N,\D}[t^{\pm 1};\tau]$. Here $\D_{H,\D}$ is the division closure of $\varphi(E*H)$ in $\D$.  This definition is equivalent to saying that $\{t^i\}_{i\in \N}$ are $\D_{N,\D}$-linearly independent, in the sense that there exists no non-trivial expression $a_0+a_1t+\dots+a_nt^n = 0$ in $\D_{H, \D}$ with coefficients in $\D_{N,\D}$. If $(\D,\varphi)$ is Hughes-free, then $\varphi$ is injective and we will consider $E*G$ as a subring of $\D$.

 We  
introduce the following generalization. Let $\rk$ be a Sylvester matrix rank function on the crossed product $E*G$. We say that $\rk$ is {\bf Hughes-free} if for every non-trivial finitely generated subgroup $H$ of $G$ and every expression $H= N\rtimes C$  as above, we have that the natural extension of $\rk|_{E*N}$ exists and coincides with $\rk|_{E*H}$. 

The next lemma states that this is indeed a generalization of the Hughes-free notion for epic division rings.
\begin{lem} \label{rkdiv}
 An epic division $E*G$-ring $\D$ is Hughes-free if and only if $\rk_\D$ is Hughes-free as a Sylvester matrix rank function on $E*G$.
\end{lem}
 \begin{proof} Let $H$ be a finitely generated subgroup of $G$ and assume that we have that $H = N\rtimes C $, for $C$ infinite cyclic. Let $t\in E^{\times}H$ be such that its image under $E^{\times}H\rightarrow  H$ generates $C$, and let $\tau$ denote the automorphism of $E*N$ induced by left conjugation by $t$.  By Proposition \ref{prop7.4iv}(3) we know that $\D_{N,\D}(t;\tau)$ is the epic division envelope of $ \widetilde{\rk_{\D_{N,\D}}}$ viewed as a Sylvester matrix rank function on $E*H$. Thus, the lemma follows from the Hughes-freeness definition and Theorem \ref{isomdiv}.
 \end{proof}
 Now, let us present the main example of a Hughes-free Sylvester rank function.

\begin{pro}\label{mainexample}
  Let $G$ be a group, $K$ a subfield of $\CC$ and $H$ a non-trivial finitely generated indicable subgroup of $G$. If $N\trianglelefteq H$ is a normal subgroup of $H$ such that $H/N\cong \Z$, then, as a rank function over $K[H]$, $\rk_H$ is the natural extension of $\rk_N$. In particular, if $G$ is locally indicable, then  $\rk_G$ restricted to $K[G]$ is Hughes-free.
\end{pro}
 
\begin{proof}
This is a particular application of \cite[Corollary 12.2]{Ja17surv}.
\end{proof}

Observe that if $\rk$ is a $*$-regular Hughes-free Sylvester matrix  rank  function on $K[G]$ with positive definite $*$-regular envelope $(\mathcal U,\rk^\prime,\varphi)$, then we can construct a diagram (\ref{dia6}) for any non-trivial finitely generated subgroup $H$ of $G$. 
In Section \ref{proofs} we will prove that $\mathcal U$ is the Hughes-free epic division $K[G]$-ring, and so, in fact, that $\mathcal U_{H}$ can be seen as a subring of $\mathcal U_{N}((t;\tau))$. 

\section{Rational $U$-semirings}\label{rat}

In this section we recall the notion of rational $U$-semiring for a multiplicative group $U$ and two of the main examples that appear in \cite{DHS}. We also present two new examples. The first is the case of division $E*G$-closures, where $E*G$ denotes a crossed product of a skew field $E$ with a group $G$. The second is the case of epic $*$-regular $K[G]$-rings for any subfield $K$ of $\CC$ closed under complex conjugation, which will be shown to be a $K^{\times}G$-rational semiring with rational operation given by taking relative inverses. Except for some minor notation details, we will stick to the definitions and notation used in \cite{DHS}. 
  
By a semiring $R$ we understand a set together with an associative addition and an associative product with identity element $1_R$ which is distributive over the addition. Let $U$ be a multiplicative group and let $R$ be  a semiring. We say that $R$ is a {\bf rational $U$-semiring} if
  
  \begin{enumerate}
  \item There  is a map  $\diamond: R \rightarrow R$ (with $r\mapsto r^{\diamond}$) defined on $R$ (this is a rational structure on $R$).
  
  \item $R$ is a $U$-biset ($U$ acts on both sides of $R$ in a compatible way, i.e. $(ur)v = u(rv)$ for any $u,v\in U$, $r\in R$).
  
  \item For every $u,v \in U$ and $r\in R$, 
$
(urv)^{\diamond} = v^{-1} r^{\diamond} u^{-1}.
$

  \end{enumerate}
 
 A \textbf{morphism of rational $U$-semirings} $\Phi:R_1\rightarrow R_2$ is a map respecting all of the operations, i.e., satisfying, for all $r,r'\in R_1$ and $u,v\in U$
\begin{enumerate}
\item $ \Phi(r+r') = \Phi(r) + \Phi(r')$,  $\Phi(1_{R_{1}}) = 1_{R_2}$ and  $\Phi(rr') = \Phi(r)\Phi(r')$;
 \item $\Phi(r^{\diamond }) = \Phi(r)^{\diamond}$;
 \item $ \Phi(urv)=u\Phi(r)v$.
\end{enumerate}

Each of the following subsections is devoted to showing a particular example of rational $U$-semiring. Notice that a $U$-semiring is also a $V$-semiring for every $V\leq U$.

\subsection{Finite rooted trees}

Let $\mathcal{T}$ be the set of all finite (oriented) rooted trees up to isomorphism. We will just recall here that $\mathcal{T}$ has a well-order satisfying some desirable properties and that can be trivially seen to be a $U$-semiring for any multiplicative group $U$. This order will define later a measure of complexity of elements in $\Rat (U)$ and, therefore, a measure of complexity of elements in division $E*G$-closures and epic $*$-regular $K[G]$-rings.

Denote by $0_{\mathcal{T}}$  the one-vertex tree. If  $0_{\mathcal{T}}\ne X \in \mathcal{T}$, we denote by $\fm(X)$ the finite family of finite rooted trees obtained from $X$ by deleting the root and all incident edges, and we call {\bf width} of $X$ to the number of elements in $\fm(X)$. The {\bf height} of $X$ is defined recursively as the maximum height of the elements in $\fm(X)$   plus one, with $\hei(0_{\mathcal{T}})=0$. Finally, we denote by $\exp(X$) the tree  obtained from $X$ by adding a new vertex which is declared to be the root of $\exp(X)$, and a new edge joining it to the root of $X$.

Let $X, Y\in \mathcal T$.  The sum of $X$ and $Y$ consists of identifying their roots, and declare it to be the root of $X+Y$. With this operation $\mathcal{T}$ is an additive monoid with neutral element $0_{\mathcal{T}}$.
  The product of $X$ and $Y$ consists of adding pairwise the elements of $\fm(X)$  with the elements of $\fm(Y)$, and then connecting all the resulting finite rooted trees by adding a new vertex (the root of $X\cdot Y$) with incident edges to their roots. In other words,
  $$
   X\cdot Y = \sum_{\substack{X'\in ~\fm(X) \\ Y'\in ~\fm(Y)}} \exp(X'+Y').
  $$
  With this operation, $\mathcal{T}$ is a commutative multiplicative monoid with identity element $1_{\mathcal{T}} = \exp(0_{\mathcal{T}})$, the one-edge rooted tree.    
   The rational map is given by $$X^{\diamond}=\exp^2(X).$$  
 The $U$-semiring structure will be the trivial one, with $uX = Xu = X$ for every $u\in U$.

If $\mathcal{T}_n$ denotes the {subset of $\mathcal T$ consisting of all elements} with at most $n$ edges, the following defines a well-order in $\mathcal{T}$ (\cite{DHS} Lemma 3.3):
\begin{itemize}
 \item[-] $0_{\mathcal{T}}$ is the least element of $\mathcal{T}$. {Set $\mathcal T_0 = \{0_{\mathcal T}\}$}.
 \item[-] Suppose $n\geq 1$ and that $\mathcal{T}_{n-1}$ is already ordered. Take $X,Y \in \mathcal{T}_n\backslash \{0_{\mathcal{T}}\}$. Let log($X$) denote the largest element of {$\mathcal{T}_{n-1}$ in $\fm(X)$}, so $\exp(\log(X$)) is a summand of $X$, and denote its complement by $X-\exp(\log(X))\subseteq \mathcal{T}_{n-1}$. We say that $X>Y$ if either log($X$)$>$log($Y$) or log($X$) = log($Y$) and $X-\exp(\log(X)) > Y- \exp(\log(Y))$.
\end{itemize}
In particular, if $\hei(X)>\hei(Y)$, then $X>Y$, and essentially, what we do to compare two different rooted trees $X$ and $Y$ is to recursively compare the largest element in $\fm(X)$  with the largest element in $\fm(Y)$; if they are equal, we move on to the next largest element in each of the families; and we continue until we can declare $X>Y$ or $Y>X$.

This order satisfies, among many others (cf. \cite[Remark 5.18]{Sthesis}), the following properties.
\begin{lem} \label{>}
Let $X, Y, X', Y' \in \mathcal{T}$:
\begin{itemize}
 \item[(i)] If $X'\leq X$ and $Y'\leq Y$, then $X'+Y'\leq X+Y$, and equality holds if and only if $X'= X$ and $Y' = Y$. In particular, if $Y\neq 0_{\mathcal{T}}$, then $X<X+Y$.
  \item[(ii)] If $X'\leq X$ and $Y'\leq Y$, then $X'\cdot Y' \leq X\cdot Y$.  If $X',Y' \neq 0_{\mathcal T}$, then equality holds if and only if $X' = X$ and $Y' = Y$. In particular, if $X, Y \neq 0_{\mathcal{T}}$, then $X\leq X\cdot Y$ and they are equal if and only if $Y = 1_{\mathcal{T}}$.
  \end{itemize}
\end{lem}

\subsection{The universal rational $U$-semiring}

 Given a multiplicative group $U$, the universal rational $U$-semiring $\Rat (U)$ is constructed inductively as a formal analog of the construction of a division or a $*$-regular closure, starting with the elements of $U$, constructing at each inductive step a bigger rational $U$-semiring by means of sums, products and rational operations $\diamond$ of the object in the previous step, and then taking unions. Before defining $\Rat (U)$, we present some definitions and notation:
\begin{itemize}[leftmargin=.2in]
  \item If $X$ is a set, then the free additive monoid on $X$ is $\mathbb{N}[X]$ and the free additive semigroup on $X$ is $\mathbb{N}[X]\backslash \{0\}$. This way we can consider \textit{formal sums} of elements in $X$. Moreover, when $X$ is a multiplicative monoid with $U$-biset structure, these have a $U$-semiring structure naturally inherited from the one on $X$.

  \item If $X$ is a $U$-biset, then $X^{\times_U^n}$ is the set of equivalence classes of  words in $X$ of length $n$ with respect to the relation generated by 
  $$x_1\ldots (x_i u)x_{i+1}\ldots x_n\sim x_1\ldots x_i (ux_{i+1})\ldots x_n \textrm{\  for all \ } u\in U, 1\le i\le n-1.
  $$
  $X^{\times_U^n}$ has a natural $U$-biset structure given by $$u(x_1x_2\dots x_n) = (ux_1)x_2\dots x_n \textrm{\ and \ }(x_1x_2\dots x_n)u = x_1x_2\dots (x_nu).$$
The multiplicative free monoid on $X$ over $U$ is defined as
  $$
   U\natural X = \bigcup_{n=0}^\infty X^{\times_U^n} 
  $$
  where we understand $X^{\times_U^0}=U$. This object is again a $U$-biset with the natural structure. In this manner we can consider \textit{formal products} of elements of $X$. In addition, observe that  $\mathbb{N}[U\natural X]$ has  a $U$-semiring structure, where the multiplication is naturally inherited from the one on $U\natural X$. 
  
  \item If $X$ is a $U$-biset, then $X^{\diamond}$ denotes a disjoint copy of $X$ together with a bijective map $X\rightarrow X^{\diamond}$, $x\mapsto x^{\diamond}$, and a $U$-biset structure given by $ux^{\diamond}v:=(v^{-1}xu^{-1})^{\diamond}$. This will allow us to construct a \textit{formal rational operation} in $X$.
\end{itemize}

The  \textbf{universal rational $U$-semiring} is defined as follows (compare with the definition of $*$-regular closure).

 \begin{itemize}[leftmargin=.2in]
    \item[-] Consider the $U$-semiring $\mathbb{N}[U]\backslash \{0\}$, and set $X_0:=\emptyset$, $X_1:=(\mathbb{N}[U]\backslash \{0\})^{\diamond}$. Trivially $X_0$ is a $U$-sub-biset of $X_1$.
     \item[-] Suppose $n\geq 1$, $X_n$ is a $U$-biset and $X_{n-1}$ a $U$-sub-biset of $X_n$. Consider the $U$-semiring $\mathbb{N}[U\natural X_n]$ and the $U$-sub-biset $\mathbb{N}[U\natural X_n]\backslash \mathbb{N}[U\natural X_{n-1}]$, and define
     $$
     X_{n+1}:= (\mathbb{N}[U\natural X_n]\backslash \mathbb{N}[U\natural X_{n-1}])^{\diamond} \cup X_n.
     $$
     \item[-] Then, $X=\bigcup X_n$ is a $U$-biset and the universal rational $U$-semiring $\Rat (U)$ is defined as 
 $$
 \Rat (U):= \mathbb{N}[U\natural X] \backslash \{0\}.
 $$
 \end{itemize}
  
 Its rational map $\diamond$ can be shown to carry $\mathbb{N}[U]\backslash \{0\}$ to $X_1$, $\mathbb{N}[U\natural X_n]\backslash \mathbb{N}[U\natural X_{n-1}]$ to $X_{n+1}\backslash X_n$ for $n\geq 1$, $\mathbb{N}[U\natural X_n]\backslash \{0\}$ to $X_{n+1}$ for $n\geq 0$, and $\Rat (U)$ to $X$.

In order to understand the resulting object of this definition, it is important to notice that starting from $U$, at each step we just allow formal sums and products of the elements in the previous step, and define a formal rational operation on the new elements obtained this way.

The universality of $\Rat (U)$ comes from the following property (\cite{DHS} Lemma 4.7).
\begin{lem}\label{Phimap}
If $U$ is a multiplicative group and $R$ a rational $U$-semiring , then there exists a unique morphism of rational $U$-semirings $\Phi:\Rat (U) \rightarrow R$.
\end{lem}

$\Phi$ extends to a morphism of $U$-semirings $\Phi: \Rat (U)\cup \{0\} \rightarrow R$ whenever $R$ has a zero element with $\{0_R\}\cdot R = R\cdot \{0_R\} = \{0_R\}$. 
In particular, we obtain a morphism of $U$-semirings:
$$ \Tree: \Rat (U)\cup\{0\} \rightarrow \mathcal{T}.$$
We call $\Tree (\alpha)$ the \textbf{complexity} of $\alpha$.

As shown in \cite[Example 5.35]{Sthesis}, if $V$ is a subgroup of $U$ then the universal morphism
$$
 \Psi: \Rat(V) \rightarrow \Rat(U)
$$ 
is naturally injective at every inductive step,     
 so we can think that $\Rat(V)\subseteq \Rat(U)$. Moreover, if $\alpha\in \Rat(V)\subseteq \Rat(U)$, then $\Tree(\alpha)$ does not depend on whether we consider $\alpha$ as an element of $\Rat(V)$ or $\Rat(U)$, because the universal property of $\Rat(V)$ implies that $\Tree_V = \Tree_U\circ \Psi$. 
 
{The following lemma collects some of the properties of the complexity. In order to state it properly, add to $\mathcal T$ a new least element $\{-\infty\}$ and turn $\mathcal T \cup \{-\infty\}$ into a semiring by setting $\mathcal T + \{-\infty\} = \mathcal T \cdot \{-\infty\} = \{-\infty\} \cdot \mathcal T = \{-\infty\}$. Now define $\log (0_{\mathcal T}) = -\infty$ and $\log (-\infty) = -\infty$.} 
\begin{lem} \label{tree}
If $\alpha,\beta \in \Rat (U)\cup \{0\}$, then the following holds.
\begin{itemize}
 \item[(i)] $\Tree (\alpha) = 0_{\mathcal{T}}$ if and only if $\alpha = 0$.
 \item[(ii)] $\Tree (\alpha) = 1_{\mathcal{T}}$ if and only if $\alpha \in U$.
 \item[(iii)] $\Tree (\alpha + \beta) = \Tree (\alpha) + \Tree (\beta)$.
 \item[(iv)] $\Tree (\alpha)\leq \Tree (\alpha + \beta)$ and they are equal if and only if $\beta = 0$.
 \item[(v)] $\Tree (\alpha \beta) = \Tree (\alpha)\Tree (\beta)$.
 \item[(vi)] If $\alpha,\beta \neq 0$, then $\Tree (\alpha)\leq \Tree (\alpha \beta)$ and they are equal if and only if $\beta \in U$.
 \item[(vii)] $\log \Tree (\alpha + \beta) = \max\{\log \Tree (\alpha), \log \Tree (\beta)\}$.
 \item[(viii)] $\log \Tree (\alpha \beta) = \log \Tree (\alpha) + \log \Tree (\beta)$.
 \item[(ix)] $\log^2 \Tree (\alpha + \beta) = \max\{\log^2 \Tree (\alpha), \log^2 \Tree (\beta)\}$.
 \item[(x)] $\log^2 \Tree (\alpha \beta) \leq \max\{\log^2 \Tree (\alpha), \log^2 \Tree (\beta)\}$ and they are equal if and only if $\alpha,\beta\neq 0$.
 \item[(xi)] If $\alpha\ne 0$, $\Tree (\alpha^{\diamond}) = \exp^2 \Tree (\alpha)$.
 \item[(xii)] If $\alpha\ne 0$, $\Tree (\alpha^{\diamond})>\log^2  \Tree (\alpha^{\diamond}) = \Tree (\alpha)$.
 \item[(xiii)] If $\alpha \in U\natural X$, then $\wid(\Tree (\alpha))$ = 1.
\end{itemize}
\end{lem} 
\begin{proof}
Properties $(i)$-$(x)$ and $(xii)$ can be found in \cite[Lemma 4.9]{DHS}, property $(xi)$ holds because over non-zero elements $\Tree$ is a morphism of rational $U$-semirings, and property $(xiii)$ is observed in \cite{Sthesis}, page 112.
\end{proof}

A crucial step for the inductive method used in \cite{DHS} is the existence, for every element $\alpha$ in $\Rat (U)$, of a subgroup $ \source (\alpha)$ of $U$ with the following properties (\cite{DHS}, Lemma 5.2, Lemma 5.4, Theorem 5.7).
\begin{teo} \label{source}
If $\alpha \in \Rat (U)$, then the following holds.
\begin{itemize}
\item[(i)] $ \source (\alpha)$ is finitely generated and $\alpha\in \Rat ( \source (\alpha))\cdot U$. The elements satisfying  $\alpha\in \Rat ( \source (\alpha))$ are called \textbf{primitive}.
\item[(ii)] The set $P$ of all primitive elements satisfies $PU=UP=\Rat (U)$. 
 {If $\alpha = \alpha^{\prime}u$ with $\alpha^{\prime}\in P$ and $u\in U$, then $\source(\alpha)= \source(\alpha^{\prime})$.}
\item[(iii)] If $V$ is a subgroup of $U$ such that $\alpha\in \Rat (V)\cdot U$, then $ \source (\alpha)\leq V$.
\end{itemize}
\end{teo}

\subsection{Division  $E*G$-closures}\label{ss:divclo} 

The following example is a modification of \cite[Example 1.43(d)]{Sthesis} in the case we deal with division closures. Let $R$ be a subring of a ring $S$. We will denote by $\D_{R,S}$ the division closure of $R$ in $S$. 
As it happens with the $*$-regular closure, it is easy to see that it can be constructed as follows.
\begin{itemize}
 \item[-] Put $ {Q}_0:=R$.
 \item[-] Suppose $n\geq 1$ and that we have constructed a subring $Q_n$ of $S$. Then $Q_{n+1}$ is the subring of $S$ generated by the elements of $Q_n$ and its inverses (whenever they exist).
 \item[-] $\D_{R,S} = \bigcup_{n=0}^{\infty} Q_n$.
\end{itemize}
Observe that if $S$ is $*$-regular, then $\D_{R,S}$ is contained in the $*$-regular closure of $R$ in $S$. The following lemma is just a straightforward consequence of the definitions.
\begin{lem} \label{divreg}
Let $R$ be a subring of a ring $S$, and let $\D_{R,S}$ denote the division closure of $R$ in $S$. Then
\begin{itemize}
 \item[(1)] If $T$ is a subring of $R$, then $\D_{T,S} = \D_{T,\D_{R,S}} \subseteq \D_{R,S}$.
 \item[(2)] If $\mathcal U$ is a regular subring of $S$ containing $R$, then $\D_{R,S} = \D_{R,\mathcal U}$. 
\end{itemize} 
\end{lem}
\begin{proof}
  The first assertion is clear. To prove the second, take $x\in \mathcal U$ and note that if $x$ is invertible in $S$ then it is a non-zero-divisor in $\mathcal U$. Since $\mathcal U$ is regular, this means that $x$ is invertible in $\mathcal U$.
\end{proof}

Now, let $E*G$ be a crossed product of a division ring $E$ with any group $G$, and let $\phi: E*G\to \mathcal A$ be an $E*G$-ring. For each subgroup $H$ of $G$  denote by $\D_{H, \mathcal A}$ the division closure of $\phi(E*H)$ in $\mathcal A$. Then we can define an $E^\times H$-rational structure on $\D_{H,\mathcal A}$ by putting $a^\diamond=a^{-1}$ if $a$ is invertible in $\D_{H,\mathcal A}$ and $a^\diamond=0$ otherwise.

Therefore, for any $H \leq G$ we can apply Lemma \ref{Phimap} to $\D_{H,\mathcal A}$ to obtain a unique morphism of $E^\times H$-semirings 
$$
  \Phi_{H,\mathcal A}: \Rat (E^{\times}H) \cup \{0\}\rightarrow \mathcal D_{H,\mathcal A}
$$
with $\Phi_{H,\mathcal A}(0)=0$. Reasoning as in \cite[Examples 5.37 and 5.38]{Sthesis}, every $\Phi_{H,\mathcal A}$ is surjective and the restriction of $\Phi_{G,\mathcal A}$ to $\Rat(E^{\times}H)$ is a morphism of $E^{\times}H$-rings whose image is $\D_{\phi(E*H),\D_{G,\mathcal A}}=\D_{H,\mathcal A}$. Therefore, the uniqueness in Lemma \ref{Phimap} implies that the following diagram is commutative. 

\begin{equation} \label{diaRat}
\begin{gathered}
     \xymatrixcolsep{3pc}\xymatrix{ \Rat (E^{\times}H) \ar@{->>}[r]^-{\Phi_{H,\mathcal A}} \ar@{^{(}->}[]+<0ex,-3ex>;[d] & \D_{H,\mathcal A} \ar@{^{(}->}[]+<0ex,-3ex>;[d] \\
     \Rat (E^{\times}G) \ar@{->>}[r]^-{\Phi_{G,\mathcal A}} & \D_{G,\mathcal A}
     }
\end{gathered}
\end{equation}
    
Now we can define the $H$-complexity of any element of $\mathcal D_{H,\mathcal A}$. Let $a\in \mathcal D_{H,\mathcal A}$. Then we put
$$
\Tree _H(a)= \min \{\Tree (\alpha): \alpha \in \Rat (E^{\times}H)\cup \{0\}, \Phi(\alpha) = a\}.
$$
This notion is always defined since the rooted trees are  well-ordered. Notice also that $\Tree_G (a)\leq \Tree _H(a)$, for all $a\in \mathcal D_{H,\mathcal A}$. We will say that $\alpha$ {\bf realizes} the $H$-complexity of $a\in  \mathcal D_{H,\mathcal A}$ if $\Phi(\alpha)=a$ and $\Tree(\alpha)=\Tree_H(a)$.

As an important remark, suppose that $(\mathcal A_1, \phi_1)$ and $(\mathcal A_2,\phi_2)$ are two $E*H$-rings such that $\D_{H,\mathcal A_1}$ and $\D_{H,\mathcal A_2}$ are isomorphic $E*H$-rings. If we denote this isomorphism by $\varphi$, then for every $a\in \D_{H,\mathcal A_1}$ we have that $\Tree_H(a) = \Tree_H(\varphi(a))$. Indeed, since $\varphi$ is an isomorphism of $E*H$-rings, it preserves the $E^{\times}H$-rational structure, and so $\varphi \circ \Phi_{H,\mathcal A_1}$ is a morphism of $E^{\times}H$-semirings. Uniqueness in Lemma \ref{Phimap} implies that $\varphi \circ \Phi_{H,\mathcal A_1} = \Phi_{H,\mathcal A_2}$, and the claim follows.

We finish the section with some comments in the case we are really interested in. Suppose that $G$ is a locally indicable group, $H$ a finitely generated subgroup, and  $H=N\rtimes C$, where $C$ is infinite cyclic. Let $t$ be an element of $E^{\times}H$ such that its image under the map $E^{\times}H\rightarrow   H$ generates $C$, and let $\tau: E^{\times}N \rightarrow E^{\times}N$ denote the automorphism given by left conjugation by $t$. Then $\tau$ can be extended, respectively, to an automorphism of $\D_{N,\mathcal A}$ and to an automorphism of the semiring $\Rat(E^{\times}N)$. Both extensions will also be denoted by $\tau$. We will write $\Rat (E^{\times}N)<t>$ to refer to the multiplicative submonoid of $\Rat (E^{\times}H)$ whose elements are of the form  $\alpha t^n$, for $\alpha \in \Rat (E^{\times}N)$ and $n\in \mathbb{Z}$, and with $(\alpha t^n)\cdot(\beta t^m) = \alpha \tau^n(\beta) t^{n+m}$. Observe in particular that the following holds.
  \begin{itemize}
    \item[-] If $\alpha \in \Rat (E^{\times}N)$, then $t^n \alpha = \tau^n(\alpha)t^n\in \Rat (E^{\times}N)<t>$. 
      
    \item[-] If $\alpha, \beta \in \Rat (E^{\times}N)<t>$, then $\alpha\beta \in \Rat (E^{\times}N)<t>$.
    
    \item[-] If $\alpha, \beta \in \Rat (E^{\times}N)t^n$, then $\alpha+\beta \in \Rat (E^{\times}N)t^n$.
  \end{itemize}

\subsection{Epic $*$-regular $K[G]$-rings as $K^{\times}G$-semirings}
The next example of rational semiring was central in a previous version of our proof of the Atiyah conjecture for locally indicable groups \cite{JL18} and, although it will not play a role in the proof shown in this paper, we think it can be relevant for future reference.

 For the rest of the section let $G$ be a group, $K$ a subfield of the complex numbers $\mathbb{C}$ closed under complex conjugation, and endow the group ring $K[G]$ with the usual proper  involution $*$, which is defined by $(\lambda g)^* = \bar{\lambda}g^{-1}$ and extended by linearity.

 Suppose that we have a $*$-regular $K[G]$-ring $\mathcal U$ with an epic $*$-homomorphism $\varphi: K[G]\rightarrow \mathcal U$. The following lemma shows that $\U$ is a rational $K^{\times}G$-semiring.

\begin{lem} \label{Rsemiring}
If $K[G]$ is a $*$-subring of a $*$-regular ring $\mathcal U$ such that $K[G]  \hookrightarrow \mathcal U$ is epic, then $\mathcal U$ is a rational $K^{\times}G$-semiring with rational operation given by taking relative inverses. 
\end{lem}

\begin{proof}
  We have to show that, for every $u,v\in K^{\times}G$ and $x\in \mathcal U$, the equality $(uxv)^{[-1]}=v^{-1}x^{[-1]}u^{-1}$ holds. Observe first that $K\subseteq Z(\mathcal U)$ by Lemma \ref{center}. Put $e= \RP(x), f = \LP(x)$. Then, by the previous observation and the definition of $\LP(x)$, we have that $\LP(uxv) = ufu^{-1}$. Indeed, if $u=\lambda g$ for some $\lambda\neq 0$ and $g\in G$,
\begin{itemize} 
  \item[-] $ufu^{-1}$ is idempotent, and 
  $$
   (ufu^{-1})^*= \bar{\lambda}^{-1}gf\bar{\lambda}g^{-1}= \bar{\lambda}gf\bar{\lambda}^{-1}g^{-1} = ufu^{-1}
  $$
   so it is a projection.
  \item[-] $uxvR = uxR = ufR = ufu^{-1}R$.
\end{itemize}
Similarly we have that $\RP(uxv) = v^{-1}ev$. To conclude the result, just observe that $(uxv)(v^{-1}x^{[-1]}u^{-1}) = \LP(uxv)$, $v^{-1}x^{[-1]}u^{-1}uxv = \RP(uxv)$, and $v^{-1}x^{[-1]}u^{-1} = \RP(uxv)v^{-1}x^{[-1]}u^{-1}\LP(uxv)$.
\end{proof}

As a consequence, in the previous setting we obtain a morphism of rational $K^{\times}G$-semirings 
$$
  \Phi_G: \Rat (K^{\times}G) \rightarrow \mathcal U.
$$
 
Again, as in the case of division closures, for any $H\leq G$ we can think of $\Phi_H$ as the restriction of $\Phi_G$ to $\Rat(K^{\times}H)$ and, as a mere rewriting of \cite[Example 5.37 and 5.38]{Sthesis}, we obtain that the image of $\Phi_H$ is $\mathcal U_{H}$, the $*$-regular closure of $K[H]$ inside $\mathcal U$.  In particular $\Phi_G$ is surjective. 
Therefore, we can understand that for every $N\leq H \leq G$, the following diagram is commutative
    $$
     \xymatrix{ \Rat (K^{\times}N) \ar@{->>}[r]^-{\Phi_N} \ar@{^{(}->}[d] & \mathcal U_{N} \ar@{^{(}->}[d] \\
     \Rat (K^{\times}H) \ar@{->>}[r]^-{\Phi_H} & \mathcal U _{H}
     }
    $$
Using this, we can push forward to $\mathcal U_{H}$ the notion of $H$-complexity in the same way we defined it before.

\section{A key auxiliary result and first applications}\label{mainsection}
 In this section we are going to present a key result for the proof of the Atiyah and the L\"uck approximation conjectures, and one of its immediate applications. The structure of the proof mimics the steps of the proof of Hughes theorem presented in \cite{Sthesis}. In what follows, for any $E*G$-ring $(A,\phi)$ and for any subgroup $H\leq G$, $\mathcal D_{H,A}$ will denote the division closure of $\phi(E*H)$ in $A$. If we are working with an indexed family of $E*G$-rings, say $\{(A_i,\phi_i)\}$, then we will replace the previous notation by $\mathcal{D}_{H,i}$ for the sake of readability.

\subsection{A key auxiliary result}
Let $H$ be a finitely generated group and let $N$ be a normal subgroup of $H$ such that $H/N\cong \Z$.  Let $E$ be a division algebra, let $E*H$ be a crossed product of $E$ and $H$ and take  $t$ a preimage in $E^{\times}H$ of a generator of the quotient  $E^\times H/E^\times N$. Denote by $\tau$ the automorphism of $E*N$ induced by the left conjugation by  $t$, i.e., $ta=\tau(a)t$ for $a\in E*N$.   Then $E*H$ is isomorphic to the skew Laurent polynomial ring $(E*N)[t^{\pm 1}; \tau]$.  Assume that we have the following
\begin{itemize}
  \item[(i)] A von Neumann regular $E*N$-ring ($\mathcal A$, $\phi$).
  \item[(ii)] An automorphism of $\mathcal A$, also denoted by $\tau$, such that $\tau\circ\phi = \phi\circ\tau$.
  \item[(iii)] A ring $\mathcal{P}$ such that $\mathcal A((t; \tau)) \subseteq \mathcal{P}$.
\end{itemize} 
Then $\phi$ can be extended to a homomorphism 
$$\phi :E*H\cong (E*N)[t^{\pm 1};\tau]\to \mathcal A((t;\tau))\subseteq \mathcal P,$$
and so we can consider $\mathcal D_{N,\mathcal P}$ and $\mathcal D_{H,\mathcal P}$. As it was explained in Subsection \ref{ss:divclo}, we can define a notion of $H$-complexity on $\D_{H,\mathcal P}$ by means of the corresponding map $\Phi:\Rat(E^\times H)\to \mathcal D_{H,\mathcal P}$. 

There are two important things to notice before stating our key result. Firstly, since $\mathcal A$ is regular, Lemma \ref{divreg} states that $\D_{N,\mathcal P}$ equals $\D_{N, \mathcal A}$. Secondly, it follows from the condition $(ii)$ and the construction of a division closure that the restriction of $\tau$ to $\D_{N,\mathcal A}$ is an automorphism of $\D_{N,\mathcal A}$. Indeed, let $D_{N,\mathcal A} = \bigcup Q_i$ as at the beginning of this section, with $Q_0 = \phi(E*N)$. Condition $(ii)$ assures that $\tau(Q_0) = Q_0$. Now assume $i\geq 1$ and $\tau(Q_{i})= Q_i$. If $x\in Q_i$ is invertible, then $\tau(x^{-1}) = \tau(x)^{-1}\in Q_{i+1}$ by the induction hypothesis. From here $\tau(Q_{i+1}) \subseteq Q_{i+1}$, and since we can play the same game with $\tau^{-1}$, we are done. Therefore, it makes sense to consider $\D_{N,\mathcal P}((t;\tau))$, a subring of $\mathcal A((t;\tau))$.

\begin{pro} \label{key}
{Assume the previous notation.} Let $a\in \mathcal D_{H,\mathcal P}$ and assume that for every $0\ne c \in\mathcal D_{H,\mathcal P}$ such that $\Tree_H(c)<\Tree_H(a)$, $c$ is invertible in  $\mathcal D_{H,\mathcal P}$. Then for every $b\in \D_{H,\mathcal P}$ such that $\Tree_H(b)\le \Tree_H(a)$ the following holds

\begin{enumerate}
\item $b$ belongs to $\mathcal D_{N,\mathcal P}((t;\tau))$ and 

\item  if  $0\ne b = \sum b_k$ with $b_k\in \mathcal D_{N,\mathcal P}t^k$, then $$\Tree _{H}(b_k)\leq \Tree _H(b)$$ for all $k$, and the equality holds for some $n$ if and only if  $b =b_n \in \mathcal D_{N,\mathcal P}t^n$ and $$\bigg\lbrace \beta  \in \Rat (E^{\times}H)\cup \{0\}: \Phi(\beta )=b \textrm{~and~}\Tree (\beta )= \Tree _H(b)\bigg\rbrace\subseteq \Rat (E^{\times}N) t^n.$$

\end{enumerate}

\end{pro}
\begin{proof}
 For $b=0$ it is clear, so let $b\ne 0$. If $\Tree _H(b)=1_{\mathcal{T}}$ and $\beta$ realizes the $H$-complexity of $b$, then we have that $\beta \in E^{\times}H = E^{\times}N< t >$ and $b\in \phi(E^{\times}H) = \phi(E^{\times}N)<t>$, so the result holds. Suppose now that $\Tree _H(b)> 1_{\mathcal{T}}$ and that the result holds for every element $c\in \mathcal \D_{H,\PP}$ with $\Tree _H(c)<\Tree _H(b)$. Fix an arbitrary element $\beta\in \Rat (E^{\times}H)$ realizing the $H$-complexity of $b$. We are going to divide $\Rat (E^{\times}H)$ in four disjoint subsets
  $$
  \begin{matrix}
     U=E^{\times}H & & X & & U\natural X \backslash (X\cup U) & & \mathbb{N}[U\natural X]\backslash (U\natural X \cup \{0\}).
  \end{matrix}
  $$
As far as we are assuming $\Tree_H (b)>1_{\mathcal{T}}$, we know that $\beta\notin U$, so we have three possibilities left:

{\bf Case 1.} If $\beta\in U\natural X \backslash (X\cup U)$, then there exist $\gamma,\delta \in U\natural X\backslash U$ such that $\beta = \gamma\delta$.  By Lemma \ref{tree}(vi), 
  $$
  \Tree (\gamma),\Tree (\delta)<\Tree (\beta).
  $$ 
Setting $c = \Phi(\gamma), d = \Phi(\delta)$, we obtain a decomposition $b = cd$. We claim that $\gamma$ realizes the $H$-complexity of $c$, i.e., $\Tree _H(c) = \Tree (\gamma)$. Otherwise, there would exist $\gamma'$ with $\Phi(\gamma') = c$ satisfying $\Tree (\gamma') < \Tree (\gamma)$, from where using Lemma \ref{tree}(v) and Lemma \ref{>}(ii)
$$
  \begin{matrix*}[l]
  \Tree (\gamma'\delta) &= \Tree (\gamma')\Tree (\delta) < \Tree (\gamma)\Tree (\delta) \\[6pt]
  &= \Tree (\gamma\delta) = \Tree (\beta).
  \end{matrix*}
$$
Since $\Phi(\gamma'\delta) = b$, this contradicts the minimality of $\beta$. Similarly, $\delta$ realizes the $H$-complexity of $d$, and therefore we have found a decomposition $b = cd$ with $\Tree _H(b)>\Tree _H(c),\Tree _H(d)$. Now, by the induction hypothesis, we can write $c = \sum c_n$, $d = \sum d_n$ with $\Tree _H(c_n)\leq \Tree _H(c)$ and $\Tree _H(d_n)\le \Tree _H(d)$. Hence, we have an expression $b = \sum b_n$ with $b_n = \sum c_md_{n-m}$. Let $\beta_n$, $\gamma_n$, $\delta_n$ be elements in $\Rat (E^{\times}H)\cup \{0\}$ such that $\Tree (\beta_n) = \Tree _H(b_n)$, $\Tree (\gamma_n) = \Tree _H(c_n)$, $\Tree (\delta_n)= \Tree _H(d_n)$, for all $n$. From the previous expression we obtain 
$$
\Tree (\beta_n) \leq \sum \Tree (\gamma_m)\Tree (\delta_{n-m}).
$$ 
Therefore, using Lemma \ref{tree},
$$
\begin{matrix*}[l]
  \log \Tree _H(b_n) &\leq \log \left(\sum \Tree (\gamma_m)\Tree (\delta_{n-m})\right) \\[6pt]
  &= \max\left\lbrace \log (\Tree (\gamma_m)\Tree (\delta_{n-m})) \right\rbrace \\[6pt]
  & = \max\left\lbrace \log \Tree (\gamma_m)+ \log \Tree (\delta_{n-m}) \right\rbrace \\[6pt]
  & \leq \log \Tree (\gamma) + \log \Tree (\delta) = \log \Tree (\gamma\delta) \\[6pt]
  & = \log \Tree (\beta)= \log \Tree _H(b).
\end{matrix*}
$$
If $\log \Tree _H(b_n)<\log \Tree _H(b)$ for all $n$, then $\Tree _H(b_n)<\Tree _H(b)$ for all $n$. If there exists $n$ such that the equality holds, then by the previous expression there exists some integer $m$ such that
$$
\begin{matrix}
\log \Tree (\gamma_m) = \log \Tree (\gamma) & & & \log \Tree (\delta_{n-m}) = \log \Tree (\delta).
\end{matrix}
$$ 
Since $\gamma,\delta \in U\natural X$, Lemma \ref{tree}(xiii) tells us that $\wid(\gamma)$ = $\wid(\delta)$ = 1, and consequently $\Tree (\gamma_m)\geq \Tree (\gamma)$ and $\Tree (\delta_{n-m}) \geq \Tree (\delta)$. Therefore, we have equality, and the induction hypothesis says that there exist $c_m', d_{n-m}' \in \D_{N,\PP}$, $\gamma', \delta'\in \Rat (E^{\times}N)$ such that $c = c_m = c_m't^m$, $d = d_{n-m}'t^{n-m}$, $\gamma = \gamma't^m$, $\delta = \delta' t^{n-m}$, and so 
$$
 \begin{matrix*}[l]
    &b = cd = c_m'\tau^{m}(d_{n-m}')t^n \in \D_{N,\PP}t^n \\[6pt]
    &\beta = \gamma\delta = \gamma'\tau^m(\delta') t^n \in \Rat (E^{\times}N)t^n.
 \end{matrix*}
$$

{\bf Case 2.}  If $\beta \in \mathbb{N}[U\natural X]\backslash (U\natural X \cup \{0\})$, then there exist $\gamma,\delta \in \mathbb{N}[U\natural X]\backslash \{0\}$ such that $\beta = \gamma+\delta$.  By Lemma \ref{tree}(iv), 
  $$
  \Tree (\gamma),\Tree (\delta)<\Tree (\beta).
  $$ 
Setting $c = \Phi(\gamma), d = \Phi(\delta)$, we obtain a decomposition $b = c+d$. We claim that $\gamma$ realizes the $H$-complexity of $c$, i.e., $\Tree _H(c) = \Tree (\gamma)$. Otherwise, there would exist $\gamma'$ with $\Phi(\gamma') = c$ satisfying $\Tree (\gamma') < \Tree (\gamma)$, from where using Lemma \ref{tree}(iii) and Lemma \ref{>}(i)
$$
  \begin{matrix*}[l]
  \Tree (\gamma'+\delta) &= \Tree (\gamma')+\Tree (\delta) < \Tree (\gamma)+\Tree (\delta) \\[6pt]
  &= \Tree (\gamma+\delta) = \Tree (\beta).
  \end{matrix*}
$$
Since $\Phi(\gamma'+\delta) = b$, this contradicts the minimality of $\beta$. Similarly, $\delta$ realizes the $H$-complexity of $d$, and therefore we have found a decomposition $b = c+d$ with $\Tree _H(b)>\Tree _H(c),\Tree _H(d)$. Now, by the induction hypothesis, we can write $c = \sum c_n$, $d = \sum d_n$ with $\Tree _H(c_n)\leq \Tree _H(c)$ and $\Tree _H(d_n)\le \Tree _H(d)$. Hence, we have an expression $b = \sum b_n$ with $b_n = c_n+d_n$. Let $\beta_n$, $\gamma_n$, $\delta_n$ be elements in $\Rat (E^{\times}H)\cup \{0\}$ such that $\Tree (\beta_n) = \Tree _H(b_n)$, $\Tree (\gamma_n) = \Tree _H(c_n)$, $\Tree (\delta_n)= \Tree _H(d_n)$, for all $n$. From the previous expression we obtain that, for any $n$,  
$$
\begin{matrix*}[l]
  \Tree _H(b_n) &= \Tree (\beta_n) \leq \Tree (\gamma_n)+ \Tree (\delta_{n}) \\[6pt]

  & \leq \Tree (\gamma) + \Tree(\delta) = \Tree _H(b).
\end{matrix*}
$$
If there exists $n$ such that the equality holds, then 
$$
\begin{matrix}
\Tree (\gamma_n) = \Tree (\gamma) & & & \Tree (\delta_n) = \Tree (\delta)
\end{matrix}
$$ 
and by induction there exist $c_n', d_n' \in \D_{N,\PP}$, $\gamma', \delta'\in \Rat (E^{\times}N)$ such that $c = c_n = c_n't^n$, $d = d_{n}'t^{n}$, $\gamma = \gamma't^n$, $\delta = \delta't^{n}$. Hence, 
$$
 \begin{matrix*}[l]
    &b = c+d = (c_n'+d_n')t^n \in \D_{N,\PP}t^n \\[6pt]
    &\beta = \gamma+\delta = (\gamma'+\delta')t^n\in \Rat (E^{\times}N)t^n.
 \end{matrix*}
$$ 

{\bf Case 3.} If $\beta\in X$, then there exists $\gamma\in \mathbb{N}[U\natural X]\backslash \{0\}$ such that $\beta = \gamma^{\diamond}$. By Lemma \ref{tree}(xii), $\Tree (\gamma)<\Tree (\beta)$, and setting $c = \Phi(\gamma)\in \D_{H,\PP}$ we obtain that $$
b = \Phi(\gamma^{\diamond})= c^{\diamond}
$$
and since $b$ is non-zero, $b = c^{-1}$.
We claim that $\gamma$ realizes the $H$-complexity of $c$, i.e., $\Tree _H(c)=\Tree (\gamma)$. Otherwise, there would exist $\gamma'$ with $\Phi(\gamma')=c$ satisfying $\Tree (\gamma')<\Tree (\gamma)$, from where using Lemma \ref{tree}(xi)
$$
 \Tree ((\gamma')^{\diamond})<\Tree (\gamma^{\diamond}) = \Tree (\beta).
$$
Since $\Phi(\gamma') = c^{-1} = b$, this contradicts the minimality of $\beta$. Hence, $b = c^{-1}$ with $\Tree _H(c)<\Tree _H(b)$. Now, by the induction hypothesis, we can write $c = \sum c_n$ with $\Tree _H(c_n) \leq \Tree _H(c)$. It is important to notice also that 
$$
   \Tree _H(c_n) \leq \Tree _H(c)<\Tree _H(b)
             \leq \Tree _{H}(a)$$
Thus, all non-zero $c_n$ are invertible in $\D_{H,\PP}$. Therefore, $c_nt^{-n}\in \D_{N,\PP}$ is invertible in $\D_{H,\PP}$, and hence in $\D_{N,\PP}$, and so $c$, which is invertible in $\D_{H,\PP}$ with inverse $b$, is also invertible as an element of $\D_{N,\PP}((t;\tau))$.  Thus,   
we can express $b$ as a Laurent series $b = \sum b_n$ by taking the inverse of $\sum c_n$. 

Let $k=\min\{n:c_n\ne 0\}$. Then $b_n$ can be expressed using sums and products of elements $c_k^{-1}$ and ${- c_m}$, for $m \in C_n=  \{k+1,\dots,2k+n\}$. Let $\beta_n, \gamma_n,\in \Rat (E^{\times}H)\cup \{0\}$ be such that $\Tree (\beta_n) = \Tree _H(b_n)$, $\Tree (\gamma_n) = \Tree _H({- c_n})$, for all $n$. By Lemma \ref{tree}(xii),

\begin{equation}
\label{eq1}
 \log^2\Tree  (\gamma_k^{\diamond}) = \Tree  (\gamma_k)\leq \Tree  (\gamma) = \log^2 \Tree (\gamma^{\diamond})                           = \log^2 \Tree  (\beta), 
  \end{equation} and observe that we also have
 
\begin{equation}\label{eq2}
\log^2\Tree (\gamma_m)  <  \Tree (\gamma_m) \leq \Tree (\gamma) 
      = \log^2\Tree (\beta).
  \end{equation} 
Therefore, using Lemma \ref{tree} (ix) and (x),
$$
\log^2\Tree (\beta_n) \leq \max \left \lbrace \log^2\Tree (\gamma_k^{\diamond}), \displaystyle{\max_{m\in C_n}} \left\lbrace \log^2\Tree (\gamma_m) \right\rbrace \right\rbrace   \leq \log^2\Tree (\beta).
  $$ 
If for every $n$, $\log^2\Tree (\beta_n) < \log^2\Tree (\beta)$, then we conclude that for every $n$, 
$\Tree (\beta_n)<\Tree (\beta)$. If equality holds for some $n$, then since the inequality in (\ref{eq2}) is strict, we obtain from (\ref{eq1})
 that $\Tree (\gamma_k) = \Tree (\gamma)$. By induction, there exist $c_k'\in \D_{N,\PP}$, $\gamma'\in \Rat (E^{\times}N)$ such that $$c = c_k = c_k't^k,\ \gamma = \gamma't^k, $$ and so
$$
 \begin{matrix*}[l]
    &b = c^{-1} = t^{-k}(c_k')^{-1} = \tau^{-k}(c_k')t^{-k} \in \D_{N,\PP}t^{-k} \\[6pt]
    &\beta = \gamma^{\diamond} {=} t^{-k}(\gamma')^{\diamond} = \tau^{-k}((\gamma')^{\diamond})t^{-k}\in \Rat (E^{\times}N)t^{-k}.
 \end{matrix*}
$$
This finishes the proof.
 \end{proof}
 
\subsection{The uniqueness of Hughes-free epic division rings for locally indicable groups}
 In this subsection we give an example of the use of Proposition \ref{key}, presenting  an alternative argument for the last part of the proof from \cite{DHS} of the uniqueness of Hughes-free epic division ring for locally indicable groups. 
 \begin{teo} Let $E$ be a division ring, $G$ a locally indicable group and $E*G$ a crossed product of $E$ and  $G$. Let $(\mathcal D_1,\varphi_1)$ and $(\mathcal D_2,\varphi_2)$ be two Hughes-free epic division $E*G$-rings. Then $\mathcal D_1$ and $\mathcal D_2$ are isomorphic as $E*G$-rings.
 
 \end{teo}
 \begin{proof} Set $\mathcal S = \mathcal D_1\times \mathcal D_2$, $\varphi=(\varphi_1,\varphi_2):E*G\to \mathcal S$ and $\mathcal D=\mathcal D_{G,\mathcal S}$. Denote by $\pi_i:\mathcal D\to\mathcal D_i$ ($i=1,2$) the canonical projections. The ring $\mathcal D$ is an $E^\times G$-rational semiring on which we can define the notion of $G$-complexity using the surjective map $\Phi:\Rat (E^\times G)\to \D$. By induction on the $G$-complexity $\Tree_G(a)$ of $a$ we will show that any non-zero element $a\in \mathcal D$ is invertible. This would imply that $\pi_1$ and $\pi_2$ are two $E*G$-isomorphisms,
  and so, $\mathcal D_1$ and $\mathcal D_2$ are isomorphic as $E*G$-rings. 

The base of induction, when $\Tree_G(a)= 1_{\mathcal T}$, is clear, because in this case $a\in \varphi(E^\times G)$. Now assume that $\Tree_G(a)> 1_{\mathcal T}$ and that for every $0\ne b \in\mathcal D$ such that $\Tree_G(b)<\Tree_G(a)$, $b$ is invertible. Let $\alpha\in \Rat (E^\times G)$ realize the $G$-complexity of $a$.
Using  Proposition \ref{source}, we obtain a finitely generated  subgroup $\source (\alpha)$ of $E^{\times}G$, and we can assume without loss of generality that $\alpha$ is primitive, since multiplying by a unit in $E^{\times}G$ does not change the complexity nor the conclusion for $a$. Let $H$ be the image of  $\source (\alpha)$ in $G\cong E^\times G/E^\times$. Observe that $H$ is finitely generated as well, and since $\alpha\in \Rat(\source(\alpha))$, then $a\in \mathcal D_{H,\D} = \D_{H,\mathcal S}$. In particular we obtain the corresponding diagram (\ref{diaRat}), and as a consequence
$$
 \Tree_H(a) = \Tree_G(a) = \Tree(\alpha)
$$

Clearly we can assume that $H\ne \{1\}$, so there exists a normal subgroup $N$ of $H$ such that $H/N\cong \Z$. Take $t\in E^{\times}H$ whose image under the compositions of canonical maps $E^{\times}H \rightarrow  H \rightarrow   H/N$ generates  $H/N$. Set $\mathcal{A} = \D_{N,1}\times \D_{N,2}$ and $\mathcal{B} = \D_{H,1}\times \D_{H,2}$. Inasmuch as $\D_i$ is Hughes-free, we have {an embedding (of $E*H$-rings)} $\D_{H,i} \hookrightarrow \D_{N,i}((t;\tau_i))$, where $\tau_i$ denotes the automorphism of $D_{N,i}$ induced by left conjugation by $t$  in $E*N$. Therefore, $\mathcal{B}$ embeds in $\mathcal P = \mathcal A((t;\tau)) \cong \D_{N,1}((t;\tau_1))\times \D_{N,2}((t;\tau_2))$,  
where $\tau = (\tau_1,\tau_2)$. Since $\mathcal A$ and $\mathcal B$ are regular and the following diagram commutes
$$
    \xymatrix{ E*N \ar@{^{(}->}[]+<5ex,0ex>;[r] \ar@{->}[d]^{\varphi} & E*G  \ar@{->}[d]^{\varphi} & E*H \ar@{->}[d]^{\varphi} \ar@{_{(}->}[]+<-5ex,0ex>;[l] \\
     \mathcal A \ar@{^{(}->}[]+<3ex,0ex>;[r] \ar@{^{(}->}[]+<2ex,-2ex>;[dr] & \mathcal S & \mathcal B \ar@{_{(}->}[]+<-3ex,0ex>;[l] \ar@{_{(}->}[]+<-2ex,-2ex>;[dl] \\
     &\mathcal P&
     }
$$
Lemma \ref{divreg} implies that $\D_{N,\D} = \D_{N,\mathcal S} = \D_{N,\mathcal P}$ and $\D_{H,\D} = \D_{H,\mathcal S} = \D_{H,\mathcal P}$.  Observe that, for every $0\neq b\in \D_{H,\mathcal S}$ with $\Tree_{H}(b)< \Tree_H(a)$, we have 
$$
  \Tree_G(b)\leq \Tree_H(b) < \Tree_H(a) = \Tree_G(a)
$$
and consequently the induction hypothesis implies that $b$ is invertible. Thus, the conditions in Proposition \ref{key} are satisfied, and therefore we obtain that $a$ belongs to $\mathcal D_{N,\mathcal S}((t;\tau))$. Moreover, if we write  $a = \sum a_k$ with $a_k\in \mathcal D_{N,\mathcal S}t^k$, then there are at least two non-zero summands. In the contrary case, if $a = a_n$, then by the same proposition $\alpha\in \Rat(E^\times N) t^n$, from where Proposition \ref{source}(iii) states that $H\le N$, a contradiction.
 
 Hence, $\Tree_H(a_k)<\Tree_H(a)$ for all $k$. In particular, if $n$ is the smallest $k$ such that $a_k$ is non-zero, we deduce as before that the element $a_n\in  \mathcal D_{N,\mathcal S}t^n$ is invertible in $\mathcal D_{H,\mathcal S}$. This implies that $a$ is invertible in $\mathcal D_{N,\mathcal S}((t;\tau))\subseteq \mathcal P$ , and hence in $\D_{H,\mathcal P} = \D_{H,D} \subseteq D$.
  \end{proof}

\section{The Atiyah conjecture for locally indicable groups}
\label{proofs}
\subsection{A generalization of Theorem \ref{main}} This subsection is entirely devoted to stating and proving one of the main theorems in this paper, related to $*$-regular Hughes-free ranks, and its immediate consequence regarding the strong Atiyah conjecture for locally indicable groups. 
\begin{teo}\label{maintheorem}
Let $G$ be a locally indicable group, $K$ a subfield of $\mathbb{C}$ closed under complex conjugation. Let  $\rk$ be a $*$-regular Hughes-free Sylvester rank function on $K[G]$ with epic positive definite $*$-regular envelope $(\mathcal U,\phi)$. Then $\mathcal{U}$ is a division ring.
\end{teo} 
\begin{proof}
Let $\D = \D_{G,\U}$ be the division closure of $\phi(K[G])$ in $\U$, and for any subgroup $H\le G$, denote by $\D_{H,\U}$ and $\U_H$ the division and the $*$-regular closures of $\phi(K[H])$, respectively, in $\U$. Consider the universal morphism of rational $K^{\times}G$-semirings $\Phi: \Rat(K^{\times}G)\rightarrow \D$. We are going to show that $\D$ is a division ring by induction on the $G$-complexity. Since $\phi$ is epic, this will imply that $\U = \D$. 

Consider a non-zero element $a\in \D$. If $\Tree_G(a) = 1_{\mathcal T}$, then $a\in \phi(K^{\times}G)$ is invertible. Now assume that $\Tree_G(a)>1_{\mathcal T}$ and that the result holds for all $0\ne b\in \D$ with $\Tree_G(b)<\Tree_G(a)$. Take $\alpha\in \Rat(K^{\times}G)$ realizing the $G$-complexity of $a$. We can assume that $\alpha$ is primitive because multiplying by a unit in $K^{\times}G$ does not change the complexity nor the conclusions about the invertibility of $a$. Set $H$ to be the image of $\source(\alpha)$ under $K^{\times}G \rightarrow K^{\times}G/K^{\times} = G$. By Proposition \ref{source}, $H$ is finitely generated and $\alpha \in \Rat(K^{\times}H)$, so $a\in \D_{H,\D} = \D_{H,\U}=\D_{H,\U_H}$, this latter equality due to Lemma \ref{divreg}. If $H$ is trivial, then $\D_{H,\U_H} \cong K$ and since $a$ is non-zero, it is invertible. If $H$ is non-trivial, then there exists a normal subgroup $N\trianglelefteq H$ and an element $t\in H$ of infinite order such that $H = N\rtimes_{\tau} <t>$, where $\tau$ is given by left conjugation by $t$. Set $\mathcal A = \U_N$, fix any non-principal ultrafilter $\omega$ on $\N$ and consider $\PP = \PP_{\omega,\tau}^{\U_N}$, which is $*$-regular since $\U$ is positive definite. We have an injective $*$-homomorphism
$$
 f_{\omega}: \mathcal A((t;\tau))\rightarrow \PP
$$
 and, since $\rk$ is Hughes-free, we can identify, as discussed after (\ref{cond}), the $K[H]$-rings $\mathcal R(f_{\omega} (\phi(K[N])[t^{\pm 1};\tau]), \PP)$ and $\U_H$. Hence, $\D_{H,\PP} \cong \D_{H,\U_H}$ as $K[H]$-rings and $\D_{N,\PP} = \D_{N,\U_N}$. Now, observe that we can apply Proposition \ref{key} to $\D_{H,\PP}$, since in an isomorphism of $K[H]$-rings the invertibility and $H$-complexity of elements is preserved and, for every $0\ne b \in \D_{H,\U_H}$ with $\Tree_H(b)<\Tree_H(a)$, we have by definition that 
$$
 \Tree_G(b)\le \Tree_H(b) < \Tree_H(a) = \Tree_G(a)
$$
and therefore the induction hypothesis states that $b$ is invertible in $\D$, and so in $\D_{H,\U_H}$. By an abuse of notation, we denote the image of $a$ under the above isomorphism of $K[H]$-rings also by $a$. Then, we have that $a \in \D_{N,\U_N}((t;\tau))$, and we claim that if $a = \sum a_i$, then there are at least two non-zero summands. Otherwise, we would have that $\alpha\in \Rat(K^{\times}N)t^n$ for some $n$, and so Proposition \ref{source} would tell us that $\source(\alpha)\subseteq K^{\times}N$, and hence $H\le N$, a contradiction. Thus, the same Proposition \ref{key} implies that $\Tree_H(a_i)<\Tree_H(a)$ for all $i$. By the inductive assumption, if  $a_i\ne 0$, then it  is invertible. As a consequence, $a$ is invertible in $\D_{N,\U_N}((t;\tau))\subseteq \PP$, and hence in $\D_{H,\PP} \cong \D_{H,\U_H}$. This finishes the proof.
\end{proof}

 Theorem \ref{maintheorem} implies in particular Theorem \ref{main}, as we show in the next corollary. 
\begin{cor}\label{cor4.3} 
 Let $G$ be a locally indicable group. Then $\mathcal R_{\CC[G]}$ is a Hughes-free {epic} division $\CC[G]$-ring. In particular, $G$ satisfies the strong Atiyah conjecture over $\CC$.
\end{cor}
\begin{proof}
It is enough to prove the result when $G$ is countable. $\mathcal R_{\CC[G]}$ is positive definite and Proposition \ref{mainexample} tells us that the von Neumann rank $\rk_G$ is Hughes-free. Now, Theorem \ref{maintheorem} and Lemma \ref{rkdiv} imply that  $\mathcal R_{\CC[G]}$ is a Hughes-free  {epic} division $\CC[G]$-ring.  This implies that $\rk_G$ (as the unique Sylvester rank function on a division ring) takes integer values, from where the strong Atiyah conjecture follows.
\end{proof}
 
\begin{cor}\label{cor5.4}
 Let $G$ be a locally indicable group and $K$ a subfield of $\mathbb{C}$. Then the division closure $\D_{K[G]}$ of $K[G]$ in $\mathcal R_{\CC[G]}$ is a Hughes-free epic division $K[G]$-ring.
\end{cor}
\begin{proof} By Corollary \ref{cor4.3}, $\D_{K[G]}$ is a division ring. Let us check the Hughes-free condition. For any finitely generated $H\leq G$ and decomposition $H = N\rtimes_{\tau} <t>$   with $t\in H$ of infinite order, Corollary \ref{cor4.3} tells us that the elements $\{t^i\}_{i \in \N}$ are $\mathcal R_{\CC[N]}$-linearly independent,  i.e., there exists no non-trivial expression $a_0+a_1t+\dots+a_nt^n= 0$ with coefficients in $\mathcal R_{\CC[N]}$. In particular, the elements $\{t^i\}_{i\in \N}$ are $\D_{K[N]}$-linearly independent, and hence $\D_{K[G]}$ is Hughes-free.
\end{proof}
   
 It was proved in \cite{DL07} that if $G$ is a locally indicable group of homological dimension one, then any two-generator subgroup is free. From Corollary \ref{cor4.3} and \cite[Theorem 2]{KLL}, which we mentioned during the introduction, we deduce the result for any finitely generated subgroup.
\begin{cor}
  Any locally indicable group of homological dimension one is locally free.  
\end{cor}

  We finish this subsection with another application of the techniques used in the proof of Proposition \ref{key}, regarding the stability of the strong Atiyah conjecture under extensions by locally indicable groups. This was pointed out by Fabian Henneke and Dawid Kielak. 
  
\begin{pro} Let  $K$ be a subfield of $\CC$. Let  $G_2$ be a group and  $G_1$   a torsion-free normal subgroup of $G_2$   satisfying the strong Atiyah conjecture over $K$. Assume that $G_2/G_1$ is locally indicable. Then $G_2$ satisfies the strong Atiyah conjecture over $K$.
 \end{pro}
\begin{proof}
 Let $\D_{K[G_2]}$ be the division closure of $K[G_2]$ in $\mathcal R_{\CC[G_2]}$ and let $$\Phi_{G_2}: \Rat(K^{\times}G_2)\rightarrow \D_{K[G_2]}$$ be the universal map. Take $0\ne a \in \D_{K[G_2]}$. If $\Tree_{G_2}(a)= 1_{\mathcal T}$, then $a\in K^{\times}G_2$, and so it is invertible. Suppose that $\Tree_{G_2}(a)>1_{\mathcal{T}}$ and that any non-zero element of lower $G_2$-complexity is invertible. Take an $\alpha\in \Rat(K^{\times}G_2)$ realizing the $G_2$-complexity of $a$, and observe that we can assume that $\alpha$ is primitive. We put  $H = \pi(\source(\alpha))$, where $\pi: K^{\times}G_2\rightarrow K^{\times}G_2/K^{\times} = G_2$. Observe that   $a$ lies in the division closure of $K[H]$ in $\mathcal R_{\CC[G_2]}$. Now, if  $H\le G_1$, then  $a$ is invertible because $G_1$ satisfies the strong Atiyah conjecture over $K$. Otherwise, $HG_1/G_1$, and so $H$, is indicable, and therefore there exists $N\trianglelefteq H$ with $H/N \cong \Z$. Proposition \ref{mainexample} allows us to construct the corresponding diagram (\ref{dia6}) (for $\CC$, and setting $\mathcal U=\mathcal R_{\CC[G_2]}$), and so, using Proposition \ref{key} we deduce, as in Theorem \ref{maintheorem}, that $a$ is invertible in $\D_{K[G_2]}$. Thus, $\D_{K[G_2]}$ is a division ring and $G_2$ satisfies the strong Atiyah conjecture over $K$.
 
\end{proof}

 \subsection{The proof of other  corollaries}

In this subsection, we make use of the existence and uniqueness of the Hughes-free epic division ring to prove some other related conjectures regarding the group ring $K[G]$ where $G$ is locally indicable.

\begin{cor}[The independence conjecture]\label{ind}
Let $G$ be a locally indicable group, $K$ a field of characteristic zero and $\varphi_1, \varphi_2: K \to \mathbb{C}$ two different embeddings of $K$. Then, for every matrix $A\in \Mat_{n\times m}(K[G])$,
$$
  \rk_G(\varphi_1(A)) = \rk_G(\varphi_2(A))
$$
\end{cor}

\begin{proof}
Let us denote $\varphi_1(K) = K_1$ and $\varphi_2(K) = K_2$.   Corollary \ref{cor5.4} tells us that the division closures of $K_1[G]$ and $K_2[G]$ in $\mathcal{R}_{\CC[G]}$ are both Hughes-free epic division $K[G]$-rings, and so by uniqueness there exists a commutative diagram  
 $$
  \xymatrix{  & \D_{K_1[G]} \ar[dd]^{\cong} \\
            K[G] \ar[ru]^{\varphi_1} \ar[rd]_{\varphi_2} & \\
               & \D_{K_2[G]}
    }
 $$
 In particular, $\rk_G(\varphi_1(A)) = \rk_G(\varphi_2(A))$. \end{proof}
Now we are ready to prove Corollary \ref{exh}.
\begin{cor} \label{existence}
 Let $G$ be a locally indicable group and $K$ a field of characteristic zero. Then $K[G]$ has a Hughes-free {epic} division ring.
\end{cor}

\begin{proof}

Let $K_0$  be a finitely generated subfield of $K$. Let $\varphi:K_0\to \CC$ be 
 any embedding of $K_0$ into $\mathbb{C}$. Extend this embedding to $\varphi:K_0[G]\to {\CC[G]}$. For any matrix $A$ over $K_0[G]$,  we put
 $$\rk(A)=\rk_G(\varphi(A)).$$ By Corollary \ref{ind}, the value of $\rk(A)$ does not depend on the embedding $\varphi$. Thus, we have constructed a Sylvester matrix rank function $\rk$ on $K[G]$ which takes only integer values. Therefore,  it has an epic division envelope $\D$ which is a division ring. Moreover, since $\rk_G$ is Hughes-free, $\rk$ is also  Hughes-free. Hence by Lemma \ref{rkdiv}, $\D$ is a Hughes-free epic division $K[G]$-ring.
 \end{proof}

Given any field $K$ and a field extension $L/K$ we can, under some extra assumptions, relate the Hughes-free epic division rings of $K[G]$ and $L[G]$. We record it as a lemma.
\begin{lem}\label{basechange}
  Let $G$ be a locally indicable group, $K$ a field and $L/K$ a field extension. If there exists a Hughes-free epic division $K[G]$-ring $\D$ and $\D\otimes_K L$ is a domain, then the (left) classical division ring of quotients $\mathcal{Q}_l(\D\otimes_K L)$ is a Hughes-free epic division ring for $L[G]$.
\end{lem}

\begin{proof}
 First of all, note that for any subfield $L'$ of $L$ which is a finitely generated extension of $K$, the tensor product $\D\otimes_K L'$ is noetherian by the Hilbert basis theorem, and therefore it is a left Ore domain. Hence, $\D\otimes_K L$ is a left Ore domain and it makes sense to consider its left classical division ring of fractions $\mathcal{Q}_l(\D\otimes_K L)$.
 
 Now, for any subgroup $N\leq G$, and identifying $L[G]\cong K[G]\otimes_K L$, we have that the division closure of $L[N]$ in $\mathcal{Q}_l(\D\otimes_K L)$ is $\mathcal{Q}_l(\D_{N,\D}\otimes_K L)$.  Indeed, we have $L[N]\cong K[N]\otimes_K L \subseteq \mathcal Q_l(\D_{N,\D} \otimes_K L)$, and since the latter is a division subring of $\mathcal Q_l(\D \otimes_K L)$, we conclude that the division closure of $K[N]\otimes_K L$ is contained in $\mathcal Q_l(\D_{N,\D} \otimes_K L)$.     
     Conversely, any element in $\D_{N,\D}\otimes_K L$, and hence in $\mathcal Q_l(\D_{N,\D}\otimes _KL)$, is obtained from elements of $K[N]\otimes_K L$ by means of sums, products and inverses (by definition of $\D_{N,\D}$), so we have equality.  
        
 Therefore, we need to prove that for every finitely generated subgroup $H\leq G$ and any decomposition $H = N\rtimes_\tau <t>$ where $t\in H$ has infinite order and $\tau$ is given by left conjugation by $t$, the elements $\{(t\otimes 1)^i\}_{i\in \N}$ are $\mathcal{Q}_l(\D_{N,\D}\otimes_K L)$-linearly independent. Clearing denominators, it suffices to prove that the elements $\{(t\otimes 1)^i\}_{i\in \N}$ are  $\D_{N,\D}\otimes_K L$-linearly independent, and this is clear because
 $$
 \begin{array}{lll}
    < \D_{N,\D}\otimes_K L, t\otimes 1 >&=& <\D_{N,\D}, t> \otimes_K L ~\stackrel{(*)}{\cong}~ (\D_{N,\D}[x;\tau]) \otimes_K L \\[6pt]
                                                    &\cong& (\D_{N,\D} \otimes_K L)[x; \tau \otimes \Id] 
  \end{array}
 $$
 where $(*)$ comes from the Hughes-freeness of $\D$.
\end{proof}

\begin{cor}[The strong algebraic eigenvalue conjecture]\label{eigen}
Let $G$ be a countable locally indicable group and $K$ a subfield of $\mathbb{C}$. Then, for any $\lambda\in \mathbb{C}$ which is not algebraic over $K$ and for any $A\in \Mat_n(\mathcal{R}_{K[G]})$, the matrix $A-\lambda I$ is invertible in $\mathcal{U}(G)$.
\end{cor}
\begin{proof}
  Let $\D_{K[G]}$ and $\D_{K(\lambda^{-1})[G]}$ denote, respectively, the Hughes-free epic division $K[G]$ and $K(\lambda^{-1})[G]$-rings, which can be constructed as division subrings of $\mathcal{U}(G)$ by Corollary \ref{cor5.4}. 
  
  Since $\lambda$ is not algebraic over $K$, 
 we have that $\D_{K[G]}\otimes_K K[\lambda^{-1}] \cong \D_{K[G]}[x]$, the polynomial ring in the indeterminate $x$. In particular, $\D_{K[G]}\otimes_K K[\lambda^{-1}]$, and hence $\D_{K[G]}\otimes_K K(\lambda^{-1})$, is a domain. In addition, notice that $\mathcal Q_l(\D_{K[G]}\otimes_K K(\lambda^{-1})) = \mathcal Q_l(\D_{K[G]}\otimes_K K[\lambda^{-1}])$, so adding up we obtain
$ \mathcal Q_l(\D_{K[G]}\otimes_K K(\lambda^{-1})) \cong \D_{K[G]}(x) $.

Thus, Lemma \ref{basechange} and the uniqueness of Hughes-free epic division ring, tell us \color{black} that $\D_{K(\lambda^{-1})[G]} \cong \D_{K[G]}(x) $ and, therefore, we can define an injective homomorphism 
$$
\begin{matrix}
  \D_{K(\lambda^{-1})[G]} & \longrightarrow & \D_{K[G]}((x))
\end{matrix}
$$
in which $\lambda \mapsto x^{-1}$. The image $A-x^{-1}I~$ 
 of the matrix $A-\lambda I\in \Mat_n(\D_{K(\lambda^{-1})[G]})$ under the above homomorphism is invertible,  and so, in particular, it is a non-zero-divisor. The injectivity implies that $A-\lambda I$ is a non-zero-divisor in the von Neumann  regular ring $\Mat_n(\D_{K(\lambda^{-1})[G]})$, and hence invertible.
\end{proof}

Just as a remark before stating the next corollary, recall that if  $G$ is a countable group and $H\leq G$, then  we can identify any element $\varphi$ of the group von Neumann algebra $\mathcal{N}(H)$ with the element of $\mathcal{N}(G)$ that assigns to any tuple in $\oplus_{t\in T} \,t\,l^2(H)$ the tuple obtained by applying $\varphi$ component-wise.  
In addition, an element of $G\backslash H$ (as an operator in $\mathcal{N}(G)$) does not fix the components of  $\oplus_{t\in T} \,t\,l^2(H)$. Therefore, we conclude that $\mathcal{N}(H) \cap G = H$.  Since the elements of $G$, viewed as operators on $l^2(G)$, are bounded, we also have that $\mathcal{U}(H) \cap G = H$.  Therefore,  $\mathcal R_{\CC [H]} \cap G = H$ and, in particular,    for any subfield $K$ of $\CC$, the division closure $\D_{K[H]}$ of $K[H]$ in $\mathcal{U} (H)$ satisfies $\D_{K[H]} \cap G = H$.  The former equality also implies that $\mathcal R_{\CC[H]} \cap K^{\times}G = K^{\times}H$.  

The following result  contains  the center conjecture as a particular case when the field $K$ is closed under complex conjugation. 
\begin{cor}\label{cent} 
Let $G$ be a countable locally indicable group and $K$ a subfield of $\mathbb{C}$. If $\D_{K[G]}$ is the division closure of $K[G]$ in $\mathcal{U}(G)$ then 
$$
 \D_{K[G]}  \cap \mathbb{C} = K.
$$
\end{cor}

\begin{proof}

By Corollary \ref{cor5.4}, $\D_{K[G]}$ is an epic Hughes-free division $K[G]$-ring. Thus, for every finitely generated subgroup $H\leq G$ and decomposition $H = N\rtimes <t>$ with $t\in H$ of infinite order, we can construct the following diagram 
$$
\xymatrix{ & \D_{K[H]} \ar@{^(->}[]+<4ex,0ex>;[r] \ar@{^(->}[]+<0ex,-2.5ex>;[d]  & \D_{K[N]}((t;\tau)) \ar@{^(->}[]+<0ex,-2.5ex>;[d] \\
\CC \ar@{^(->}[]+<2ex,0ex>;[r] &  \D_{\CC [H]} \ar@{^(->}[]+<4ex,0ex>;[r] & \D_{\CC [N]}((t;\tau))}
$$
The elements of $\CC$ are identified then with the Laurent series  $\D_{\CC [N]}((t;\tau))$ with just one possible non-zero summand corresponding to the constant term.   Consider the universal morphism $\Phi: \Rat(K^{\times}G) \rightarrow \D_{K[G]}$, and take any non-zero element $a\in  \D_{K[G]}  \cap \mathbb{C}$. If $\Tree_G(a) = 1_{\mathcal{T}}$, then $a\in K^{\times}G \cap \CC = K^{\times}$.

 If $\Tree_G(a)>1_{\mathcal{T}}$, then let $\alpha \in \Rat(K^{\times}G)$ realize the $G$-complexity of $a$. There exist a primitive element $\alpha'$ and $u\in K^{\times}G$ such that $\alpha = \alpha'u$. Setting $H = \pi (\source(\alpha)) = \pi(\source(\alpha'))$, where $\pi: K^{\times}G\rightarrow K^{\times} G/K^{\times} = G$ is the natural map,  and $a' = \Phi(\alpha')$, we obtain that $a'\in \D_{K[H]}$ and $\Tree_H(a') = \Tree(\alpha')$. The same reasoning from Theorem \ref{maintheorem} applies and gives us a decomposition of $a'$ as an element of $\D_{K[N]}((t;\tau))$ with at least two summands. Moreover, we have that $u = (a')^{-1}a \in \mathcal{R}_{\CC [H]} \cap K^{\times}G = K^{\times}H$, so $a=a'u \in \D_{K[H]}$ is a complex number whose representation as a Laurent series has at least two summands, a contradiction. We deduce that a non-zero element $a$ in $\D_{K[G]} \cap \mathbb{C}$ must have $\Tree_G(a) = 1_{T}$, and therefore $a\in K$.
 
\end{proof}

Recall that a group $G$ is called {\bf ICC} if all non-trivial conjugacy classes of $G$ are infinite.
\begin{cor}\label{centICC} Let $G$ be a locally indicable  ICC   group, $K$ a field of characteristic zero and $\D$ a Hughes-free epic division $K[G]$-ring. Then $Z(\D)=K$.
\end{cor}
\begin{proof} Assume that $a\in Z(\D)\setminus K$. Then there are a finitely generated subgroup $H_0\le G$ and a finitely generated subfield $K_0$ of $K$ such that $a\in \D_{K_0[H_0]}$ (here $\D_{K_0[H_0]}$ denotes the division closure of $K_0[H_0]$ in $\D$). We can embed $H_0$ in a countable  ICC subgroup $H$ of $G$. Indeed, starting with $H_0$ we can define for every $i>0$ a countable subgroup $H_i$ of $G$ such that all $H_i$-conjugacy classes of non-trivial elements of $H_{i-1}$ are infinite, and so $H = \bigcup H_i$ is a countable ICC group containing $H_0$. 

Embed now $K_0$ into $\CC$. By Corollary \ref{cor5.4}, $\D_{K_0[H]}$ is isomorphic to the division closure $\D'_{K_0[H]}$ of $K_0[H]$ inside $\mathcal R_{\CC[H]}$ . Since $H$ is ICC,
 $$C_{ \mathcal U(H)}(H)=Z(\mathcal U(H))=\CC.$$ Thus, by Corollary \ref{cent}, 
  $ C_{\D'_{K_0[H]}}(H) =K_0.$
  and, therefore, we obtain that
 $a\in C_{\D_{K_0[H]}}(H) =K_0.$
 This contradicts our assumption on $a$.
  
\end{proof}

  Using the previous result and Corollary \ref{existence} we can prove that the conditions on Lemma \ref{basechange} are satisfied when we deal with fields of characteristic zero.

\begin{cor}
Let $L/K$ be an extension of fields of characteristic zero, $G$ a locally indicable group and  $\D$  a Hughes-free epic division $K[G]$-ring. Then $\D\otimes _K L$ is a domain. In particular the Hughes-free epic division $L[G]$-ring is isomorphic to the classical ring of quotients of $\D\otimes_K L$
\end{cor}

\begin{proof}
First let us assume that $G$ is ICC. In this case  $Z(\D)=K$ by Corollary \ref{centICC}, and therefore $\D \otimes_K L$ is a simple ring. By Corollary  \ref{existence}, there exists a Hughes-free epic division $L[G]$-ring $\tilde \D$. Identify $\D$ with the division closure of $K[G]$ in $\tilde \D$. Since $\D \otimes_K L$ is simple, it is isomorphic to the subring of $\tilde \D$ generated by $\D$ and $L$, and hence it is a domain. 

For an arbitrary $G$, the wreath product $G\wr \Z$ is again locally indicable and ICC, and so $\D\otimes_K L$ can be embedded in a domain.  This concludes the proof, and the last assertion follows from Lemma \ref{basechange}. 
\end{proof}

 If $K$ and $L$ are subfields of $\CC$, this corollary states that the division closure $\D_{L[G]}$ of $L[G]$ inside $\mathcal U(G)$ is isomorphic to the classical   ring of quotients of $\D_{K[G]}\otimes_K L$. It is proved throughout \cite{Ja17base} that the same statement holds when we consider sofic groups satisfying the strong Atiyah conjecture  instead of locally indicable groups. We expect that this property holds in general.

\begin{Conjecture} Let $L/K$ be an extension of subfields of $\CC$ and $G$ any group satisfying the strong Atiyah conjecture. Let $\D_{K[G]}$ and $\D_{L[G]}$ denote, respectively, the division closures of $K[G]$ and $L[G]$ in $\mathcal U(G)$. Then $\D_{L[G]}$ is isomorphic to the classical   ring of quotients $\D_{K[G]}\otimes_K L$.
\end{Conjecture}

\section{The L\"uck approximation for locally indicable groups} \label{Lucksection}

The goal of this section is to give a proof of the L\"uck approximation conjecture in the space of marked groups when the group being approximated is virtually  locally indicable. More precisely, we want to prove the following statement.

\begin{teo} \label{Lueck}  Let $F$ be a finitely generated free group and assume that $(M_i)_{i\in \N}$ converges to $M$ in the space of marked groups of $F$. If $F/M$ is virtually  locally indicable, then $\rk_{F/M_i}$ converges to $\rk_{F/M}$ in the space of Sylvester matrix rank functions on $\CC[F]$.
 \end{teo}
In this context, by convergence of the rank functions $\rk_{F/M_i}$ to $\rk_{F/M}$ we mean that, for every matrix $A \in \Mat(\CC[F])$, we have
$$
 \lim_{i\rightarrow \infty} \rk_{F/M_i}(A) = \rk_{F/M}( A).  
$$
Our strategy will be to prove that  for any non-principal ultrafilter $\omega$ on $\N$  
$$
 \lim_{\omega} \rk_{F/M_i}= \rk_{F/M} \textrm {\ as Sylvester rank functions on $\CC[F]$.}
$$
During the proof, the capability to compare between all the involved rank functions will play an important role.  
If $R$ is any ring and $\rk_1$, $\rk_2 \in \mathbb{P}(R)$, we write $\rk_1\le \rk_2$ if, for every matrix $A\in \Mat(R)$, we have
$$
\rk_1 (A) \leq \rk_2(A).
$$ 
In this section we fix a non-principal ultrafilter $\omega$ on $\N$. The starting point of our proof is the Kazhdan inequality (see, for example, \cite[Proposition 10.7]{Ja17surv}).
\begin{pro}\label{kazhdan}
With the previous notation $$
 \lim_{\omega} \rk_{F/M_i}\ge \rk_{F/M} \textrm {\ as Sylvester rank functions on $\CC[F]$.}
$$

\end{pro}
\subsection{Comparing Sylvester matrix rank functions}
We present in this subsection some general results on the comparison between rank functions, and some more specific ones regarding the von Neumann rank functions that appear in our context. 
We begin with a useful observation, namely, if we have $\rk_1 \le \rk_2$ on $R$, then this property extends to its division closure $\D$ inside a certain ring.
\begin{pro} \label{comppr} Let $R$ be a ring and let $\{\rk_i\}_{i=1}^n$ be a family of Sylvester rank functions on $R$ such that $\rk_{i}\le \rk_{i+1}$ for any $i$. Assume that  $( S_i,\varphi_i, \rk^\prime_i)$ is an envelope of $\rk_i$, and set $S=\prod S_i$, $\varphi=(\varphi_i):R\to S$. If $\D$ is the division closure of $\varphi(R)$ in $S$ and $\pi_i:\D\to S_i$ is the standard projection, then $$\pi_i^{\sharp}(\rk_i^\prime)\le \pi_{i+1}^{\sharp}(\rk'_{i+1}).$$ In particular, $\pi_{n}$ is injective.
\end{pro}
 \begin{proof}
 Let $A$ be a matrix over $\D$.
 By \cite[Proposition 4.2.2]{CohSF} and Cramer's rule (\cite[Proposition 4.2.3]{CohSF}), there exist $k\ge 1$, a matrix $A'$ over $\varphi(R)$ and invertible matrices $P$, $Q$ over $\D$ such that
 $$
    A \oplus I_k = PA'Q.
 $$ 
Suppose $A' = \varphi(B)$ for some $B\in \Mat(R)$. This implies that
 $$\rk_i'(\pi_i(A))+k = \rk_i'(\pi_i(A')) = \rk_i'(\varphi_i(B)) = \rk_i(B) $$
and the first assertion of the proposition follows because $\rk_i \leq \rk_{i+1}$ for all $i$. Finally, if $d = (d_i)\in \D$ is such that $\pi_n(d) = 0$, then by faithfulness of $\rk_i'$ we deduce that $\pi_i(d) = 0$ for all i, and hence $d = 0$. 
 \end{proof}
 Notice that in the previous proposition the image of the standard projection $\pi_i$ actually lies in $\D_i$, the division closure of $\varphi_i(R)$ in $S_i$. This is due to the fact that every element $(a_i)\in \mathcal \D$ is obtained from  $\varphi(R)$ by taking sums, products and inverses, so every $a_i$ is obtained from $\varphi_i(R)$ by means of the same operations. Thus, we can consider $\pi_i: \D\rightarrow \D_i$, what will be used in the subsequent consequences.  

\begin{cor} \label{commphi}
 Let $G$ be a group, and let $\rk_1, \rk_2$ be regular Sylvester matrix rank functions on $R = \CC[G]$ with regular envelopes $(\mathcal U_i, \rk_i',\varphi_i)$, $i = 1, 2$. Let $\D_{R,\mathcal U}$ and $\D_{R,1}$ be the division closures of the image of $R$ in $\mathcal U = \mathcal U_1 \times \mathcal U_2$ and $\mathcal U_1$, respectively. Then, if $\rk_1\le \rk_2$, the following diagram commutes
 $$
   \xymatrix{ \Rat (\CC^{\times}G) \ar@{->}[r]^-{\Phi_{\mathcal U}} \ar@{->}[rd]_{\Phi_1} & \D_{R, \mathcal U} \ar@{->}^{\pi_1}[d] \\
      & \D_{R,1}
     }
 $$
where $\Phi_{\mathcal U}$ and $\Phi_1$ are the universal morphisms induced by the $\CC^{\times}G$-rational structures on $\D_{R,\mathcal U}$ and $\D_{R,1}$ described in subsection \ref{ss:divclo}, and $\pi_1$ is the standard projection. 
\end{cor}

\begin{proof}
  For any $d=(d_1,d_2)\in \D_{R,\mathcal U}$, we have $\rk_1'(d_1)\le \rk_2'(d_2)$ by Proposition \ref{comppr}. Since $\mathcal{U}_2$ is regular, we deduce from here that $d$ is invertible in $\D_{R,\mathcal U}$ if and only if $\pi_1(d) = d_1$ is invertible in $\D_{R,1}$. This implies that $\pi_1$ respects the rational operation, i.e., $\pi_1(d^{\diamond}) = \pi_1(d)^{\diamond}$, and therefore, it is also a morphism of rational $\CC^{\times}G$-semirings. By the uniqueness of $\Phi_1$, necessarily $\Phi_1 = \pi_1 \circ \Phi_{\mathcal U}$.
\end{proof}
The following corollary is a technical result which will be essential for the proof of the L\"uck approximation conjecture.

\begin{cor} \label{key2}
 Let $H$ be a group, and let $\rk_1, \rk_2, \rk_3$ be regular Sylvester matrix rank functions on $R = \CC[H]$ with regular envelopes $(\mathcal U_i, \rk_i',\varphi_i)$, $i = 1, 2, 3$. Let $\D_{R,12}$ and $\D_{R,13}$ be the division closures of the image of $R$ in $\mathcal U_1 \times \mathcal U_2$ and $\mathcal U_1 \times \mathcal U_3$, respectively. Assume that $\rk_1\le \rk_2 \le \rk_3$ and consider the universal morphisms 
 $$
 \begin{matrix}
   \Phi_{12}: \Rat(\CC^{\times}H)\rightarrow \D_{R,12} & \Phi_{13}: \Rat(\CC^{\times}H)\rightarrow \D_{R,13}
 \end{matrix}  
 $$ 
Then, for any $\alpha \in \Rat(\CC^{\times}H)$, if $\Phi_{12}(\alpha)$ is non-zero then $\Phi_{13}(\alpha)$ is non-zero. Moreover, $\Phi_{12}(\alpha)$ is invertible if and only if $\Phi_{13}(\alpha)$ is invertible.
\end{cor}

\begin{proof}
 Let $\mathcal U = \mathcal U_1\times\mathcal U_2\times\mathcal U_3$, $\varphi = (\varphi_i)$, $\D_{R,\mathcal U}$ the division closure of $\varphi(R)$ in $\mathcal U$, and $\Phi_{\mathcal{U}}: \Rat(\CC^{\times}H)\rightarrow \D_{R,\mathcal U}$ the universal morphism. Denote by $\pi_{12}$ and $\pi_{13}$ the natural projections from $\D_{R,\mathcal U}$ to $\D_{R,12}$ and $\D_{R,13}$, respectively. Reasoning as in Corollary \ref{commphi}, both $\pi_{12}$ and $\pi_{13}$ preserve the corresponding rational structure, and this implies that we have a commutative diagram
 $$
   \xymatrix{ & \ar@{->}[dl]_{\pi_{12}} \D_{R,\mathcal U} \ar@{->}[rd]^{\pi_{13}} & \\
       \D_{R,12}  & \ar@{->}[l]_{\Phi_{12}} \Rat(\CC^{\times}H) \ar@{->}[r]^{\Phi_{13}} \ar@{->}[u]_{\Phi_{\mathcal U}} & \D_{R,13}  
     }
 $$
Take any $\alpha \in \Rat(\CC^{\times}H)$ and set $\Phi_{\mathcal U}(\alpha) = (a_1,a_2,a_3)$. Hence, we have that 
 $$
 \begin{matrix}
    \Phi_{12}(\alpha) = (a_1, a_2), & \Phi_{13}(\alpha) = (a_1,a_3),
  \end{matrix} 
 $$ 
and also, by Proposition \ref{comppr}, $\rk_1'(a_1)\leq \rk_2'(a_2) \leq \rk_3'(a_3)$. Thus, if $\Phi_{12}(\alpha)$ is non-zero, then $\rk_3'(a_3)>0$, and hence $\Phi_{13}(\alpha)$ is non-zero. In addition, since $\U_i$ is regular for any $i$, both of them are invertible if and only if $a_1$ is invertible. 
  
\end{proof}
The motivation and necessity of this result lies in the fact that, at a certain point, we will need to prove that an element, expressible as $\Phi_{12}(\alpha)$, is invertible. Corollary \ref{key2} will then allow us to pass from the ambient $\mathcal U_1 \times \mathcal U_2$ to an appropriate ambient $\mathcal P = \mathcal U_1 \times \mathcal U_3$ on which the conditions of Proposition \ref{key} are satisfied, and therefore, to tackle instead the invertibility of the non-zero element $\Phi_{13}(\alpha)$ by means of induction on the complexity.

Now we turn to a  specific case we are going to deal with later. Let $F$ be a finitely generated free group and assume that $(M_i)_i$ converges to $M$ in the space of marked groups of $F$. 
   Assume that there exists a normal subgroup $N$ of $F$ such that $M\leq N$ and $F/N \cong \Z$.   We put  $K_i=M_i \cap N$. It is clear that $(K_i)_i$ also    converges to $M$. 
In the proof of  Theorem \ref{Lueck} we will pass from $(M_i)_i$ to  $(K_i)_i$ and, therefore, to get the result we need to know the relation between the following rank functions on $\CC[F]$
$$
\rk_{(F/M_i)_i} = \displaystyle \lim_{\omega}  \rk_{F/M_i} \textrm{ and } \rk_{(F/K_i)_i}=\displaystyle \lim_{\omega} \rk_{F/K_i}.$$

 The next lemmas will be useful to our purpose.

\begin{lem}\label{natextt+1}
Let $\U$ be a regular ring, $\rk$ a Sylvester matrix rank function on $\U$, and let $\wrk$ denote the natural extension of $\rk$ to $\U[t]$. Then, for any matrix $A = A(t)\in \Mat(\U[t])$,
$$
 \wrk(A(t)) = \wrk(A(t+1)).
$$ 
\end{lem}
\begin{proof}
 By \cite[Lemma 7.3]{Ja17base}, we have that  for any $A\in \Mat_{n\times m}(\U[t])$, 
 $$
   \wrk(A(t)) = \lim_{i\rightarrow \infty} \frac{\dim ((Q_{i-1})^nA(t))}{i},
 $$
where $Q_{i-1}$ is the set of polynomials in $\U[t]$ of degree at most $i-1$ {and $\dim$ denotes the Sylvester module rank function associated to $\rk$}. Observe that, as a set, $Q_{i-1}$ coincides with $Q_{i-1}(t+1)$, the set of polynomials of degree at most $i-1$ in the indeterminate $t+1$, and that there exists a natural $\U$-isomorphism $(Q_{i-1})^nA(t) \cong (Q_{i-1}(t+1))^nA(t+1)$. Therefore 
$$
 \wrk(A(t)) = \lim_{i\rightarrow \infty} \frac{\dim ((Q_{i-1})^nA(t))}{i} = \lim_{i\rightarrow \infty} \frac{\dim ((Q_{i-1})^nA(t+1))}{i} = \wrk(A(t+1)).
$$
\end{proof}
\begin{lem}\label{claim3.4}
 Let $G$ be a group and assume that $G = H\times \Z$ for some $H\leq G$. If $\pi: \CC[G]\rightarrow \CC[H]$ is the map induced by the projection $\pi: G\rightarrow G/\Z\cong H$, then $\pi^{\sharp}(\rk_H) \le \rk_G$.
\end{lem}
\begin{proof}
 Observe that $\CC[G] = (\CC[H])[t^{\pm 1}]$, and that for any polynomial $p=\sum a_it^i$ with coefficients in $\CC[H]$, $\pi(p) = \sum a_i$. By Proposition \ref{mainexample}, we know that as rank functions over $\CC[G]$, $\rk_G$ is the natural extension of $\rk_H$. Therefore, we want to show that, for any matrix $A = A_0+A_1t+\dots+A_nt^n\in \Mat(\CC[H][t])$, 
 $$
 \wrk_H(A)\myeq{Lemma \ref{natextt+1}}  \wrk_H(A(t+1))\ge \rk_H\left(\sum_{i=0}^n A_i\right),
 $$
what follows from   \cite[Proposition 7.6]{Ja17base}.
\end{proof}
 Now we are ready to compare $\rk_{(F/M_i)_i}$   and $\rk_{(F/K_i)_i}$. 
\begin{pro} \label{appchange}
 Let $F$ be a finitely generated free group and assume that $(M_i)_i$ converges to $M$ in the space of marked groups of $F$. Assume that there exists a normal subgroup $N$ of $F$ such that $M\leq N$ and $F/N \cong \Z$, and set $K_i=M_i \cap N$. Then the following holds:
 \begin{itemize}
   \item[(1)] $\rk_{(F/K_i)_i}$ is the natural extension of the restriction of $\rk_{(F/M_i)_i}$ to $\CC[N]$.
   \item[(2)] $\rk_{(F/M_i)_i} \leq \rk_{(F/K_i)_i}$.
 \end{itemize}
\end{pro}

\begin{proof}
 (1) Since $F/N \cong \Z$, there exists $t\in F$ such that $\CC[F] \cong (\CC[N])[t^{\pm 1};\tau]$ and $\CC[F/K_{i}] \cong (\CC[N/K_i])[t^{\pm 1};\tau]$, where $\tau$ is induced by left conjugation by $t$. Proposition \ref{mainexample} states that, as a rank function over $\CC[F/K_i]$, $\rk_{F/K_i}$ is the natural extension of $\rk_{N/K_i}$ as a rank function on $\CC[N/K_i]$, and consequently, as a rank function on $\CC[F]$, $\rk_{F/K_i}$ is the natural extension of $ \rk_{N/K_i} $ as a rank function on $\CC[N]$. Thus, by Corollary \ref{wlimext}, $\rk_{(F/K_i)_i}$ is the natural extension of $\lim_{\omega}  \rk_{N/K_i} $. Since $K_i = M_i\cap N$,   then, as rank functions on $\CC[N]$, $ \rk_{N/K_i} $ coincides with the restriction of $\rk_{F/M_i}$.
 Hence, the result follows.
 
 (2)  For each $i\in\N$ we set $G_i = F/M_i \times F/N$. The natural maps $F\to G_i$ and $G_i\to F/M_i$ induce the homomorphisms 
 $$\pi _1 :\CC[F]\to \CC[ G_i] \textrm{\ and \ } \pi_2: \CC[G_i]\to \CC[F/M_i].$$   
 Since $K_i = M_i \cap N$, $ \pi_1(F)\cong F/K_i$. Thus, we obtain that
 $$\rk_{F/K_i}=\pi_1^\sharp(\rk_{G_i})  \textrm{\ and \ } \rk_{F/M_i}=(\pi_2\circ \pi_1)^\sharp(\rk_{F/M_i})=\pi_1^\sharp(\pi_2^\sharp(\rk_{F/M_i}))$$  as    rank functions on $ \CC[F]$.
 (In the second and third appearances of $\rk_{F/M_i}$ in the previous formula we see it as a rank function on $\CC[F/M_i]$.)  By  Lemma \ref{claim3.4},
 $\pi_2 ^\sharp(\rk_{F/M_i}) \le \rk_{G_i}$. Hence $\rk_{F/M_i}\le \rk_{F/K_i}$ as  rank functions on $ \CC[F].$ This clearly implies that $\rk_{(F/M_i)_i} \leq \rk_{(F/K_i)_i}$.  \end{proof}

\subsection{The proof of the L\"uck approximation conjecture}

As at the beginning of the section, let $F$ be a finitely generated free group and assume that $(M_i)_i$ converges to $M$ in the space of marked groups of $F$. We put $G=F/M$. Consider the positive definite $*$-regular ring $\mathcal R = \prod \mathcal R_{\CC[F/M_i]}$ and let $p_i: \mathcal R \rightarrow \mathcal R_{\CC[F/M_i]}$ be the $i$-th projection. Then  $\rk = \lim_{\omega} p_i^{\sharp}(\rk_{F/M_i})$ defines a rank function on $\mathcal R$.  If $\pi:\CC[F]\to \mathcal R$ denotes the canonical map, then  we have that
$ \rk_{(F/M_i)_i} = \pi^{\sharp} (\rk)$. 
 
Set $\mathcal R_{\omega}=\mathcal R/\ker \rk$. Then $\pi$ induces a $*$-homomorphism 
$
 \pi_{\omega}: \CC[F] \rightarrow \mathcal R_{\omega}$.
Note that $\mathcal R_{\omega}$ is positive definite because $\mathcal R$ is positive definite and $\ker \rk$ is an ideal of $\mathcal R$.  
Hence, if $\U = \mathcal R(\pi_{\omega}(\CC[F]), \mathcal R_{\omega}))$, then $(\mathcal U, \pi_{\omega},\rk)$ is the positive definite $*$-regular envelope of $\rk_{(F/M_i)_i}$. 

Moreover, since $(M_i)_i$ approximates $M$, we have  that the kernel of $\pi_\omega$ is the ideal $I$ generated by $\{m-1: m\in M\}$. Indeed, for any $m$ in $M$ we have $\rk(\pi_{\omega}(m-1)) = 0$. Thus, the properties of rank functions imply that the image of any element in $I$ has zero rank, and so $I\subseteq\ker \pi_{\omega}$. On the other hand, Proposition \ref{kazhdan} tells us that $ \rk_{F/M}\le \pi_{\omega}^{\sharp}(\rk)$. Hence, if $a\notin I$, then $\rk_{F/M}(a)\ne 0$ and consequently $\rk(\pi_\omega(a))\ge \rk_{F/M}(a)> 0$, from where $\pi_{\omega}(a)\ne 0$. This means $a\not \in \ker \pi_\omega$, and so,   $\pi_{\omega}(\CC[F]) \cong \CC[F]/I \cong \CC[G]$.  

 Let $\psi$ denote the induced isomorphism between $\CC[G]$ and $\pi_\omega(\CC[F])$. Then we can think of $\rk_{(F/M_i)_i}$ also as a $*$-regular rank function on $\CC[G]$ with positive definite $*$-regular envelope $(\U, \psi, \rk)$. 
  
   In addition, observe that if $H$ is a non-trivial finitely generated subgroup of $G$ and $F'$ is a finitely generated subgroup of $F$ such that $H = F'M/M \cong F'/(F'\cap M)$, then $(M_i^\prime=F^\prime \cap M_i)_i$ converges to $M^\prime=F^\prime\cap M$ in the space $\MG(F^\prime)$ and
$$ 
  \rk_{(F/M_i)_{i |\CC[F']} }=  \rk_{(F^\prime/M_i^\prime )_i} \textrm{\ as rank functions on \ } \CC[F^\prime].
$$
We will obtain Theorem \ref{Lueck} as a consequence of the following proposition.

\begin{pro} \label{maintheoremluck}
 Let $F$ be a finitely generated free group, and assume that $(M_i)_i$ converges to $M$ in the space of marked groups of $F$ and that $G=F/M$ is locally indicable. Let $(\mathcal U, \psi, \rk)$ be the $*$-regular envelope of $\rk_{(F/M_i)_i}$ as a rank function on $\CC[G]$. Let $\varphi=(i,\psi): \CC[G] \rightarrow \mathcal R_{\CC[G]} \times \U$ be the induced map. 
 Then, the division closure of $\varphi(\CC[G])$ in $\mathcal R_{\CC[G]}\times \U$ is a division ring.
\end{pro}

\begin{proof}
 Observe that $\rk_{(F/M_i)_i}$, and hence $\U$, is uniquely determined by the approximation $(F/M_i)_i$ of $G$. Set $\s = \mathcal R_{\CC[G]} \times \U$, let $\D = \D_{G,\s}$ be the division closure of $\varphi(\CC[G])$ in $\s$ and denote, for any subgroup $H\le G$, $\s_H = \mathcal R_{\CC[H]} \times \U_H$, where $\U_H$ is the $*$-regular closure of $\psi(\CC[H])$ in $\U$. Consider the universal morphism of rational $\CC^{\times}G$-semirings $\Phi_{G,\s}: \Rat(\CC^{\times}G)\rightarrow \D$. We are going to prove the result for all locally indicable groups $G$ and all approximations $(F/M_i)_i$ at the same time, using induction on the complexity. More precisely, we are going to show that if $\alpha\in \Rat(\CC^{\times}G)$ realizes the $G$-complexity of a non-zero $a\in \D_{G,\s}$ for some locally indicable group $G$ and for some approximation $(F/M_i)_i$ of $G$, then $a$ is invertible in  $\D_{G,\s}$. Since $\Phi_{G,\s}$ is surjective for every choice of $G$ and $(F/M_i)_i$, this gives the result.
 
 For any locally indicable group $G$ and for any approximation $(F/M_i)_i$ of $G$, if $\alpha\in \Rat(\CC^{\times}G)$ satisfies $\Tree(\alpha) = 1_{\mathcal T}$, then $\alpha \in \CC^{\times}G$ realizes the complexity of $a = \Phi_{G,\s}(\alpha) \in \varphi(\CC^{\times}G)$, and $a$ is invertible. Now suppose that $\Tree(\alpha)>1_{\mathcal T}$ realizes the $G$-complexity of a non-zero $a\in \D_{G,\s}$ for some locally indicable $G$ and for some approximation $(F/M_i)_i$ of $G$, and assume by induction that we have proved the result for all locally indicable $G'=F'/M'$, for all approximations $(F'/M'_i)_i$ of $G'$ and for all $\beta\in \Rat(\CC^{\times}G')$ with $\Tree(\beta)<\Tree(\alpha)$. We want to show that $a$ is invertible in $\D_{G,\s}$. 
 
 Observe that we can assume, without loss of generality, that $\alpha$ is primitive, because multiplying by a unit in $\CC^{\times}G$ does not change the complexity of $\alpha$ nor the invertibility of its image $a$. Define $H$ to be the image of $\source(\alpha)$ under $\CC^{\times}G \rightarrow \CC^{\times}G/\CC^{\times} = G$ and observe that $H$ is finitely generated by Proposition \ref{source}. Then, $\alpha \in \Rat(\CC^{\times}H) $ and therefore $a\in \D_{H,\D} = \D_{H,\s} = \D_{H,\s_H}$, this latter equality due to Lemma \ref{divreg}. If $H$ is trivial, then $\D_{H,\s_H} \cong \CC$ and the result is clear. Otherwise, there exists a normal subgroup $N\trianglelefteq H$ such that $H/N \cong \Z$. 
 
 We would like to apply Proposition \ref{key} to $a$ in order to decompose it into elements of lower complexity, but a priori we do not have a candidate for the ring $\mathcal{A}$. Therefore, we are going to use Propositions \ref{appchange} and \ref{key2} to rephrase the question of invertibility of $a$ to the question of invertibility of a certain $a'$, lying in an appropriate ring on which we can use Proposition \ref{key}. 
 
 First, we can take a finitely generated subgroup $F'$ of $F$ such that $H =  F'M/M \cong F'/(F'\cap M)$. In this case, the preimage $N'$ of $N$ is a normal subgroup of $F'$ containing $M' = F' \cap M$ such that $F'/N' \cong \Z$. Set $M_i' = F'\cap M_i$ and $K_i' = M_i'\cap N'$, and observe that $H = F'/M'$ is a locally indicable group and that both $(M_i')_i$ and $(K_i')_i$ converge to $M'$ in the space of marked groups of $F'$. 
 
 Choose an element $t\in F^\prime$ such that $F^\prime=\langle N^\prime ,t\rangle$.   Hence $$\CC[H] \cong (\CC[N])[t^{\pm 1};\tau]\textrm{\ and \ } \CC[F'] \cong (\CC[N'])[t^{\pm 1};\tau],$$ where $\tau$ is given by left conjugation by $t$. Since $\tau$ can be extended to an automorphism of both $\mathcal R_{\CC[N]}$ and $\U_N$, it extends to an automorphism of the regular ring $\mathcal A = \s_N = \mathcal R_{\CC[N]}\times \U_N$. Therefore, we have a homomorphism 
 $$
   \varphi: \CC[H] \cong (\CC[N])[t^{\pm 1};\tau] \rightarrow \mathcal A((t;\tau)).
 $$
Moreover, since $\mathcal R_{\CC[G]}$ is the epic Hughes-free division $\CC[G]$-ring, we have that $\mathcal R_{\CC[H]} \cong \mathcal R_{\CC[N]}(t;\tau)$, the Ore division ring of fractions of $\mathcal R_{\CC[N]}[t^{\pm 1};\tau]$, and so $\varphi(\CC[H]) $ lies in a subring of  $\mathcal A((t;\tau))$ isomorphic to  $\mathcal R_{\CC[H]}\times \U_N((t;\tau))$. 

Recall that  for the given ultrafilter $\omega$ we can construct, as in (\ref{eq4}), a ring $\mathcal P_{\omega,\tau}^{\U_N}$ which is $*$-regular and positive definite (since $\U_N$ is both), an injective $*$-homomorphism 
$$
f_{\omega} : \mathcal U_N[t^{\pm 1};\tau] \rightarrow \mathcal P_{\omega,\tau}^{\U_N}
$$ 
and a rank function $\rk_{\omega}$ such that $f_{\omega}^{\sharp}(\rk_{\omega})$ is the natural extension of the restriction of $\rk$ to $\U_N$. 

 Since $\psi$ embeds  $\CC[N]$ into $\mathcal U_N$, the subring $\psi(\CC[N])[t^{\pm 1};\tau]$ of  $\mathcal U_N[t^{\pm 1};\tau]$ is isomorphic to $\CC[H]$. Denote this isomorphism by $\psi^\prime$.  Proposition \ref{prop7.4} implies that $ f_{\omega}^{\sharp}(\rk_{\omega})$, as a rank function on  $\CC[N][t^{\pm 1};\tau]$, is the natural extension of the restriction of $\rk_{(F/M_i)_i} $ to $\CC[N]$.  Thus, from Proposition \ref{appchange}(1), we obtain that  $  ( f_{\omega}\circ\psi^\prime)^{\sharp}(\rk_{\omega})$  equals $\rk_{(F'/K_i')_i}$ as rank functions on $\CC[H]$. Let  $\U' = \mathcal R((f_{\omega}\circ \psi^\prime)(\CC[H]), \mathcal P_{\omega,\tau}^{\U_N})$ and let $\rk^\prime$ be the restriction of $\rk_\omega$ to $\U^\prime$. Then $(\mathcal U^\prime,\psi^\prime,\rk^\prime )$ is a $*$-regular envelope of $\rk_{(F'/K_i')_i}$.  
 
 In addition, by Proposition \ref{appchange}(2) and Proposition \ref{kazhdan}, we have that
\begin{equation}\label{rk>}
\rk_H\le \rk_{(F'/M_i')_i} \le  \rk_{(F'/K_i')_i}\ \textrm{ \ as rank functions on\  }\CC[H].\end{equation}
Observe that $f_{\omega}$ can be extended to an injective $*$-homomorphism  $f_{\omega}:\U_N((t;\tau))\rightarrow \mathcal P_{\omega,\tau}^{\U_N}$ as in (\ref{eqLaurent}), and therefore, setting $\mathcal P = \mathcal R_{\CC[N]}((t;\tau))\times \mathcal P_{\omega,\tau}^{\U_N}$, we have an embedding
 $$\mathcal A ((t;\tau)) \xrightarrow{(i,f_{\omega})} \mathcal P.$$ 
By Lemma \ref{divreg}, regularity of $\mathcal A$ implies that $\D_{N,\mathcal P} = \D_{N,\s_N}$, and regularity of $\mathcal B = \mathcal R_{\CC[H]} \times \U'$ implies that $\D_{H,\mathcal P} =  \D_{H,\mathcal B}$. 
Proposition \ref{key2}, applied to the triple of ranks appearing in (\ref{rk>})  (whose $*$-regular envelopes are, respectively, $\mathcal R_{\CC[H]}$, $\U_H$ and $\U'$), tells 
us that $a$ is invertible if and only if the non-zero element $a' = \Phi_{H,\mathcal B}(\alpha) \in \D_{H,\mathcal B} =\D_{H,\mathcal P} $ is invertible, where $\Phi_{H,\mathcal B}: \Rat(\CC^{\times}H)\rightarrow \D_{H,\mathcal B}$ is the universal morphism. 

Observe that, by definition, $\Tree_{H}(a')\le \Tree(\alpha)$, and therefore any $0\ne b\in \D_{H,\mathcal P}$ with $\Tree_H(b)<\Tree_H(a')$ is invertible by the induction hypothesis (now the locally indicable group is $H\cong F^\prime/M^\prime$ and we approximate $M^\prime$ by $(F'/K_i')_i$). Now we are in position to use Proposition \ref{key}, that assures that $a' \in \D_{N,\mathcal P}((t;\tau))$. Moreover, $a' = \sum a'_k$ with at least two non-zero summands, because otherwise $\alpha\in \Rat(\CC^{\times}N)t^n$ and therefore $H\le N$, a contradiction. The same proposition then states that $\Tree_H(a_k')<\Tree_H(a')$, for all $k$, and consequently every non-zero $a_k'$ is invertible in $\D_{H,\mathcal P}$. Let $n$ be the smallest $k$ such that $a'_k$ is non-zero. Then $a'_nt^{-n} \in \D_{N,\mathcal P}$ is invertible in $\D_{H,\mathcal P}$, and hence in $\D_{N,\mathcal P}$. Thus, $a' = \sum a'_k \in \D_{H,\mathcal P}$ is invertible in $\D_{N,\mathcal P}((t;\tau)) = \D_{N,\mathcal A}((t;\tau))\subseteq \mathcal P$, and hence in $\D_{H,\mathcal P}$. As we already mentioned, this implies the invertibility of $a$, and the result follows.
\end{proof}
With the above result, the L\"uck approximation theorem follows.
\begin{proof}[Proof of Theorem \ref{Lueck}] First assume that $G = F/M$ is locally indicable.
Borrowing the notation of Proposition \ref{maintheoremluck}, we know that $\D_{G,\s}$, the division closure of $\varphi(\CC[G])$ in $\s = \mathcal R_{\CC[G]} \times \U$, is a division ring, and this implies that the projections from $\D_{G,\s}$ onto each factor are $\CC[G]$-isomorphisms. Therefore $\mathcal R_{\CC[G]}$ and $\U$ are isomorphic division $\CC[G]$-rings, and hence $\psi^{\sharp}(\rk) = \rk_G$, from where $\rk_{(F/M_i)_i} =  \rk_{F/M}$, as we wanted to show.

Now assume that $G=F/M$ is virtually locally indicable and ICC. Then there exists a normal subgroup $ F^\prime $ of $F$ of finite index such that  $M\le F^\prime$ and $G^\prime=F^\prime /M$ is locally indicable. If $(M_i)_i$ converges to $M$ in $\MG(F)$, then $(M_i \cap F')_i$ converges to $M$ in $\MG(F^\prime)$. Hence, if $(\U,\psi,\rk)$ is the $*$-regular envelope of $\rk_{(F/M_i)_i}$ as a rank function on $\CC[G]$ and $\U_{G'}$ denotes the $*$-regular closure of $\psi(\CC[G'])$ in $\U$, we obtain that $\mathcal R_{\CC[G^\prime]}$ and $\mathcal U_{G^\prime}$ are isomorphic $\CC[G^\prime]$-rings. Denote this $*$-isomorphism by $\phi$.  We are going to show that $\phi$ extends to a $*$-isomorphism of $\CC[G]$-rings between $\mathcal R_{\CC[G]}$ and $\U$.

Denote by  $\mathcal R_{\CC[G']}G$  the $\CC$-vector subspace of  $\mathcal R_{\CC[G]}$ generated by  $\mathcal R_{\CC[G']}$ and $G$. 
It is easy to check that every element in $G$ normalizes $\mathcal R_{\CC[G']}$. Thus, in fact,   $\mathcal R_{\CC[G']}G$ is a $*$-subring of $\mathcal R_{\CC[G]}$ containing $\CC[G]$, and using \cite[Lemma 2.1]{Li93}, we see that $\mathcal R_{\CC[G']}G = \mathcal R_{\CC[G']}*G/G'$.   Moreover, $\mathcal R_{\CC[G']}$ is a division subring of $\mathcal R_{\CC[G']}G$, so $\mathcal R_{\CC[G']}G$ is artinian. Since a proper $*$-ring cannot have non-trivial nilpotent ideals,  
$\mathcal R_{\CC[G']}G$ is also semisimple, and hence regular. Thus, $\mathcal R_{\CC[G]} = \mathcal R_{\CC[G']}G$. 

Since $\mathcal R_{\CC[G]}$ is semisimple, $\mathcal R_{\CC[G]}$ coincides with its $\rk_{G}$-completion, and since $G$ is ICC, it follows from \cite[Propositions 5.7 and 5.8]{Ja17base} that $\mathcal R_{\CC[G]}$ is simple and that $\rk_G$ is the only Sylvester matrix rank function on $\mathcal R_{\CC[G]}$.

  Let $\{t_1,\dots, t_n\}$ be a transversal of $G'$ in $G$   and observe that we have $\mathcal R_{\CC[G]} = \bigoplus_{i}  t_i\mathcal R_{\CC[G']}$. 
We extend $\phi$ to the map $\phi :\mathcal  R_{\CC[G]}\to \mathcal  U$ by sending $\sum_it_ia_i$ ($a_i\in \mathcal R_{\CC[G^\prime]}$) to $\phi(\sum_it_ia_i):=\sum_i\psi(t_i)\phi(a_i)$. Since the above sum is direct, $\phi$ is well-defined, and it is a $*$-homomorphism of $\CC[G]$-rings because for any $a\in \mathcal R_{\CC[G']}$, one can show  
that 
$$
  \phi(t_iat_i^{-1}) = \psi(t_i)\phi(a)\psi(t_i^{-1})
$$
Since $\mathcal R_{\CC[G]}$ is simple, $\phi$ is injective, and taking into account that $\phi$ is epic and $\mathcal R_{\CC[G]}$ is regular, $\phi$ is surjective (see \cite[Proposition XI.1.2]{Ste}). Therefore, $\phi$ is bijective. By uniqueness of the rank function on $\mathcal R_{\CC[G]}$, we obtain  $\rk_G=\phi^\sharp(\rk$) as rank functions on $\mathcal R_{\CC[G]}$ and, hence, since $\phi$ is a $\CC[G]$-isomorphism, we have $\rk_G = \psi^{\sharp}(\rk) = \rk_{(F/M_i)_i}$ as rank functions on $\CC[G]$.

Finally, we assume only that $G=F/M$ is virtually locally indicable.  Consider the free product  $F^\prime=F*\Z$. Let  $M^\prime$ be the normal subgroup of $F^\prime$ generated by $M$ and for each $i$, let $M^\prime_i$ be  the normal subgroup of $F^\prime$ generated by $M_i$. 

Then $G^\prime=F^\prime/M^\prime\cong G*\Z$ is virtually locally indicable and ICC. Indeed, if $N$ is a finite index locally indicable normal subgroup of $G$, then the normal subgroup $N^\prime$ of $G^\prime$ generated by $N$ and the free factor $\Z$,   which is precisely the kernel of the natural map $G*\Z \rightarrow G/N$ sending $\Z$ to the trivial element,  is of finite index and it is isomorphic to the free product of $N$ and $|G:N|$ copies of $\Z$.

Observe also that $(M_i^\prime)_i$ converges to $M^\prime$,  and so $\rk_{G^\prime}=\rk_{(F'/M_i')_i}$. Notice that $\rk_G$ is the restriction of $\rk_{G^\prime}$ to $\CC[G]$ and $\rk_{(F/M_i)_i}$ is the restriction of $\rk_{(F'/M_i')_i}$. Hence we are done.

\end{proof} 

\section{On the universality of Hughes-free Sylvester matrix rank functions}\label{universal}

Considering different division $R$-algebras for a given ring $R$,
it is  natural to ask whether there exists   the largest possible in some sense. P. Cohn made this notion precise by introducing
the notion of universal epic division ring of a ring \cite[Section 7.2]{CohFI}. In our language an epic division $R$-ring $\D$ is {\bf universal} if for every division $R$-ring $\E$ we have that
$$ \rk_\E \le  \rk_\D \textrm{\ as Sylvester matrix rank functions on $R$} .$$
  If, additionally, $R$ embeds in $\D$, then $\D$ is its \textbf{universal division ring of fractions}.  Given a family $\mathcal F\subseteq \mathbb P(R)$ of Sylvester matrix rank functions on $R$, we say that $\rk \in \mathcal F$ is {\bf universal in $\mathcal F$} if 
$$\rk^\prime \le \rk \textrm{\ for every\ }\rk^\prime\in \mathcal F.$$
For example, if $F$ is a free group, then $\rk_F$ is universal in $\mathbb P(\CC[F])$,  because it coincides with the inner rank (see, for example, the proof of \cite[Proposition 5.3]{Ja17HN}). The proof of Theorem \ref{mainluck} suggests that some of the following questions might have a positive answer.
\begin{Question} Let $G$ be a locally indicable group. Is it true that $\rk_G$ is universal in  $\mathbb P(\Q[G])$, $\mathbb P_{reg} (\Q[G]), \mathbb P_{*reg}(\Q[G])$ or  $\mathbb P_{div}(\Q[G])$? 
\end{Question}
It is shown in \cite{Ja19rk} that if $\rk_{\{1\}}\le \rk_G$ as  Sylvester matrix rank functions on $\Q[G]$, then $G$ is locally indicable.  $\{1\}$ denotes here the trivial group and, therefore, $\rk_{\{1\}}$, as a rank on $\Q[G]$, is obtained from the map $\Q[G]\rightarrow \Q$ that sends any $g\in G$ to $1$. Because of this, we  consider  the previous questions for locally indicable groups only.

In this section we consider again the more general situation of crossed products  and we prove the following theorem.

\begin{teo}\label{maximal}
Let $E*G$ be a crossed product of a division ring $E$ and a locally indicable group $G$, and assume that there exists a Hughes-free epic division $E*G$-ring $ \D$. Then the  Sylvester matrix  rank function $ \rk_\D $ is maximal in  $\mathbb P_{div}(E*G)$. \end{teo}
 As an immediate application of the previous theorem and Corollary \ref{cor5.4} we obtain the following result.
 \begin{cor}
Let $E*G$ be a crossed product of a division ring $E$ and a locally indicable group $G$. If there exist a Hughes-free division ring $\D$ and a universal    epic  division ring $\E$ for $E*G$, then they are isomorphic as $E*G$-rings. In particular, if there exists a universal epic   division $K[G]$-ring for a subfield $K$ of $\CC$, then it is isomorphic to the division closure of $K[G]$ in $\mathcal U(G)$.
\end{cor}
The rest of the section will be dedicated to the proof of Theorem \ref{maximal}, which is very similar to the proof of Theorem \ref{mainluck}. We will need the following auxiliary result.
 
\begin{lem} Let $R$ be a ring and $\tau$ an automorphism of $R$. Set $S=R[t^{\pm 1};\tau]$ and let $\rk$ be a $\tau$-compatible integer-valued Sylvester matrix rank function on $R$. Then the natural extension $\wrk$ of $\rk$ is universal among the Sylvester matrix rank functions on $S$ that extend $\rk$.
\end{lem}

\begin{proof} Let $(\D,\phi)$ be the division envelope of $\rk$. By Proposition \ref{prop7.4iv}   $\tau$ and $\phi$ extend, respectively, to an automorphism $\tau$ of $\D$ and to a homomorphism $\phi: S\to \D[t^{\pm 1};\tau]$, and $\wrk=\phi^\sharp(\widetilde{\rk_\D})$.

 Let $\rk'$ be a Sylvester matrix rank function on $S$ whose restriction to $R$ coincides with $\rk$. We want to show that $\rk^\prime \in \im \phi^\sharp$. Let $\Sigma$ be the set of matrices over $R$   of maximum $\rk'$-rank 
 (and so of maximum $\rk$-rank), and denote by $R_\Sigma$ and $S_\Sigma$ the localizations of $R$ and $S$ at $\Sigma$, respectively. Observe that $\tau(\Sigma) = \Sigma$, and so $\tau$ extends to an automorphism of $R_\Sigma$ and $S_\Sigma\cong R_\Sigma[t^{\pm 1};\tau]$. Now, \cite[Theorem 7.5]{Scho} allows us to extend $\rk'$ to rank functions on $R_\Sigma$ and $S_\Sigma$, that we also denote by $\rk'$. Moreover, since $\rk'$ is integer-valued on $R$, \cite[Theorem 7.6]{Scho} also tells us that $M=\ker \rk'$ is the maximal ideal of $R_\Sigma$ and that $R_\Sigma/M$ is isomorphic to $\D$ as an $R$-ring.   Therefore, $S_\Sigma/S_\Sigma M$ is isomorphic to  $\D[t^{\pm 1};\tau]$ as an $S$-ring, and inasmuch as $\rk^\prime$ can be viewed as a rank function on  $S_\Sigma/S_\Sigma M$, $\rk^\prime \in \im \phi^\sharp$.
  
 Now, $\widetilde{\rk_\D}$ is universal in $\mathbb P(\D[t^{\pm 1};\tau])$, because the $\widetilde{\rk_\D}$-rank of any non-zero element is $1$, and any matrix over $\D[t^{\pm 1};\tau]$ can be written as a product $PDQ$, where $P$ and $Q$ are invertible over $\D[t^{\pm 1};\tau]$ and $D$ is diagonal. 
Thus, $\rk^\prime\le \wrk$. 
\end{proof} 
The following proposition is an analog of Proposition \ref{maintheoremluck}.

\begin{pro}  
Let $E*G$ be a crossed product of a division ring $E$ and  a locally indicable group $G$, and assume that there exists a Hughes-free epic division $E*G$-ring $\D$. Let $\rk$ be an integer-valued rank function on $E*G$ such that $\rk\ge \rk_{\D}$, and with epic division envelope $(\E,\phi)$. If $\varphi=(i,\phi): E*G \rightarrow \mathcal \D \times \E$ denotes the induced map, then the division closure of $\varphi(E*G)$ in $\D\times \E$ is a division ring.
\end{pro}

\begin{proof}
 Set $\s = \D \times \E$, let $\mathcal L = \D_{G,\s}$ be the division closure of $\varphi(E*G)$ in $\s$ and denote, for any subgroup $H\le G$, $\s_H = \D_H \times \E_H$, where $\D_H$ and $\E_H$ are the division closures of the image of $E*H$ in $\D$ and $\E$, respectively. Consider the universal morphism of rational $E^{\times}G$-semirings $\Phi_{G,\s}: \Rat(E^{\times}G)\rightarrow \mathcal L$. Considering $E$ fixed, we are going to show that if $\alpha\in \Rat(E^{\times}G)$ realizes the $G$-complexity of a non-zero $a\in \D_{G,\s}$ for some locally indicable group $G$ and for some crossed product $E*G$ for which there exists a Hughes-free epic division $E*G$-ring $\D$, then $a$ is invertible in  $\D_{G,\s}$. Since $\Phi_{G,\s}$ is surjective for every $G$ and for every crossed product $E*G$, this gives the result.
 
 For any locally indicable group $G$, and for any crossed product $E*G$ in the previous setting, if $\alpha\in \Rat(E^{\times}G)$ satisfies $\Tree(\alpha) = 1_{\mathcal T}$, then $\alpha \in E^{\times}G$ realizes the $G$-complexity of $a = \Phi_{G,\s}(\alpha) \in \varphi(E^{\times}G)$, and $a$ is invertible. Now suppose that $\Tree(\alpha)>1_{\mathcal T}$ realizes the $G$-complexity of a non-zero $a\in \D_{G,\s}$ for some locally indicable $G$ and for some crossed product $E*G$ for which there exists a Hughes-free epic division $E*G$-ring, and assume by induction that we have proved the result for all locally indicable $G'$, for all crossed products $E*G'$ for which there exists a Hughes-free epic division $E*G'$-ring, and for all $\beta\in \Rat(E^{\times}G')$ with $\Tree(\beta)<\Tree(\alpha)$. We want to show that $a$ is invertible in $\D_{G,\s}$. 
 
We can assume, without loss of generality, that $\alpha$ is primitive. Let $H$ be the (finitely generated) image of $\source(\alpha)$ under $E^{\times}G \rightarrow E^{\times}G/E^{\times} = G$. Then, $\alpha \in \Rat(E^{\times}H) $ and therefore $a\in \D_{H,\mathcal L} = \D_{H,\s} = \D_{H,\s_H}$. If $H$ is trivial, the result is clear. Otherwise, there exists a normal subgroup $N\trianglelefteq H$ such that $H = N\rtimes C$ where $C$ is infinite cyclic. Let $t$ be a preimage in $E^{\times}H$ of a generator of $C$, and let $\tau$ denote the automorphism of $E*N$ induced by left conjugation by $t$. \color{black} Since $\tau$ can be extended to an automorphism of both $\D_N$ and $\E_N$, it extends to an automorphism of the regular ring $\mathcal A = \s_N = \D_N\times \E_N$. Therefore, we have a homomorphism 
 $$
   \varphi: E*H \cong (E*N)[t^{\pm 1};\tau] \rightarrow \mathcal A((t;\tau)) = \PP
 $$
In addition, since $\D$ is the Hughes-free epic division $E*G$-ring, we have that $\D_H \cong \D_N(t;\tau)$  as $E*H$-rings, where $\D_N(t;\tau)$   denotes the Ore division ring of fractions of $\D_N[t^{\pm 1};\tau]$, and so, by an abuse of notation, we can write $\varphi(E*H) \subseteq \D_H\times \E_N(t;\tau)$. Set $\rk' = \rk_{\E_{N}(t;\tau)}$. Since $\phi^{\sharp}(\rk')$ is the natural extension of $\rk|_{E*N}$ by Proposition \ref{prop7.4iv}, we deduce in view of the previous lemma that, as rank functions over $E*H$,
\begin{equation*} 
 \begin{gathered}
   \phi^{\sharp}(\rk')\ge \rk_{|E*H} \ge \rk_{\D_H},
 \end{gathered} 
\end{equation*}
  where the second inequality follows from the hypothesis $\rk \ge \rk_{\D}$.  
By Lemma \ref{divreg}, regularity of $\mathcal A$ implies that $\D_{N,\mathcal P} = \D_{N,\s_N}$, and regularity of $\mathcal B = \D_H \times \E_N(t;\tau)$ implies that $\D_{H,\mathcal P} =  \D_{H,\mathcal B}$. Proposition \ref{key2} applied to the above triple of ranks tells us that $a$ is invertible if and only if the non-zero element $a' = \Phi_{H,\mathcal B}(\alpha) \in \D_{H,\mathcal B} =\D_{H,\mathcal P} $ is invertible, where $\Phi_{H,\mathcal B}: \Rat(E^{\times}H)\rightarrow \D_{H,\mathcal B}$ is the universal morphism. By definition, $\Tree_{H}(a')\le \Tree(\alpha)$, and therefore applying the induction hypothesis to $H$ (which has Hughes-free epic division $E*H$-ring $\D_H$) and $\phi^{\sharp}(\rk')$, we have that any $0\ne b\in \D_{H,\mathcal P}$ with $\Tree_H(b)<\Tree_H(a')$ is invertible. Using Proposition \ref{key}, we obtain that $a' \in \D_{N,\mathcal P}((t;\tau))$. Moreover, $a' = \sum a'_k$ with at least two non-zero summands, because otherwise $\alpha\in \Rat(E^{\times}N) t^n\color{black}$ and therefore $H\le N$, a contradiction. The same proposition then states that $\Tree_H(a_k')<\Tree_H(a')$, for all $k$, and so every non-zero $a_k'$ is invertible in $\D_{H,\mathcal P}$. Let $n$ be the smallest $k$ such that $a'_k$ is non-zero. Then $a'_nt^{-n} \in \D_{N,\mathcal P}$ is invertible in $\D_{H,\mathcal P}$, and hence in $\D_{N,\mathcal P}$. Thus, $a' = \sum a'_k \in \D_{H,\mathcal P}$ is invertible in $\D_{N,\mathcal P}((t;\tau)) = \D_{N,\mathcal A}((t;\tau))\subseteq \mathcal P$, and hence in $\D_{H,\mathcal P}$. Therefore, $a$ is invertible and the result holds.
\end{proof}

Now we are ready to prove Theorem \ref{maximal}.
\begin{proof}[Proof of Theorem \ref{maximal}] 
  Assume that there exists a division $E*G$-ring $\E$ such that $\rk_{\E}\ge \rk_{\D}$ as rank functions on $E*G$. Applying the previous proposition and borrowing the notation, we obtain that the division closure $\D_{G,S}$ of $\varphi(E*G)$ in $S = \D\times \E$ is a division ring, and this implies that the projections from $\D_{G,S}$ onto each factor are $E*G$-isomorphisms. Therefore, $\D$ and $\E$ are isomorphic $E*G$-rings, and in particular $\rk_\D = \rk_\E$. 
\end{proof}

\end{document}